\documentclass[review,,11pt]{article}
\usepackage[utf8]{inputenc}     
\usepackage[english]{babel}     
\usepackage{placeins} 
\usepackage[a4paper, left=1in, right=1in, top=1in, bottom=1in]{geometry}
\usepackage{amsmath,amssymb}
\usepackage{amsthm}   
\usepackage{amsrefs}
\usepackage{authblk}
\usepackage{graphicx}
\usepackage{enumitem}
\usepackage{float} 
\usepackage{color}
\usepackage{xcolor}
\usepackage{subcaption}
\usepackage{adjustbox}
\usepackage{cancel}
\usepackage{ulem}
\usepackage{mathtools}
\usepackage{comment}
\usepackage{hyperref}   
\usepackage{cleveref}   
\usepackage{bm}
\usepackage{booktabs} 
\usepackage{titlesec}
\usepackage{algorithm}
\usepackage{algorithmic}

\titleclass{\subsubsubsection}{straight}[\subsection]
\newcounter{subsubsubsection}
\titleformat{\subsubsubsection}{\normalfont\normalsize\bfseries}{\thesubsubsubsection}{1em}{}
\titlespacing*{\subsubsubsection}{0pt}{3.25ex plus 1ex minus .2ex}{1.5ex plus .2ex}
\renewcommand{\thesubsubsubsection}{\thesubsection.\arabic{subsubsection}.\arabic{subsubsubsection}}

\newcounter{AssDD}
\stepcounter{AssDD}

\newcounter{AssM}
\stepcounter{AssM}
\numberwithin{equation}{section}
\begin{document}
\title{Robust, fast, and adaptive splitting schemes for nonlinear doubly-degenerate diffusion equations}
\author[1]{A. Javed}
\author[2]{K. Mitra}
\author[1]{I.S. Pop}
\affil[1]{Hasselt University, Belgium}
\affil[2]{Eindhoven University of Technology, The Netherlands}

\date{\today}

\def \a  {\alpha}
\def \oa {\overline{\alpha}}
\def \b  {\beta}
\def \g  {\gamma}
\def \G  {\Gamma}
\def \d  {\delta}
\def \D  {\Delta}
\def \e  {\varepsilon}
\def \eps {\epsilon}
\def \f  {\varphi}
\def \j {\ell}
\def \vr  {\varrho}
\def \vs  {\omega}
\def \k  {\kappa}
\def \l  {\lambda}
\def \om {\omega}
\def \Om {\Omega}
\def \r  {\rho}
\def \s  {\sigma}
\def \t  {\tau}
\def \th {\theta}
\def \z  {\zeta}
\def \del {\nabla}
\def \p  {\partial}
\def \N  {{\mathbb{N}}}
\def \R  {{\mathbb{R}}}
\def \dd {\mathrm{d}}
\def \m {\mathrm{m}}
\def \x {{\bm{x}}}

\def \calA {\mathfrak{A}}
\def \calB {\mathfrak{B}}
\def \calE {\mathcal{E}}
\def \calF {\mathfrak{F}}
\def \calZ {\mathcal{Z}}
\def \calV {\mathcal{V}}
\def \calR {\mathfrak{R}}
\def \calT {\mathcal{T}}
\def \fli {\bm{\varsigma}^{i-1}_n}

\def \Maxi {{\mathrm{M}}}
\def \Mini {{\mathrm{m}}}

\newcommand{\linErr}[1]{\calE_{\!_{\,{\rm lin},n}}^{#1}}
\newcommand{\linEst}[1]{\eta_{\!_{\,{\rm lin},n}}^{#1}}
\newcommand{\normErr}[1]{\tilde{\calE}_{\!_{\,{\rm lin},n}}^{\,#1}}
\newcommand{\normEst}[1]{{\eta}_{\!_{\,{\rm lin},n}}^{\,#1}}
\newcommand{\lp}[3]{\big(\!\big( #2,\,#3\big)\!\big)_{#1}}
\newcommand{\norm}[1]{{\left\vert\kern-0.25ex\left\vert\kern-0.25ex\left\vert #1
    \right\vert\kern-0.25ex\right\vert\kern-0.25ex\right\vert}}
    \newcommand{\efl}[1]{\bm{\sigma}_{n,\calT}^{#1}}
\newcommand{\snorm}[1]{{\left[\kern-0.25ex\left[ #1
    \right]\kern-0.25ex\right]}}

\newtheorem{lemma}{Lemma}[section]
\newtheorem{theorem}[lemma]{Theorem}
\newtheorem{corollary}[lemma]{Corollary}
\newtheorem{proposition}[lemma]{Proposition}
\newtheorem{assumption}[lemma]{Assumption}
\newtheorem{definition}[lemma]{Definition}
\newtheorem{remark}[lemma]{Remark}
\newtheorem{example}[lemma]{Example}
\newtheorem{conjecture}[lemma]{Conjecture}
\newtheorem{problem}{Problem}

\newcommand\KM[1]{%
  \protect\leavevmode
  \begingroup
        \color{red!55!yellow}%
 #1%
  \endgroup
}

\newcommand\AJ[1]{%
  \protect\leavevmode
  \begingroup
        \color{red!30!blue}%
 #1%
  \endgroup
}

\newcommand\SP[1]{%
  \protect\leavevmode
  \begingroup
        \color{red!30!green}%
 #1%
  \endgroup
}

\maketitle

\begin{abstract}
We consider linear iterative schemes for the time-discrete equations stemming from a class of nonlinear, doubly-degenerate parabolic equations. More precisely, the diffusion is nonlinear and may vanish or become multivalued for certain values of the unknown, so the parabolic equation becomes hyperbolic or elliptic, respectively. After performing an Euler implicit time-stepping, a splitting strategy is applied to the time-discrete equations. This leads to a formulation that is more suitable for dealing with the degeneracies. Based on this splitting, different iterative linearization strategies are considered, namely the Newton scheme, the L-scheme, and the modified L-scheme. We prove the convergence of the latter two schemes even for the double-degenerate case. In the non-degenerate case, we prove that the scheme is contractive, and the contraction rate is proportional to a non-negative exponent of the time-step size. Moreover, following \cite{stokke2023adaptive}, an a posteriori estimator-based adaptive algorithm is developed to select the optimal parameters for the M-scheme, which accelerates its convergence. Numerical results are presented, showing that the M- and the M-adaptive schemes are more stable than the Newton scheme, as they converge irrespective of the mesh. Moreover, the adaptive M-scheme consistently out-competes not only the M/L-schemes, but also the Newton scheme showing quadratic convergence behavior. 
\end{abstract}

\section{Introduction} \label{introduction}
This paper discusses a linearization approach for the time-discrete equations related to doubly-degenerate, parabolic advection-diffusion equations. With $\Om$ being a bounded domain in $\R^{d}$ having a Lipschitz boundary $\p\Om$ and for some $T > 0$, letting $Q=(0,T] \times  \Om$ we consider the following equation
\begin{align}\label{eq:1}
    &\p_t u+ \del\cdot \bm{F}(u)=\D w + f, \\
    &w\in \Phi(u) , \nonumber
\end{align}
which holds almost everywhere (a.e.) in $Q$. This equation is completed with e.g. homogeneous Dirichlet boundary conditions in $w$ and an initial condition $u_0$ for $u$. 
The function $\Phi:[0,\vs) \to [0,\infty)$ is increasing and locally Lipschitz continuous with $\vs=1$ (for $u$ representing a concentration) or $\vs=\infty$ ($u$ representing some density). As $u \nearrow \vs$, $\Phi$ can become either multivalued or infinite. Two types of degeneracies can arise when either $\Phi'=0$ (the slow diffusion case), or when $\Phi'= \infty$ (the fast diffusion case). In this case, \eqref{eq:1} changes its type from parabolic to hyperbolic, respectively, elliptic.
In equation $\eqref{eq:1}$, $f$ is a source/sink term and $\bm{F}(u)$ is an advective flux. The functions $\Phi$, $\bm{F}$ can also be heterogeneous, which means that they may depend explicitly on $\x\in \Om$, see \Cref{rem:heterogeneity}. The exact properties of the auxiliary functions are discussed in \Cref{section2}.

Equation $\eqref{eq:1}$ is a mathematical model for many real-world applications. A first example in this sense is the porous medium equation (PME, see \cite{vazquez2007porous}) modeling gas flow in a porous medium, where $\Phi(u)=[u]_+^m$ for some $m > 1$ (here $[u]_+ = \max\{0, u\}$), and $\bm{F}=\bm{0}$. In this case, $\vs=\infty$, and the degeneracy appears when $u=0$. Another example is the Richards equation modeling unsaturated flow through a porous medium where $\vs=1$ and $\Phi$ becomes multivalued at $u=\vs$ (the details being given in  \Cref{section2}), while $\Phi^\prime \to 0$ whenever $u \to 0$, see \cite{mitra2024posteriori}.  
In the same sense we mention biofilm growth models \cite{van2002mathematical} (with $\Phi(u)=\int_0^u \frac{\rho^a}{(1-\rho)^c}$ exploding as $u\nearrow \vs=1$), or the Stefan problem and permafrost models \cite{Malgo}.  Figure $\eqref{fig:Richards_Biofilm_phi_function}$ presents the nonlinear function $\Phi$ for the Richards equation (left) and for the biofilm growth model (right), which are representative for the cases when $\Phi$ becomes either multivalued, or singular at $\vs=1$. Also, note that in both cases one has $\Phi^\prime = 0$ for certain values of $u$. 
\begin{figure}[H]
    \centering
    \begin{subfigure}[b]{0.33\textwidth}
        \centering
        \includegraphics[width=\textwidth, height=4cm]{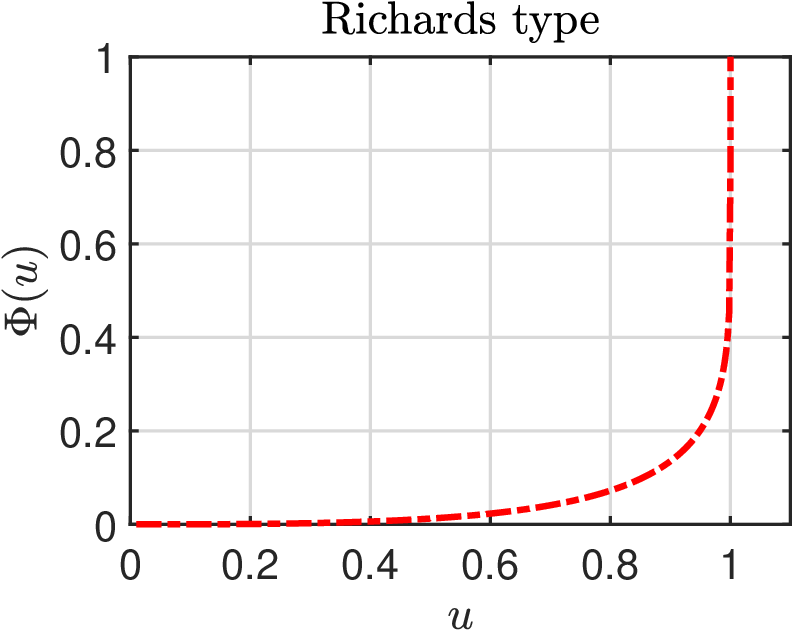} 
    \end{subfigure}
    \begin{subfigure}[b]{0.33\textwidth}
        \centering
        \includegraphics[width=\textwidth, height=4cm]{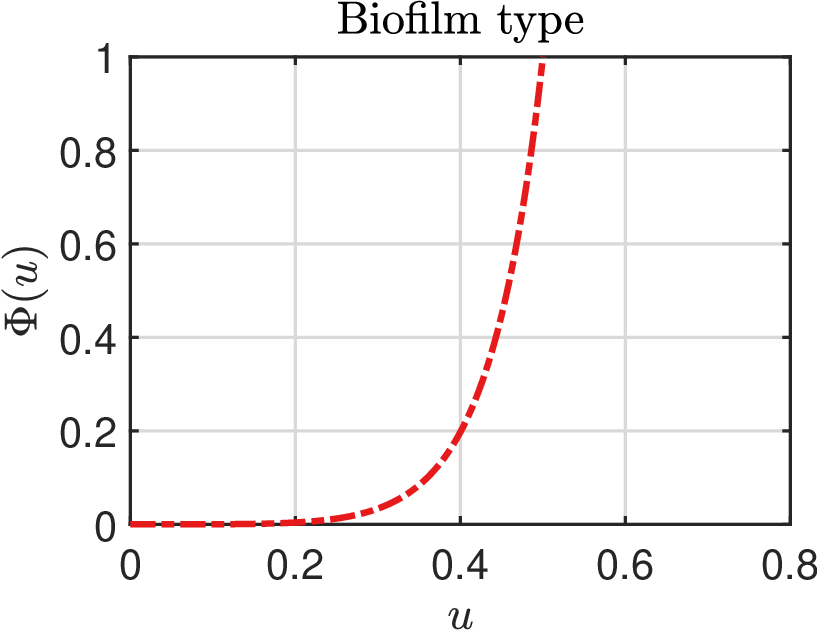} 
    \end{subfigure}
    \caption{Examples of double degenerate $\Phi$: (left) $\Phi$ becomes multivalued at $u=1$ (Richards equation type degeneracy). (right) $\Phi$ becomes infinite at $u=1$ (biofilm type singularity).}
    \label{fig:Richards_Biofilm_phi_function}
\end{figure}
To deal with the double degeneracy discussed above, one can follow the ideas in  \cite{Carrillo2004, brenner2017improving} and reformulate \eqref{eq:1} in terms of a new unknown $s$ and of two increasing functions $b,\, B\in C^1(\R)$ satisfying
\begin{align}\label{eq:condition}
    B =\Phi \circ b, \text{ and }\; 0\leq b',B'\leq 1, \text{ and }  b'+B'\geq 1.
\end{align}
With this choice, whenever $w \in \Phi(u)$, if $u = b(s)$, one immediately gets $w = B(s)$.  In this way, \eqref{eq:1} becomes 
\begin{align}\label{eq:2}
    &\p_t u+\del\cdot \bm{F}(u)=\D w + f,\\
    &u=b(s), \quad w \in B(s). \nonumber
\end{align}
The advantage of this formulation is that the functions $b$, $B$ are differentiable in $\R$. The two degeneracies appear when either $b'\searrow 0$ (originally, the fast diffusion case) or $B'\searrow 0$ (the slow diffusion). 
Such decomposition is always possible and an explicit formula to compute the $b$, $B$ functions is presented in  \Cref{sec:2.2}. For example, this is used to determine the functions $b$ and $B$ corresponding to the function $\Phi$ in the left plot of \Cref{fig:Richards_Biofilm_phi_function}. The graphs of these functions are presented in \Cref{fig:graph1}. 
\begin{figure}[h]
    \centering
 \includegraphics[width=0.35\textwidth]{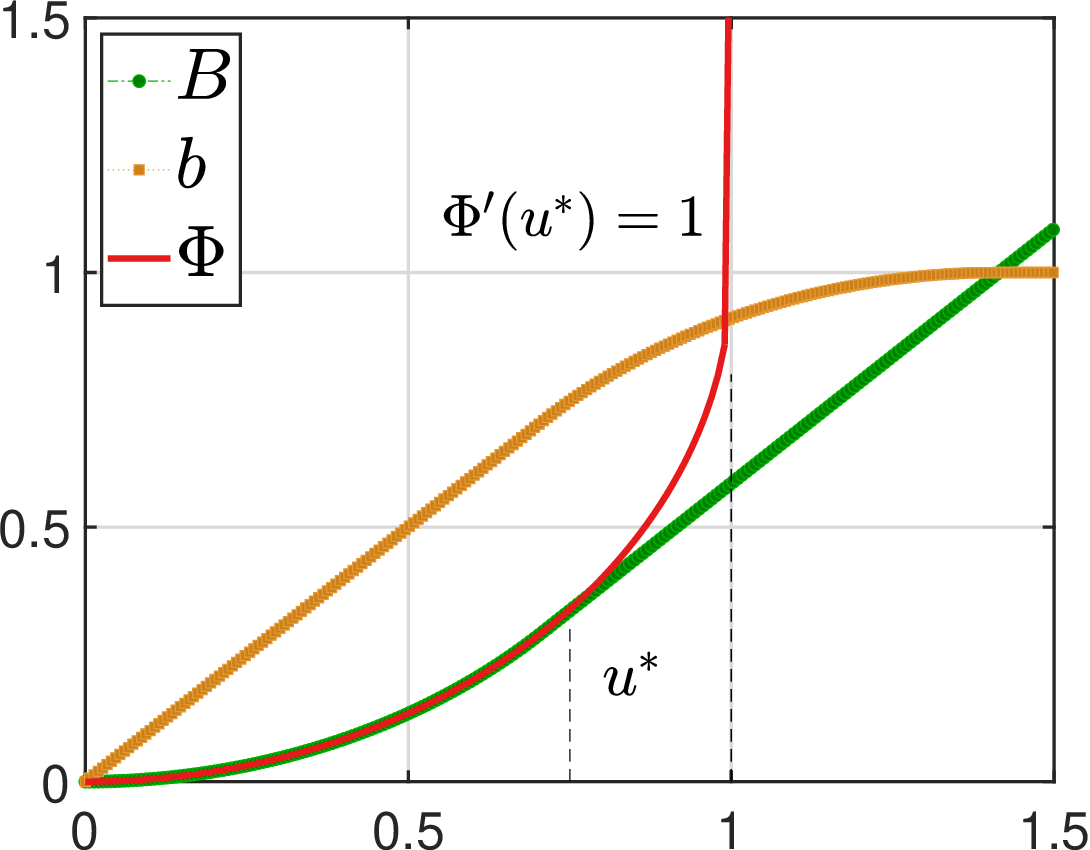}
 \caption{The function $\Phi$ in the left picture of \Cref{fig:Richards_Biofilm_phi_function} and the corresponding functions $b$ and $B$, given by \eqref{eq:bBexpression} and satisfying \eqref{eq:condition}.}
        \label{fig:graph1}
\end{figure}

\subsection{Well-posedness and discretization} \label{sec:wellposedness}
The existence and uniqueness of solutions for doubly-degenerate equations are obtained, e.g., in \cite{alt1983quasilinear, ANDREIANOV20173633,CarilloEntropy,CARRILLO199993,DroniouEymardTalbot,otto1996l1, pop2011regularization,ZouExistence}. For the time-discretization, implicit schemes are quite often used for such problems due to the lack of regularity of the solutions, see \cite{duvnjak2006time,pop2002error,radu2008error}, where error estimates are obtained for implicit time discretization of doubly-degenerate problems. Specifically, to define the Euler implicit  discretization of \eqref{eq:1}, respectively \eqref{eq:2}, we let $N \in \N$ be strictly positive  and consider a (fixed) time-step size $\t = T/N$. With $n\in \{0, \dots, N\}$, we define $t_n = n \t)$ and let $z_n$ approximate the function $z(t_n)$, where $z \in \{s, u, w\}$ is one component of the solution triple $(s, u, w)$. With this, the time discretization of 
\eqref{eq:2} requires solving at each time step $t_n$ ($n > 0$) the system 
\begin {align}\label{eq:t}
  \left\{\begin{array}{ll} \frac{1}{\tau}(u_n-u_{n-1})+ \nabla \cdot \bm{F}(u_n) =  \Delta w_n+f,  \\
 u_n=b(s_n), \text{ and }  w_n =B(s_n), 
\end{array}\right. 
\end{align}
defined in $\Om$. This is complemented by (homogeneous Dirichlet) boundary conditions for $w_n$. For $n = 1$ one uses the initial condition $u_0$. Observe that \eqref{eq:t} is a system involving a linear elliptic partial differential equation and two nonlinear, algebraic ones. Inserting the last two algebraic equations into the first one, it becomes the nonlinear elliptic equation 
\begin{equation}
    \label{eq:t000}
\frac{1}{\tau}(b(s_n)-b(s_{n-1}))+ \nabla \cdot \bm{F}(u_n) =  \Delta B(s_n)+f .
\end{equation}
If $B$ can be inverted, this can be further rewritten in terms of the variable $w_n$, see Section \ref{sec:direct}. 

For the spatial discretization of double degenerate advection-diffusion equations, various methods have been proposed. Here we restrict to porous media flow models, and namely to mathematically rigorous results. The convergence by compactness arguments or by a priori error estimates has been proved for the finite volume method in \cite{Andreianov,Bassetto,baughman1993co,eymard2002convergence,Quenjel,RaduKlausen}, the finite difference method in \cite{karlsen2002upwind}, the mixed or conforming finite element method in \cite{Arbogast,CancesNabet,radu2008error,Woodw}, the discontinuous Galerkin method in \cite{Clint,Dolejsi,Beatrice} and, in general gradient discretization schemes in \cite{cances2021error,droniou2020high,DroniouEymardUniform,DroniouLe,EymardGradient},
to name a few. A posteriori estimates for elliptic problems that are similar to \eqref{eq:t} have been derived in \cite{ahmed2024efficient}, and for the doubly-degenerate  parabolic system in \cite{CancesNabet,CancesPop,MartinRegularization,mitra2024posteriori}. 

Observing that the time-discrete problems \eqref{eq:t} are nonlinear, in what follows we discuss different linear iterative schemes for the numerical approximation of the (time-discrete) solutions. The analysis is done in a continuous-in-space setting, but for the numerical test we shall use a finite volume scheme. 

\subsection{Linear iterative schemes}
In this part we discuss different linear iterative schemes for the numerical approximation of the solutions to \eqref{eq:t}. Before discussing these methods in detail, we mention that, since the analysis below is done at the time-discrete level and not at the fully discrete one, the convergence results do not depend on the discretization method and mesh. As will be seen below, this is advantageous as it provides flexibility w.r.t. the choice of the time step size. 
More precisely, whenever the convergence can be guaranteed at the level of the time-discrete problems, this will extend to any spatial discretization and mesh. 

As mentioned above, the linear iteration schemes discussed here use the reformulation of \eqref{eq:1} in terms of a new variable. In consequence, one works with the time-discrete problems in \eqref{eq:t}. Since this involves two additional (algebraic) equations, the corresponding schemes will be called \emph{double-splitting}. For a better positioning of the present contribution in the existing literature, we start by considering \eqref{eq:1}, and even in a simplified, one-equation form. The linear iterative schemes in this case will be called \emph{direct/no-splitting}. 

\subsubsection{Direct/no-splitting method}\label{sec:direct}
 To discuss existing linear iterative schemes for doubly-degenerate equations, we assume that the inverse $\b = \Phi^{-1}$ makes sense, and rewrite \eqref{eq:1} in terms of $w$, with $u = \b(w)$. This is similar to the Kirchhoff transform used e.g. in \cite{alt1983quasilinear}. The Euler implicit discretization leads to the nonlinear elliptic, possibly degenerate equation  
\begin{equation}
\label{eq:t00}
  \frac{1}{\tau}(\b(w_n)-\b(w_{n-1}))+ \nabla \cdot \bm{F}(\b(u_n)) =  \Delta w_n+f, 
\end{equation}
defined in $\Om$, and completed by homogeneous Dirichlet boundary conditions. 

For fixed $n \in \{1, \dots, N\}$, to define the linear iterative scheme, we let $i \in \N$ stand for the iteration index and $w_n^i$ the $i^{\rm th}$ iteration at the $n^{\rm th}$ time step $t_n$. Choosing $w_n^0:=w_{n-1}$ as the initial guess, the sequence $\left(w_n^i\right)_{i \in \N}$ is the solution to 
\begin{align}\label{eq:L}
    \frac{L_{\b,n}^i}{\tau} \left(w_n^i-w_n^{i-1}\right) -\D w^i_n =-\left[\frac{\beta(w_n^{i-1})-u_{n-1}}{\tau}+ \del\cdot\bm{F}^{i-1}_n\right]+f,
\end{align}
in $\Om$. The factors $L_{\b,n}^i$  
are bounded, and may depend on the previous iterations, which explains the presence of the indices $i$ and $n$. For the schemes considered here, in case of convergence, i.e. if $w^i_n\to \Tilde{w}$ and, consequently,  $\b(w^i_n)\to \b(\Tilde{w})$, then $(w^i_n-w^{i-1}_n)\to 0$, and therefore $\Tilde{w}$ solves \eqref{eq:t00}.  

The choice of $L_{\b,n}^i$ and of $\bm{F}^{i-1}_n$ leads to in different iterative schemes, e.g. Newton, Picard, the L-and M-schemes, or combinations thereof. For the Newton scheme (NS), all nonlinearities are replaced by linear Taylor approximations around the previous iteration. In  \eqref{eq:L}, this gives the choice $L_{\b,n}^i=\b^{\prime}(w_n^{i-1})$, and $\bm{F}^{i-1}_n = \bm{F}(\b(w^{i-1}_n)) + L_{\b,n}^i \nabla \bm{F}(\b(w^{i-1}_n)) \cdot (w^i_n - w^{i-1}_n)$ \AJ{\cite{bergamaschi1999mixed,lehmann1998comparison}}. NS stands out due to its quadratic convergence rate, but this quadratic convergence property is valid only under specific conditions. For example, the iterations converge if the initial guess is close enough to the exact solution. For time-dependent problems, a natural initial guess is, as mentioned, $w_n^0:=w_{n-1}$, or some combination of the solutions of previous time-steps \cite{petrosyants2024speeding}. Having $w_n^0$ sufficiently close to $w_n$ may impose a severe restriction on the time-step size, which can be dependent on the spatial discretization and mesh, or even the spatial dimension \cite{radu2006newton}. 
This negates the advantages of using a time-implicit scheme, which grants stability to the time-discretization even for larger $\t$ values. Moreover, NS may not guarantee convergence for degenerate problems, if either $\b'=0$ or $\infty$ (correspondingly $\Phi'=\infty$ or $0$). 

An alternative to the NS is the modified Picard scheme \cite{celia1990general}, where $L_{\b,n}^i=\b'(w^{i-1}_n)$, but for the advective term one uses the previous iteration, $\bm{F}(\b(w^{i-1}_n))$. In \cite{eymard2000finite} it is shown that this scheme is quite fast despite having linear convergence. However, it also suffers from the same stability issues as the Newton scheme, \cite{radu2006newton}. In the same spirit we mention the schemes in \cite{jager1991solution,jager1995solution}, where the linearization is perturbed so that the derivatives of the nonlinearities appearing in the iterations are bounded away from zero and infinity. This ensures the convergence of the scheme in the doubly-degenerate case, but, only under restrictions for the time step size that are like for NS \cite{radu2006newton}. Further, in \cite{Casulli} a Jordan decomposition of the nonlinearity $\b$ is used to define nested Newton iterations, in which the solution is approximated successively by quadratically convergent sequences of sub- and supersolutions. However, this approach makes use of the monotonicity of the approximation, which is not suited for any spatial discretization. In this category, we mention the trust-region Newton scheme in \cite{WANG2013114}, which is tightly connected to the finite volume discretization, and therefore cannot be extended straightforwardly to general discretization schemes and meshes. 

The L-scheme (LS) is a fixed-point iteration, in which $L^i_{\b,n}$ is a sufficiently large constant, to ensure stability. In terms of \eqref{eq:L}, this leads to the choice $L_{\b,n}^i=L\geq \sup \b'$. As shown in \cite{PopYong} and later in \cite{list2016study,pop2004mixed,Slodicka}, the scheme converges in $H^1$--sense convergence to the time-discrete solution $w_n$, regardless of the initial guess, spatial dimension, discretization, or mesh, and under a mild restriction for the time step size, at least for the fast diffusion case when $\b'=0$. This convergence result is extended in \cite{RaduIMA} for H{\"o}lder-continuous $\b$ (thus, $\b'$ not necessarily bounded), but a regularization step is needed. However, as seen in \cite{list2016study,mitra2019modified,stokke2023adaptive}, LS needs significantly more iterations than NS, or Newton-like schemes. To resolve this, the modified L-scheme (MS) was introduced in \cite{mitra2019modified}, envisioned to combine LS and NS in a way to preserve both stability and speed. In context of \eqref{eq:L1}, this is given by the choice $L_{\b,n}^i=\text{max}({\b'(u_n^{i-1})+M\tau,2M\tau})$, for a constant $M>0$. Observe that for $M=0$, the MS is nothing but the NS, while for large $M$, the changes in $L_{\b,n}^i$ are  small and the MS is close to the LS. As proved in  \cite{mitra2019modified}, for this class of problems the MS was as stable as the LS, while being much faster. In fact, the convergence is linear provided $\b'$ is bounded, and the contraction rate even scales with time-step  $\t$ for non-degenerate problems $\b'>0$. Closely related are the iterative schemes in \cite{Albuja}, using a semi-implicit discretization of the nonlinear diffusion term. There, the $L_{\b,n}^i$ is chosen s.t. it decreases form one time step to another, and the problem is regularized. The convergence is proved in $L_{\b,n}^i$-weighted norms and under the assumption that the diffusion is nondegenerate, but for H{\"o}lder-continuous $\b$.   

It is to mention that the direct formulation in \eqref{eq:t00} is possible whenever $\b = \Phi^{-1}$ exists, and $\b$ is bounded. Otherwise, $L^i_{\b,n}$ for schemes discussed above might become infinite, which either excludes the slow-diffusion case, or requires a regularization step. For small regularization parameters, the factors $L^i_{\b,n}$ become very large, which reduces the efficiency of the iterations. We mention in this respect \cite{MartinRegularization}, where the regularization is done so that the induced error is in balance with the errors that are due to the discretization, linearization, or the algebraic solver.

\subsubsection{Double-splitting approach}
For doubly-degenerate problems, $\Phi'$ in \ref{eq:1} can vanish, or become unbounded, or even become multivalued. Particularly in the latter case, constructing linear iterative schemes, not to speak about obtaining mathematically rigorous convergence results is a challenging task. schemes discussed in Section \ref{sec:direct} are either restricted to the case when $\Phi$ is bijective, or rely on regularization. The approach discussed below is inspired by two works. First, we mention \cite{brenner2017improving}, where the problem is first reformulated in terms of a new unknown, so that the resulting nonlinear functions are Lipschitz. For this, a Newton-type scheme is proposed, and the local quadratic convergence is proved for the fully discrete case, for the Euler implicit - a finite-volume discretization. The second work we refer to is \cite{cances2021error}, where, as in \eqref{eq:1}, the nonlinearity is defined as a new unknown, and the linear iterations are defined at the level of such algebraic dependencies. In this context, for the slow-diffusion case (e.g. the porous medium equation), it was shown in \cite{smeets2024robust} that the MS is more stable than the NS. 

To be precise, we refer to \eqref{eq:t}, and use the functions $b,\,B$ satisfying \eqref{eq:condition}. Similar to \cite{brenner2017improving}, we consider the new unknown $s$, while $u = b(s)$ and $w = B(s)$. From now on, this approach will be called below \emph{double-splitting} (DS). As in Section \ref{sec:direct}, for fixed $n \in \{1, \dots, N\}$ and with $i \in \N$ we let $(s_n^i, u_n^i, w_n^i)$ be the $i^{\rm th}$ iteration triple at time $t_n$. Then, with the initial guess  $\left(s_n^0, u_n^0, w_n^0\right) := \left(s_{n-1}, u_{n-1}, w_{n-1}\right)$ and given  $\left(s_n^{i-1}, u_n^{i-1}, w_n^{i-1}\right)$, $\left(s_n^i, u_n^i, w_n^i\right)$ solves
\begin{subequations}\label{eq:3}
\begin{align}
    &\frac{1}{\tau}\left(u^i_n-u_{n-1}\right)+ \nabla \cdot \bm{F}(u_n^{{i-1}}) =  \Delta w_n^i+ f,\label{eq:3a}\\ 
    & L^i_{b,n}(s^i_n-s^{i-1}_n)= u^i_n-b(s^{i-1}_n), \label{eq:3b}\\ 
    & L^i_{B,n}(s^i_n-s^{i-1}_n)= w^i_n-B(s^{i-1}_n), \label{eq:3c}
\end{align}
\end{subequations}
in $\Om$. As in \eqref{eq:L}, the factors $L^i_{b,n},  L^i_{B,n}$ are computed from the $(i-1)^{th}$ step, and their choice determine the type of the scheme (NS, LS, or MS). The details are presented later in Table \eqref{tab:1}. Clearly, if the scheme converge, the limit triple is a solution to \eqref{eq:t}. 


For the LS and MS, the choice of the parameters $L$ and $M$ is important and can improve the convergence rates significantly. In the present context, the convergence is guaranteed for the LS if $L^i_{b,n}  = L^i_{B,n} = 1$. However, this may be sub-otimal, since in one iteration $s$ may take values for which $b^\prime$ and $B^\prime$ are less. Therefore, in \cite{list2016study,mitra2019modified,seu2018linear}, a a parametric study is carried out beforehand. As for the LS and MS the convergence does not depend on the spatial discretization and mesh, this study can be done on a coarse mesh, which reduces the computational complexity significantly. A method to adaptively select the linearization schemes and parameter values at each iteration step was proposed in \cite{stokke2023adaptive}, based on a posteriori error estimation ideas from \cite{mitra2019modified}. Inspired by this, in \Cref{sec:Adaptive} a similar kind of estimator is proposed, adaptively providing a nearly optimal value of $M$. 

The \textbf{main results} for the DS are as follows. First, \textit{the LS converges unconditionally, even for doubly-degenerate cases, and the convergence is linear if it has at most one degeneracy. The MS behaves in the same manner, provided that a regularity assumption is satisfied. In this case, the linear convergence rate is proportional to the time-step size. This makes MS faster especially for smaller time-step sizes. Numerical results reveal that MS is, indeed, more robust than NS. Moreover, in many cases MS outperforms NS in reaching a pre-determined error threshold, whenever the parameter $M>0$ is well chosen.} 

The aspect of how to choose $M$ is resolved adaptively, based on a posteriori estimators. \textbf{Main results} in this case are:  \textit{given $M>0$, we find the a posteriori estimator $\linEst{i,M}$ that is fully computable from the $(i-1)^{th}$ iteration step. It gives an upper bound for the linearization error $\linErr{i}$ at the $i^{th}$ iteration, namely $\linErr{i}\leq \linEst{i,M}$. Note that this estimator depends on the parameter $M>0$, used in the iteration. Inspired by this, we select the parameter $M$ that minimizes $\linEst{i,M}$, which minimizes the upper bound of the linearization error $ \linErr{i}$ in the $i^{th}$ step.} This approach is presented in \Cref{fig:flow-chart}, and the corresponding scheme will be called MAdap. As will be seen in the following, MAdap consistently outperforms NS in a wide variety of cases. 
\begin{figure}[h]
    \centering
  \includegraphics[width=0.9\textwidth]{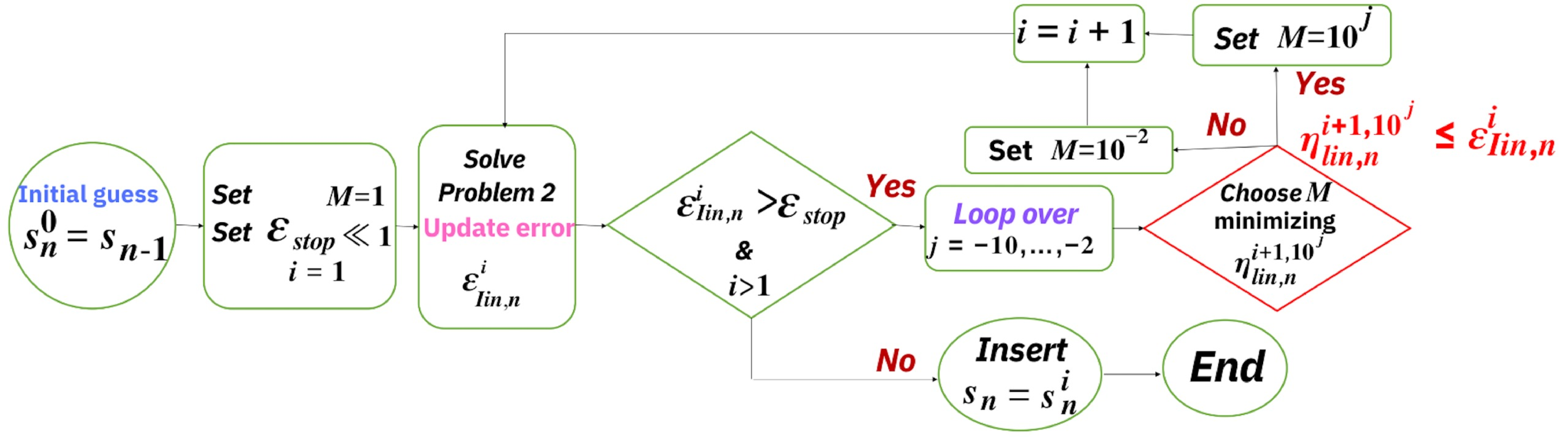}
    \caption{The flow-chart of the MAdap algorithm (\Cref{alg:M-Adap}).}
    \label{fig:flow-chart}
\end{figure}

The outline of the paper is as follows: In \Cref{section2}, we state the main assumptions and notations used throughout the paper, define the time-discrete solution, and discuss different linear iterative methods. \Cref{section3} is devoted to the mathematically rigorous convergence analysis of the schemes. In \Cref{sec:Adaptive}, a posteriori estimates are obtained for the linearization error. Based on this, the adaptive algorithm MAdap is proposed. \Cref{section5} presents numerical results for four test problems and three different linear iterative schemes, which clearly illustrate the robustness of the approach proposed here, as well as the effectiveness of the derived estimates. Our findings are summarized and discussed in \Cref{sec:conclusion}.

\section{Mathematical preliminaries}\label{section2}
\subsection{Notations and basic definitions}\label{sec:notations}
In what follows $\Om\subset \R^d$ denotes a bounded, Lipschitz $d$-dimensional domain ($d\in\N, d > 0$). 

\textbf{Functional spaces and norms}: $L^2(\Om)$ is the space of square-integrable functions defined on $\Om$, the corresponding inner product and norm being $\langle \cdot, \cdot \rangle
$ and $\|\cdot\|$.  $H^1(\Om)$ stands for the $L^2$ functions having weak derivatives in $L^2(\Om)$. $H^1_0(\Om)$ contains the functions in $H^1(\Om)$ having a vanishing trace on $\partial \Om$, and $H^{-1}(\Om)$ is the dual of $H^1_0(\Om)$ with the dual norm
\begin{align}
\|u\|_{H^{-1}(\Omega)} := \sup_{\phi\in H^1_0(\Omega)} \frac{\langle u, \phi \rangle }{\|\nabla\phi\|}
.
\end{align}
The analysis below will use the compositional Hilbert space 
\begin{equation} \label{eq:Z}
    \mathcal{Z}:=L^2(\Om)\times L^2(\Om)\times H^1_0(\Om).
\end{equation} 
Moreover, $\bm{H}(\rm{div};\Om)$  denotes the space of vector fields $\bm{\sigma}\in \bm{L}^2(\Om;\R^d)$ such that $\del\cdot \bm{\sigma}$ exists and lies in $L^2(\Om)$. In general, the norm and the duality pairing in a Banach space $\calV$ are $\|\cdot\|_{\calV}$, respectively$\langle \cdot, \cdot\rangle_{\calV^*\times \calV}$. For a Lipschitz continuous function $u$, $\snorm{u}_{\rm Lip}$ denotes its Lipschitz constant, if not specified differently.
\\[1em]
\textbf{Relevant (in)equalities:} The Poincar\'e inequality states the existence of a $C_\Om>0$ such that 
\begin{align}\label{eq:P}
   \|w\|
   \leq C_{\Om}\, h_\Om\,\| \del w\|
   , 
\end{align}
for all $w\in H^1_0(\Om)$, where $h_\Om>0$ is the diameter of $\Om$.  

The following algebraic (in)equalities, holding for all  $a,b \in \R$ and $\r>0$, will be used
\begin{align}
  &(a-b)a=\frac{1}{2}(a^2-b^2+(a-b)^2) \label{eq:equality} \\
  &ab \leq \frac{1}{2\r}a^2+ \frac{\r}{2}  b^2 \;\;\text{  (Young's inequality). }\label{young}
\end{align}
 
Finally, for any given $a,b\in \R$, we use $I(a,b):=[a,b]\cup [b,a]$ to denote the closed interval between them. Note that by this one avoids distinguishing between cases $a < b$ and $a > b$. 

\subsection{Assumptions} \label{sec:2.2}
 The functions appearing in \eqref{eq:1} satisfy the following assumptions.
\begin{enumerate}[label=(A.1\alph*)]
\item \label{ass:A1a}
Let $\vs=1$, or $\vs=\infty$. The 
function $\Phi:[0,\vs) \to [0,\infty)$ 
is locally Lipschitz, almost everywhere differentiable, and strictly increasing in $(0,\vs)$. 
Moreover, $\Phi(0)=0$ and either $\Phi'(0)>0$, or $\Phi$ is convex in a right neighbourhood of 0. Further, the limit $\Phi_\Maxi:=\lim_{u\nearrow \vs}\Phi$ is either infinite, or, if $\Phi_\Maxi<\infty$, then we extend $\Phi$ to the set $[\Phi_\Maxi,\infty)$ at $u=\vs$. 
\item \label{ass:A1b}
The functions $b, B\in C^1(\R)$ exist such that $b(0)=B(0)=0$ and they satisfy \eqref{eq:condition}.
\end{enumerate}
\noindent Assumption \ref{ass:A1a} is sufficient to prove the convergence of LS. For MS, the additional Assumption \ref{ass:Bphi} on the regularity of $\Phi$ is made below. In this situation, the functions $b$ and $B$ can be constructed explicitly, see 
\Cref{lemma:decomspoition} for the details.

\begin{enumerate}[label=(A.\theAssDD)]
\stepcounter{AssDD}
\item \label{A.2} The source term $f$ lies in $H^{-1}(\Om)$. For the advection term $\bm{F}: [0,\vs) \rightarrow \R^d$, a constant $L_F>0$ exists such that, for all $u,v\in (0,\vs)$, 
\begin{align}
\big| \bm{F}(u)-\bm{F}(v)\big|^2 \leq L_F|u-v| \big|\Phi(u)-\Phi(v)\big| . 
\end{align}
\stepcounter{AssDD}
\item \label{ass:u0} For the initial condition $u_0 \in L^\infty(\Om)$ one has that $0 \leq u_0(\x) \leq \vs-\epsilon$ almost everywhere, for some $\epsilon>0$ if $\Phi_\Maxi=\infty$ in \ref{ass:A1a}, or $\epsilon=0$ if $\Phi_\Maxi<\infty$.
\end{enumerate}
With the practical applications in mind, the nonlinear functions $\Phi$ and $F$ are defined only for positive arguments. However, they can be extended by constants for negative arguments, without affecting the theoretical results.

\begin{remark}[Generality of Assumption \ref{ass:A1a} and \ref{ass:A1b}]
    Assumption \ref{ass:A1a} states that $\Phi$ is only locally Lipschitz, and does not impose any convexity conditions. It allows $\Phi^\prime$ to vanish at $u=0$, and $\Phi$ to blow up at $u=\vs$.  Therefore, the problem can be doubly degenerate. 
    
    Further, \eqref{eq:condition} gets $B'(s)=\Phi'(b(s)) b'(s)$, so, under the assumptions \ref{ass:A1a} and \ref{ass:A1b}, either $B'(0)>0$ when $\Phi'(0)>0$, or $B$ is convex in a right neighborhood of $0$. \label{rem:Aprop}
\end{remark}

\begin{remark}[Heterogeneous $\Phi$, $\bm{F}$] \label{rem:heterogeneity}
The functions $\Phi$ and $\bm{F}$ can also depend explicitly on $\x\in \Om$, i.e., $\Phi:\Om\times\R\to \R$ and $\bm{F}:\Om\times \R\to \R^d$. In this case, $\Phi(\x,u)$ and $\bm{F}(\x,u)$ are required to be Carath\'eodory functions implying that they need to be measurable in $\x$. Then for each $\x\in \Om$, the decomposition functions $b_\x=b(\x,\cdot)$, $B_\x=B(\x,u)$ need to satisfy \eqref{eq:condition}. With this change, all the subsequent results remain valid, and hence they can be applied to heterogeneous problems as found in porous domains \cite{seu2018linear}.
\end{remark}

\begin{remark}[Boundary conditions]
    For simplicity, the boundary conditions are assumed homogeneous Dirichlet for \eqref{eq:1}. Non-homogeneous Dirichlet boundary condition, like $w=h$ at $\p\Om$ can also be considered provided $h$ is the trace of an $H^1(\Om)$ function that is bounded by $0$ and $\Phi(\vs-\epsilon)]$ a.e. in $\Om$ for some $\epsilon>0$. Moreover, it is possible to assume a homogeneous Neumann condition for the problem. However, for the case when $\lim_{v\nearrow \vs}\Phi(v)=\infty$, it has to be ensured that the solution $u$ stays bounded away from $\vs$. We refer to \cite{mitra2023wellposedness}, where necessary and sufficient conditions are given to avoid the singular value. Since discussing the intricacies of the boundary condition is not the focus of this work, we limit the future discussions to the case of homogeneous Dirichlet conditions.
\end{remark}


\subsubsection{Extra regularity assumptions on $\Phi$ and the construction of $b,\,B$ functions}\label{sec:extra_regul}
For the convergence of the M-scheme we additionally assume that,
\begin{enumerate}[label=(B.\theAssM)]
    \item $\Phi$ has locally Lipschitz continuous derivatives in $(0,\vs)$. Moreover, there exists $u^*\in (0,\vs)$ such that (after possible rescaling of $\Phi$),\label{ass:Bphi}
\begin{equation*}
    \Phi'(u)\begin{cases}
        \leq 1 &\text{ for } u\in (0,u^*),\\
        =1&\text{ for } u=u^*,\\
        \geq 1 &\text{ for } u\in (u^*,\vs).
    \end{cases}
\end{equation*}    
\end{enumerate}

Observe that \ref{ass:Bphi} is trivially satisfied if $\Phi\in C^2([0,\vs))$ and convex.
\begin{lemma}[Decomposition of $\Phi$]\label{lemma:decomspoition}
Under the assumptions \ref{ass:A1a}, \ref{ass:A1b} and \ref{ass:Bphi}, the functions
    \begin{subequations}\label{eq:bBexpression}
    \begin{align}
            &b(s):=\int_0^s \min\left\{1, \frac{1}{\Phi'(b(\rho))}\right\}d\rho=\begin{cases}
             s &\text{ if } s\leq u^*,\\
             \Phi^{-1}(\Phi(u^*) + s-u^*) &\text{ if } s\geq u^*.
         \end{cases}\\
         &B(s):=\int_0^s \min \left\{1, \Phi'(b(\rho))\right\}d\rho=\begin{cases}
             \Phi(s) &\text{ if } s\leq u^*,\\
             \Phi(u^*) + s-u^*&\text{ if } s\geq u^*.
         \end{cases}
    \end{align}   
\end{subequations}
have Lipschitz continuous derivatives, and satisfy \eqref{eq:condition}.
\end{lemma}




\subsection{Weak Formulations}
In this section, $n \in \{1, \dots, N\}$ is fixed and $u_{n-1}, s_{n-1} \in L^2(\Om)$ are assumed known, satisfying $u_{n-1} = b(s_{n-1})$. Here we give the weak forms of the time-discrete problems and their linearization at the time step $t_n$. To do so, we use the time-discrete problems given formally in Section \ref{sec:wellposedness}, and the space $\mathcal{Z}$ in \eqref{eq:Z}. 
\subsubsection{Time discretization}   
The weak formulation of the time-discrete, nonlinear problem at time step $t_n$ is given below.  
  \begin{problem}[{Weak formulation of system \eqref{eq:t}}]\label{prob:time-disc}
Find the triple $(s_n, u_n, w_n) \in \mathcal{Z}$ such that  for all $(\psi,\phi,\f)\in\calZ$, the  following holds
\begin{subequations}\label{eq:L1}
  \begin{align}
    \left(\frac{1}{\tau}(u_n-u_{n-1}), \varphi\right) + (\nabla w_n, \nabla \varphi) &= (\bm{F}(b(s_n)),\nabla \varphi) + \langle f, \varphi\rangle, \label{eq:L1a}\\
    (u_n, \phi) &= (b(s_n),\phi), \\[.5em]
    (w_n,\psi) &= (B(s_n),\psi).
\end{align}   
\end{subequations}

\begin{proposition}[Well-posedness of Problem \ref{prob:time-disc}] If $\t\in (0,L_F^{-1})$, Problem \eqref{prob:time-disc} has a unique solution $(s_n, u_n, w_n) \in \calZ$. Moreover, if $u_{n-1}, f\geq 0$ a.e., then $u_n\geq 0$ a.e. in $\Om$.\label{prop:well-posedness}
\end{proposition}
\end{problem}

\noindent
Since proving \Cref{prop:well-posedness} is not the main focus of this work, we postpone the proof to \Cref{Appendix:proof}

\subsubsection{Linearization} 
The weak form of the double-splitting linearization  in \eqref{eq:3} is as follows
\begin{problem}[{Weak formulation of system $\eqref{eq:3}$}]\label{prob:3}
Let $i \in \N$, $i > 0$ and assume $s_{n}^{i-1} \in L^2(\Om)$ known. Find the triple $\left(s_n^i, u_n^i, w_n^i\right) \in \mathcal{Z}$ such that for all $(\psi,\phi,\f)\in\calZ$, the  following holds 
\begin{subequations}\label{eq:linearization_gen}
\begin{align}
\left(\frac{1}{\tau}(u^i_n-u_{n-1}), \f\right)+ (\del w^i_n, \del \f) &=(\bm{F}(b(s_n^{i-1})),\del \f)+\langle f, \f\rangle,\label{eq:linearization_gen.a}\\
 (u^i_n-b(s^{i-1}_n),\phi)&=(L^i_{b,n}(s^i_n-s^{i-1}_n), \phi), \label{eq:linearization_gen.b}\\[.5em]
(w^i_n-B(s^{i-1}_n),\psi)&=(L^i_{B,n}(s^i_n-s^{i-1}_n),\psi). \label{eq:linearization_gen.c}
\end{align}
\end{subequations}
The bounded functions $L^i_{b,n}, L^i_{B,n}: \R \rightarrow \R^+$ are given in Section \ref{sec:commonLin}, depending on $s^{i-1}_n$.
\end{problem}
A natural choice for the starting point is $s_n^0 = s_{n-1}$, but this is not compulsory for LS.  

\begin{proposition}[Well-posedness and consistency of Problem \ref{prob:3}]
    Assume that $L^i_b$ and $L^i_B$ are bounded above and below by positive constants, uniformly in $i\in \N$. Then, Problem \ref{prob:3} has a unique solution. If $\{\left(s_n^i, u_n^i, w_n^i\right)\}_{i\in \N} \subset \mathcal{Z}$ is a Cauchy sequence, then it converges in $\mathcal{Z}$ to $(s_n,u_n,w_n)$ as $i\to \infty$.
\end{proposition}

The well-posedness follows from \Cref{lemma:equivalence} in \Cref{sec:Adaptive} below. Specifically, \eqref{eq:linearization_gen} involves a bilinear functional on $\mathcal{Z} \times \mathcal{Z}$ and a linear functional on $\mathcal{Z}$. If $L^i_{b,n}$ and $L^i_{B,n}$ are bounded as above, these are bounded and the former is also elliptic. By the strong convergence of Cauchy sequences in $\mathcal{Z}$ and the Lipschitz continuity of $b$ and $B$, the limit $i \to\infty$ of the solution to Problem \ref{prob:3} solves Problem \ref{prob:time-disc}.


\subsection{Commonly used linearization schemes}\label{sec:commonLin}
This work will focus on three different linearization schemes: Newton, L, and  M. These schemes can be written in a unified framework, for both no-splitting and double-splitting approaches. For the former, the choice of $L^i_{\beta,n}$ in \eqref{eq:L} is discussed in Section \ref{sec:direct}. For the latter, the choice of $L^i_{b,n}$ and $L^i_{b,n}$ in \eqref{eq:3} is presented in \Cref{tab:1}.


\begin{table}[h]
  \begin{center}
    \begin{tabular}{|l|c|c|}
      \hline
      \textbf{Scheme} & \textbf{$L^i_{b,n}$} & \textbf{$L^i_{B,n}$} \\
      \hline\hline
      Newton & $b'(s_n^{i-1})$ & $B'(s_n^{i-1})$ \\
      \hline
      L-scheme & $1 + \epsilon \geq \sup b'$ & $1 + \epsilon \geq \sup B'$ \\
      \hline
      M-scheme &
      \begin{tabular}{@{}l@{}}
        $\min\big( \max( b'(s_n^{i-1}) + M\tau,$\\
        \quad $2M\tau ),\ 1 + \epsilon \big)$
      \end{tabular} &
      \begin{tabular}{@{}l@{}}
        $\min\big( \max( B'(s_n^{i-1}) + M\tau,$\\
        \quad $2M\tau ),\ 1 + \epsilon \big)$
      \end{tabular} \\
      \hline
    \end{tabular}
  \end{center}
  \caption{Choices of $L^i_{b,n}$ and $L^i_{B,n}$ in the double-splitting formulation \eqref{eq:3}, leading to different linearization schemes. Here, $\epsilon > 0$ is an arbitrarily small constant.}
  \label{tab:1}
\end{table}

The parameter $\epsilon > 0$ appearing in  \Cref{tab:1} can be chosen freely. The parameter $M > 0$ is subject to restrictions depending on the nonlinear functions $b$ and $B$, see \cite{mitra2019modified} and \Cref{M_ineq_LB} below. Observe that, in the case of the Newton and the M-scheme, the factors $L^i_{b,n}$ and $L^i_{B,n}$ depend on the previous iteration $s_n^{i-1}$ and therefore they are changing spatially and with iteration. For the L-scheme instead, the factors are constant. The value $L^i_{b,n}=L^i_{B,n}=1 + \epsilon$ is due to the fact that, as stated in \eqref{eq:condition}, $b^\prime$ and $B^\prime$ are bounded by $1$. Also, note that the M-scheme is conceptualized as the combination of the NS and the LS. Taking $M > \frac 2 \tau {(1+\epsilon)}$ yields precisely the L-scheme. On the other hand, by choosing $M=0$ one obtains the Newton scheme.

\begin{remark}[Relation between the no-splitting/ double-splitting  formulations]
The no-splitting formulation \eqref{eq:L} can be thought of as special case of the double-splitting formulation corresponding to cases when either $B$ or $b$ is the identity function. In this case, inserting $L^i_{B,n}=1$ or $L^i_{b,n}=1$ respectively, one obtains either $w^i_n=s^i_n$ or $u^i_n=s^i_n$ from Problem \ref{prob:3} which converts them to the no-splitting approach.
\end{remark}



\section{Convergence of the double-splitting schemes}\label{section3}
This section contains the rigorous proof for the convergence of the L-scheme and the M-scheme. The main results are summarized in the following two theorems.
\begin{theorem}[Convergence of the L-scheme] \label{theo:LS}
Let $(s_n,u_n,w_n)\in \calZ$ be a weak solution to  Problem \ref{prob:time-disc}, and $\{(s_n^i, u_n^i, w_n^i)\}_{i\in \N}\subset \calZ$ the array of solutions to Problem \ref{prob:3},  with the choice $L^i_{b,n}=L^i_{B,n}=1+\epsilon$ for a given $\epsilon\in (0, 1-\t L_F)$ (see \Cref{tab:1}). 
Under the assumptions \ref{ass:A1a}, \ref{ass:A1b} and \ref{ass:u0}, for $\tau \in(0,L_F^{-1})$ one has
\begin{align}\label{eq:conv-LS}
\|s_n^i - s_n\|_{L^1(\Omega)} + \|u_n^i - u_n\|_{L^1(\Omega)} + \|w_n^i - w_n\|_{H^1(\Om)} &\to 0 \text{ as } i \to \infty.
\end{align}
Moreover, in the  \underline{single degenerate case} when $\ell_B := \inf B' > 0$, there exists constants $\theta,\vartheta>0$ not depending ofn $\ell_B$ or $\t$ such that
\begin{align}\label{eq:linear-conv-LS}
\|s_n^i - s_n\|^2 + \t\vartheta\|\nabla(w_n^i - w_n)\|^2 +\epsilon\vartheta\|(s_n^i - s_n^{i-1})\|^2\leq (1-\t\ell_B^2 \theta)\|s_{n}^{i-1} - s_n\|^2.
\end{align} 
\end{theorem}
\begin{remark}[Linear convergence of the L-scheme]\label{remark:LS}
In Section \ref{sec:proofLscheme} we show that $\ell_B = \inf B^\prime$, while $\theta$ does not depend on $b$ or $B$. Therefore, if $B^\prime$ is bounded away from 0, \eqref{eq:linear-conv-LS} implies that the L-scheme converges linearly, with a contraction rate $\a=(1-\t \ell_B^2 \theta)^\frac{1}{2} $. 
This is similar to the convergence results in \cite{pop2004mixed, list2016study, RaduIMA}, obtained for the no-splitting scheme \eqref{eq:L}, first for a situation in which $B$ is linear, but $b^\prime$ non-negative but bounded, and then for the case that $b$ is H{\"o}lder continuous. 
For the double-splitting scheme studied here, we have slightly extended the convergence result to the case when $0 \leq b^\prime \leq 1$, but $B$ is s.t. an $\ell_B > 0$ exists so that $\ell_B\leq B'\leq 1$. 
Despite unconditional convergence, as reported in \cite{list2016study,seu2018linear,mitra2019modified,smeets2024robust} the L-scheme has one major drawback. If either $\ell_B$ or the time-step size $\t$ is small, the contraction rate $\a$ approaches 1, which slows the convergence. 
\end{remark}
\begin{theorem}[Convergence of the M-scheme]\label{theo:MS}
Let $(s_n,u_n,w_n)\in \calZ$ be a weak solution to  Problem \ref{prob:time-disc}, and $\{(s_n^i, u_n^i, w_n^i)\}_{i\in \N}\subset \calZ$ the array of solutions to Problem \ref{prob:3},  with $L^i_{b,n},\,L^i_{B,n}$ chosen as for the M-scheme  in \Cref{tab:1}. 
Assume that $\Lambda>0$ exists such that for all $i\in \N$,
\begin{align}\label{boundedness}
   \|s_n^i - s_n\|_{L^\infty(\Omega)} \leq \Lambda \t.  
\end{align}
Let $M_0 := \Lambda \max\{\snorm{b'}_{\rm Lip},\snorm{B'}_{\rm Lip}\} > 0$. If $M > M_0+ L_F$, $0<\t<\min(1/(M+M_0), L_F^{-1})$, and $0<\epsilon<(M-M_0-L_F)\tau$, then 
under the assumptions \ref{ass:A1a}, \ref{ass:A1b}, \ref{ass:u0} and \ref{ass:Bphi}, one has
\begin{align}\label{eq:conv-MS}
    \|s_n^i - s_n\|_{L^1(\Omega)} + \|u_n^i - u_n\|_{L^1(\Omega)} + \|w_n^i - w_n\|_{H^1(\Om)} \to 0 \text{ as } i \to \infty.
\end{align}
Moreover, in the \underline{single-degenerate case} when $\ell_B := \inf B' > 0$, there exists $\Theta,\vr>0$, such that
\begin{align}\label{eq:linear-conv-MS}
\hspace{-0.5em}\|s_n^i - s_n\|^2 + \vr\|\nabla(w_n^i - w_n)\|^2 +4M\epsilon\vr\|s_n^i - s_n^{i-1}\|^2\leq (1-\ell_B^2 \Theta)\|s_{n}^{i-1} - s_n\|^2.
\end{align} 
Additionally, in the \underline{non-degenerate case} when $\ell := \min\{\inf b', \inf B'\} > 0$, then
\begin{align}\label{eq:lin-conv-MS}
\ell \|s_n^i - s_n\|^2 + \frac{\t}{2}\|\nabla(w^i_n - w_n)\|^2 +\epsilon M\tau\|s_n^i - s_n^{i-1}\|^2\leq M\t\|s_{n}^{i-1} - s_n\|^2.
\end{align}
\end{theorem}
\begin{remark}[Linear convergence of the M-scheme]\label{remark:MS}
As in \Cref{remark:LS}, if $\inf B^\prime =\ell_B>0$ (thus the problem is at most single degenerate), the M-scheme converges linearly independent of $\t$, with a contraction rate $\alpha=(1-\ell_B^2\Theta)^{\frac{1}{2}}$. 
In the non-degenerate case when $\ell>0$, for $\tau\leq \ell/M$ and using \eqref{eq:lin-conv-MS} one obtains that the M-scheme converges linearly with the  contraction rate $\a=\sqrt{\frac{M\t}{\ell}}$. In fact, one gets $\a= \min\left\{\sqrt{\frac{M\t}{\ell}}, \sqrt{(1-\ell_B^2\Theta)}\right\}$. 
This improves the convergence speed of the M-scheme when compared to the L-scheme, as, the contraction rate reduces for small values of $\t$.
\end{remark}  

\begin{remark}[Boundedness assumption \eqref{boundedness}]
The $L^\infty(\Om)$ boundedness assumed in  \eqref{boundedness} was used in \cite{mitra2019modified} to prove the convergence of the scheme for the no-splitting case, and in \cite{smeets2024robust} for the case that resembles the single splitting in \cite{cances2021error}. This assumption is motivated by the choice $s^0_n=s_{n-1}$. In a more general setting, if the solution is H{\"older}-continuous in time with the exponent $\mu\in (0,1)$, one needs the existence of a $\Lambda>0$ such that
\begin{align}\label{eq:ini-guess}
    \|s_n-s_{n}^0\|_{L^\infty(\Om)}= \|s_n-s_{n-1}\|_{L^\infty(\Om)}\leq \Lambda \t^\mu. 
\end{align}
In particular, $\mu=1$ was used in \cite{mitra2019modified}, and $\mu<1$ in \cite{smeets2024robust}. Furthermore, as follows from  Lemmata 3.1 and 4.1 in \cite{mitra2019modified}, and Proposition 4.9 in \cite{smeets2024robust}, \eqref{eq:ini-guess} implies \eqref{boundedness} for either the no-splitting M-scheme, or the single-splitting variant. Since the proof relies on elaborate arguments, for conciseness we take here \eqref{boundedness} as an assumption.
\end{remark}
\subsection{A generic convergence criterion}\label{sec:conv}
Before giving the proofs for  \Cref{theo:LS,theo:MS},  we derive a sufficient criterion for the convergence of any linearization scheme having the form given in Problem \ref{prob:3}, with $L^i_{b\slash B,n}$ bounded and strictly positive. To this end, we assume $n \in \{1, \dots, N\}$ fixed and for any $i \in \N$ we let $e_{\zeta}^{i}$ denote the errors of the $i^{\rm th}$ iterate in $\zeta \in \{s, u, w\}$, at the $n^{\rm th}$ time-step. Further, $e^i_b$, $e^i_B$ denote the errors involving the $b$, $B$ functions, namely  
\begin{subequations}\label{eq:def_errors}
\begin{align}
&e_{s}^{i} = s_{n}^{i} - s_{n}\label{eq:def_errors_a},\quad  e_{u}^{i} = u_{n}^{i} - u_{n},\quad  e_{w}^{i} = w_{n}^{i} - w_{n},\\
&e^i_b:=b(s^i_n)-b(s_n),\quad e^i_B:=B(s^i_n)-B(s_n)\label{eq:def_errors_b}.
\end{align}
Moreover, for $\rho \in \{b,B\}$, $\rho[\cdot,\cdot]: \R^2 \to \R$ denotes the difference quotient
\begin{align}\label{eq:def_errors_c}
  \rho[t,v]=\begin{cases}
      \frac{\rho(t)-\rho(v)}{t-v} &\text{ if } t\not =v\\
      \rho'(t) &\text{ if } t=v . 
  \end{cases} 
\end{align}
\end{subequations}
Obviously, by Assumption \ref{ass:A1a} one has $
\rho[t,v] \in [0,1]$, while $e_\rho^i=\rho[s^i_n,s_n]\,e^i_s.    
$

Subtracting $\eqref{eq:L1a}$ from $\eqref{eq:linearization_gen.a}$ and rearranging the terms yields
\begin{align}
(e_u^{i},\varphi)+\t (\del e_w^{i},  \del \varphi)=\t \left( \bm{F}(b(s_n^{i-1}))-\bm{F}(b(s_n)), \del \varphi\right).
\end{align}
Inserting the test function $\varphi=e_w^i\in H^1_0(\Om)$ one has
\begin{align}\label{eq:sort}
(e_u^{i},e_w^{i})+\t \|\del e_w^{i}\|^2=\t \left( \bm{F}(b(s_n^{i-1}))-\bm{F}(b(s_n)), \del e^i_w\right).
\end{align}
First let us try to estimate the $(e_u^{i},e_w^{i})$ term above. 
Observe that from $\eqref{eq:linearization_gen}$, using the shorthand notations in \eqref{eq:def_errors}, one has a.e. in $\Om$ that
\begin{subequations}
\begin{align}
    &e_u^{i}=(b(s_n^{i-1})-b(s_n))+L^{i}_{b,n}(s^{i}_n-s^{i-1}_n)\overset{\eqref{eq:def_errors_a},\eqref{eq:def_errors_b}}{=}e^{i-1}_b + L^{i}_{b,n}(e_s^{i}-e_s^{i-1}), \label{eq:b0}\\
    &e_w^{i}=(B(s_n^{i-1})-B(s_n))+L^{i}_{B,n}(s^{i}_n-s^{i-1}_n)\overset{\eqref{eq:def_errors_a},\eqref{eq:def_errors_b}}{=}e^{i-1}_B + L^{i}_{B,n}(e_s^{i}-e_s^{i-1}). \label{eq:b1}
\end{align}
\end{subequations}
Integrating the product of the above over  $\Om$,  one obtains
\begin{align*}
 (e_u^i,e_w^i)&=\int_\Om \left(e^{i-1}_bL^{i}_{B,n}+L^{i}_{b,n} e^{i-1}_B\right)(e_s^{i}-e_s^{i-1})  \nonumber +\int_\Om \left(e^{i-1}_b e^{i-1}_B\right)+ \int_\Om \left(L^{i}_{B,n} L^{i}_{b,n}\right) (e_s^{i}-e_s^{i-1})^2\nonumber\\
&\overset{\eqref{eq:def_errors_c}}=\int_\Om \bigg(b[s_n^{i-1},s_n]\, L^{i}_{B,n}+L^{i}_{b,n}\, B[s_n^{i-1},s_n]\bigg)(e_s^{i}-e_s^{i-1})(e_s^{i-1}) \nonumber \\
&+\int_\Om \bigg(b[s_n^{i-1},s_n]\,B[s_n^{i-1},s_n] \bigg) (e_s^{i-1})^2+ \int_\Om \left(L^{i}_{B,n} L^{i}_{b,n}\right) (e_s^{i}-e_s^{i-1})^2.
\end{align*}
Applying \eqref{eq:equality} in the first term on the right
\begin{align}\label{eq:big}
(e_u^i,e_w^i)&=\frac{1}{2}\int_\Om \bigg(b[s_n^{i-1},s_n]L^{i}_{B,n}+L^{i}_{b,n} B[s_n^{i-1},s_n]\bigg)|e_s^{i}|^2 \nonumber\\
&-\frac{1}{2} \int_\Om  \bigg(b[s_n^{i-1},s_n]L^{i}_{B,n}+L^{i}_{b,n} B[s_n^{i-1},s_n]-2 \big( b[s_n^{i-1},s_n]B[s_n^{i-1},s_n]\big)\bigg)|e_s^{i-1}|^2 \nonumber\\
&+\int_\Om \left(L^{i}_{B,n} L^{i}_{b,n}-\frac{1}{2} \left(   b[s_n^{i-1},s_n]\,L^{i}_{B,n}+L^{i}_{b,n} B[s_n^{i-1},s_n]\right) \right) |e_s^{i}-e_s^{i-1}|^2.
\end{align}
Next, we estimate the last term in \eqref{eq:sort}. Observe from \ref{A.2} and $B=\Phi\circ b$ from \eqref{eq:condition} that
\begin{align}
    |\bm{F}(b(s_1))-\bm{F}(b(s_2))|^2\leq L_F |b(s_1)-b(s_2)||B(s_1)-B(s_2)|.
\end{align}
Then, from Cauchy-Schwarz and Young's inequalities, one has
\begin{align}\label{eq:F-manipulation}
& \t\left( \bm{F}(b(s_n^{i-1}))-\bm{F}(b(s_n)), \nabla e_w^i\right) \leq \t \left(\int_\Om |\bm{F}(b(s_n^{i-1}))-\bm{F}(b(s_n))|^2 \right)^{\frac{1}{2}} ||\del e_w^i||  \nonumber\\
&\overset{\ref{A.2}}\leq \t \left(L_F\int_\Om |e^{i-1}_b||e^{i-1}_B| \right)^{\frac{1}{2}} ||\del e_w^i|| \overset{\eqref{eq:def_errors_c}}\leq \t\left(\int_\Om L_F\, b[s^{i-1}_n,s_n]\, B[s^{i-1}_n,s_n]\,|e^{i-1}_s|^2 \right)^{\frac{1}{2}} \|\del e_w^i||\nonumber\\
&\overset{\eqref{young}}{\leq} \frac{\t L_F}{2}\int_\Om \, b[s^{i-1}_n,s_n]\, B[s^{i-1}_n,s_n]\,|e^{i-1}_s|^2 +\frac{\t}{2}||\del e_w^i||^2.
\end{align}
Inserting $\eqref{eq:big}$ and \eqref{eq:F-manipulation} in $\eqref{eq:sort}$, after rearranging and canceling terms one gets
\begin{align*}\label{eq:main-ineq}
&\int_\Om \bigg( b[s_n^{i-1},s_n]L^{i}_{B,n}+L^{i}_{b,n} B[s_n^{i-1},s_n]\bigg)|e_s^{i}|^2 +\t \|\del e_w^{i}\|^2+\nonumber \\
& \int_\Om \bigg(2 L^{i}_{B,n} L^{i}_{b,n}-\big(b[s_n^{i-1},s_n]L^{i}_{B,n}+L^{i}_{b,n}  B[s_n^{i-1},s_n]\big)\bigg) |e_s^{i}-e_s^{i-1}|^2\nonumber \\
&\leq \int_\Om \bigg(b[s_n^{i-1},s_n]L^{i}_{B,n}+L^{i}_{b,n} B[s_n^{i-1},s_n] -(2-\t L_F) b[s_n^{i-1},s_n]\,B[s_n^{i-1},s_n]\bigg)|e_s^{i-1}|^2.
\end{align*}
We define the coefficient functions 
\begin{subequations}\label{eq:G-coeff}
    \begin{align}
    &G_1^i:=  b[s_n^{i-1},s_n]L^{i}_{B,n}+L^{i}_{b,n} B[s_n^{i-1},s_n],\\
    &G_2^i:= 2 L^{i}_{B,n} L^{i}_{b,n}-\big(b[s_n^{i-1},s_n]L^{i}_{B,n}+L^{i}_{b,n}  B[s_n^{i-1},s_n]\big),\\
    &G_3^i:= b[s_n^{i-1},s_n]L^{i}_{B,n}+L^{i}_{b,n} B[s_n^{i-1},s_n] -(2-\t L_F) b[s_n^{i-1},s_n]\,B[s_n^{i-1},s_n].
\end{align}
\end{subequations}
Observe that, by \eqref{eq:condition}, if the factors $L^{i}_{b,n}$ and $L^{i}_{B,n}$ are chosen as in Table \ref{tab:1}, the functions in \eqref{eq:G-coeff} are all positive and one has $G_1^i > G_3^i$. With this, the inequality above becomes
\begin{equation}\label{eq:main-ineq}
\int_\Om G_1^i |e_s^{i}|^2 +\t \|\del e_w^{i}\|^2 + 
\int_\Om G_2^i |e_s^{i}-e_s^{i-1}|^2 \leq \int_\Om G_3^i |e_s^{i-1}|^2.
\end{equation}
We can now state a generic criterion guaranteeing the convergence of the linearization schemes. 
\begin{lemma}[Sufficient condition for convergence] \label{lemma:criteria} 
Let $L > 0$ be an upper bound for $L^i_{b,n}$ and $L^i_{B,n}$. Assume the existence of the constants 
$C, \xi >0$  such that 
\begin{align}\label{eq:conv-criteria}
    G_1^i\geq C\geq G_3^i, \text{ and } G^i_2\geq \xi,
\end{align}
uniformly w.r.t. $i\in \N$. Then, as $i \to \infty$, one has 
$$
\|s_n^i - s_n\|_{L^1(\Omega)}+\|u_n^i - u_n\|_{L^1(\Omega)} + \|w_n^i - w_n\|_{H^1(\Om)} \to 0. 
$$  
\end{lemma}
Before giving the proof we note that, if $L^i_{b, n}$ and $L^i_{B, n}$ are chosen as in Table \ref{tab:1}, then $L = 1 + \epsilon$. Also, although $G_1^i > G_3^i$ a.e. in $\Om$, the condition on $C$ is not superfluous, as it has to be fulfilled for a.e. $x \in \Om$, and uniformly w.r.t. $i$. Similarly, the lower bound on $G_2^i$ should be strictly positive, in the same uniform way. 
\begin{proof}
    From \eqref{eq:main-ineq} applying \eqref{eq:conv-criteria} one obtains $C\|e_s^i\|^2 + \xi \|e_s^i-e_s^{i-1}\|^2+\t\|\del e_w^i\|^2\leq C\|e_s^{i-1}\|^2$. Adding from $i=1$ to $i=k$ yields after cancellation of common terms
\begin{align}
C\|e_s^k\|^2 + \xi\sum_{i=1}^k \|e_s^i-e_s^{i-1}\|^2+\t\sum_{i=1}^k \|\del e_w^i\|^2 \leq C\|e_s^{0}\|^2.
\end{align}
    The series $\sum_{i=1}^{k} \|\del e_w^i\|^2$ and $\sum_{i=1}^{k} \|e_s^i-e_s^{i-1}\|^2$ are absolutely convergent, which implies
\begin{align}\label{eq:inq}
    \|\del e_w^i\| \to 0 \quad   \text{and} \quad \|e_s^i-e_s^{i-1}\| \to 0 \quad \text{as} \quad i \rightarrow \infty
\end{align}
By the Poincar\'e inequality, we find that $e^i_w\to 0$ in $H^1(\Om)$, and consequently $w_n^i\rightarrow w_n$. Using the second term of $\eqref{eq:inq}$ together with $\eqref{eq:3c}$, we find that 
\begin{align*}
   0\leq \left\|w^i_n-B(s^{i-1}_n)\right\|=\|L^i_{B,n}(s_n^i-s_n^{i-1})\|=\|L^i_{B,n}(e_s^i-e_s^{i-1})\| \leq L \|e_s^i-e_s^{i-1}\|\rightarrow 0.
\end{align*}
This gives $w^i_n-B(s^{i-1}_n)\rightarrow 0$ as $i \rightarrow \infty$ in $L^2(\Om)$. Hence, we have that
\begin{align}
w_n^i\rightarrow w_n \text{ strongly in } H^1(\Om), \text{ and } B(s_n^i) \rightarrow w_n=B(s_n) \text{ strongly in } L^2(\Om).
\end{align}
The convergence of  $s^i_n\to s_n$ in $L^1(\Om)$  follows the convergence of $B(s^i_n)\to B(s_n)$ in $L^2(\Om)$, by applying Lemma 3.10 of \cite{smeets2024robust}. This is possible since, as discussed in \Cref{rem:Aprop}, either $B'(0)> 0$, or $B$ is locally convex at 0. We conclude the proof of the lemma by noticing that
\begin{align}
    \|u^i_n-u_n\|_{L^1(\Om)}&= \|u^i_n-b(s_n)\|_{L^1(\Om)}  \leq  \|u^i_n-b(s^{i-1}_n)\|_{L^1(\Om)}+ \|b(s^{i-1}_n) -b(s_n)\|_{L^1(\Om)}\\
    &\overset{\eqref{eq:linearization_gen}}\leq \|L^i_{b,n}(e_s^i -e_s^{i-1})\|_{L^1(\Om)} \overset{\eqref{eq:condition}}+ \|e^{i-1}_s\|_{L^1(\Om)}\to 0.\nonumber
\end{align}
\end{proof}
Now we show that LS and MS both satisfy the criterion in \eqref{eq:conv-criteria}.

\subsection{Convergence proof for the L-scheme, \Cref{theo:LS}
}\label{sec:proofLscheme}

We show here that the  convergence criterion discussed in \Cref{sec:conv} applies for LS. One takes  $L^{i}_{B,n}=L^{i}_{b,n}=1+\epsilon$ as in \Cref{tab:1}, for some  $\epsilon \in (0, 1)$ that will be mentioned below. Using the mean value theorem, from \eqref{eq:G-coeff} one gets 
\begin{subequations}\label{eq:G_LS}
    \begin{align}
        G_1^i&= b[s_n^{i-1},s_n]L^{i}_{B,n}+L^{i}_{b,n} B[s_n^{i-1},s_n]=(1+\epsilon)(b[s_n^{i-1},s_n] +  B[s_n^{i-1},s_n])\nonumber\\
        &\overset{\eqref{eq:def_errors_c}}= (1+\epsilon)\frac{(b+B)(s_n^{i-1})-(b+B)(s_n)}{s_n^{i-1}-s_n}=(1+\epsilon)(b'+B')(\Upsilon)\overset{\eqref{eq:condition}}\geq 1+\epsilon, \\
        G_2^i&= (1+\epsilon)\left[2(1+\epsilon)-\big(b[s_n^{i-1},s_n]+ B[s_n^{i-1},s_n])\right]=(1+\epsilon)\left[2 \epsilon + 2 -(b'+B')(\Upsilon)\right]\nonumber\\
        &\overset{\eqref{eq:condition}}\geq 2\epsilon (1+\epsilon),\\
        G_3^i&= (1+\epsilon)(b[s_n^{i-1},s_n] +  B[s_n^{i-1},s_n])- (2-\t L_F)b[s_n^{i-1},s_n]\, B[s_n^{i-1},s_n]\nonumber\\
        &= 1+\epsilon +\bigg(\tau L_F-(1-\epsilon)\bigg) b[s_n^{i-1},s_n] B[s_n^{i-1},s_n] -(1+\epsilon)\left(1-b[s_n^{i-1},s_n]\right) \left(1- B[s_n^{i-1},s_n]\right) \nonumber\\
        &\overset{\eqref{eq:def_errors_c}}\leq 1+\epsilon. \label{eq:G3i_LS}
    \end{align}
The argument (function) $\Upsilon$ in the above is defined almost everywhere by the mean value theorem for the function $b+B$. One has $\Upsilon \in I(s_n^{i-1},s_n)$, the interval with endpoints $s_n^{i-1}$ and $s_n$, as defined in \Cref{sec:notations}.  The last inequality in \eqref{eq:G3i_LS} holds since $\t L_F < 1$ (as stated in \Cref{theo:LS}), so there exists $\epsilon > 0$ so that  $\tau L_F-(1-\epsilon) \leq 0$. Further, by \eqref{eq:def_errors_c}, $0\leq b[s_n^{i-1},s_n],\, B[s_n^{i-1},s_n]\leq 1$. Hence, taking $C=1+\epsilon$ and $\delta = 2 \epsilon(1 + \epsilon)$ in \Cref{lemma:criteria} proves the convergence of LS.

In the non-degenerate case when $\inf B'=\ell_B>0$, multiplying \eqref{eq:b1} by $e^i_s$ gives
 \begin{align}
     e^i_w\, e^i_s&=(B[s^{i-1}_n,s_n]\, e^{i-1}_s + (1+\epsilon)(e^i_s-e^{i-1}_s))e^i_s\nonumber\\
&\overset{\eqref{eq:equality}}=\frac{1+\epsilon + B[s^{i-1}_n,s_n]} {2}|e^i_s|^2 -\frac{1+\epsilon-B[s^{i-1}_n,s_n]}{2}|e^{i-1}_s|^2 +\frac{1+\epsilon-B[s^{i-1}_n,s_n]}{2}|e^i_s-e^{i-1}_s|^2.\nonumber
 \end{align}
 Since $\ell_B\leq B[s^{i-1}_n,s_n]\leq 1 $, using  Young's inequality \eqref{young} $e^i_w e^i_s\leq \frac{1}{2\ell_B}|e^i_w|^2 + \frac{\ell_B}{2}|e^i_s|^2$  gets 
 \begin{align*}
     (1+\epsilon)\|e^i_s\|^2 -(1+\epsilon-\ell_B)\|e^{i-1}_s\|^2 \leq \frac{1}{\ell_B}\|e^i_w\|^2\overset{\eqref{eq:P}}\leq \frac{C^2_\Om h_\Om^2}{\ell_B}\|\del e^i_w\|^2. 
 \end{align*}
The last inequality follows is the Poincare inequality, where $h_\Om$ is the diameter of $\Om$. Inserting this into \eqref{eq:main-ineq} and using \eqref{eq:G_LS} gives 
\begin{align*}
    &(1+\epsilon)\left(1 +\frac{\t \ell_B}{2C_\Om^2 h_\Om^2} \right)\|e^i_s\|^2 + 2\epsilon(1+\epsilon) \|e^i_s-e^{i-1}_s\|^2 + \frac{\t}{2}\|\del e^i_w\|^2\\
    &\leq \left((1+\epsilon)\left(1 +\frac{\t \ell_B}{2C_\Om^2 h_\Om^2} \right) -\frac{\t\ell_B^2}{2C_\Om^2 h_\Om^2}\right)\|e^{i-1}_s\|^2.
\end{align*}
Since $\epsilon < 1$, $\t \leq T$ (the final time) and   $0<\ell_B \leq 1$, one gets \eqref{eq:linear-conv-LS} with $\theta:=\big[2(2C_\Om^2 h_\Om^2+T)\big]^{-1}$ and $\vartheta:=\theta C_\Om^2 h_\Om^2$.
\end{subequations}

\subsection{Convergence proof for the M-scheme, \Cref{theo:MS}}
For MS, $L^i_{B,n}$ and $L^i_{b,n}$ are given in \Cref{tab:1}, 
\begin{subequations}\label{Mimp}
  \begin{align}\label{imp}
   L_{b,n}^i &= \min\left\{\max\left(b'(s_n^{i-1}) + M\tau, 2M\tau\right),1+\epsilon\right\},  \\
\label{imp1}
   L_{B,n}^i &=\min \left\{\max\left(B'(s_n^{i-1}) + M\tau, 2M\tau\right),1+\epsilon\right\}.
\end{align}  
\end{subequations}

\begin{proposition}[Useful inequalities]\label{M_ineq_LB}
Under Assumptions \ref{ass:A1a}, \ref{ass:A1b}, \ref{ass:u0} and \ref{ass:Bphi}, let $\eqref{boundedness}$ hold for some $\Lambda>0$. Then, for $M>M_0 = \Lambda \max\{\snorm{b'}_{\rm Lip},\snorm{B'}_{\rm Lip}\}$, and with $B[\cdot,\cdot]$ and $b[\cdot,\cdot]$ defined in  \eqref{eq:def_errors_c}, it holds almost everywhere in $\Om$ that
\begin{subequations}
    \begin{align}
 &0<(M-M_0)\tau \leq L_{B,n}^i-B[s_n^{i-1},s_n]\leq 2M\tau, \label{M_Cong_ineq1} \\[1em]
&0<(M-M_0)\tau \leq L_{b,n}^i-b[s_n^{i-1},s_n]\leq 2M\tau.\label{M_Cong_ineq2}
\end{align} 
\end{subequations}
\end{proposition}
\begin{proof}
    We prove \eqref{M_Cong_ineq1}, noting that the proof of \eqref{M_Cong_ineq2} is identical. Observe that,
    \begin{align}\label{def_MVT}
    B[s_n^{i-1},s_n]=B^{\prime}(\Upsilon) \text{ for some } \Upsilon\in  I[s_n^{i-1}, s_n].
    \end{align}
This implies $|\Upsilon - s_n^{i-1}|\leq |s^i_n - s_n^{i-1}|\leq \Lambda \t$ from \eqref{boundedness} which gives
\begin{align}\label{eq:B-min-B}
\left|B[s_n^{i-1},s_n] - B^\prime(s_n^{i-1}) \right| &= \left| B'(\Upsilon) - B'(s_n^{i-1})\right|\nonumber\\
&\leq \snorm{B^{\prime}}_{\rm Lip} |\Upsilon - s_n^{i-1}| \overset{\eqref{boundedness}}{\leq}\snorm{B^{\prime}}_{\rm Lip}\Lambda\tau \leq  M_0\tau.
\end{align}
For $M \geq M_0$ if $L_{B,n}^i =M\tau+B^{\prime}(s_n^{i-1}) $ then $L_{B,n}^i-B^{\prime}(\Upsilon)\geq (M-M_0)\t$. Moreover, if $L_{B,n}^i =2M\tau$ then $B^{\prime}(s_n^{i-1})\leq M\tau$ which means that $ B^{\prime}(\Upsilon)\leq  B^{\prime}(s_n^{i-1})+M_0 \tau\leq (M+M_0)\tau$, giving  $L_{B,n}^i-B^{\prime}(\Upsilon)\geq (M-M_0)\t$. Hence, for $M> M_0$ one has 
\begin{align}\label{inq1}
 L_{B,n}^i-B[s_n^{i-1},s_n]\geq(M-M_0)\tau >0.
\end{align}
Using similar arguments, if $L_{B,n}^i =M\tau+B^{\prime}(s_n^{i-1}) $ and $M> M_0$,  then 
\begin{align*}
  L_{B,n}^i- B^{\prime}(\Upsilon)\overset{\eqref{eq:B-min-B}}{\leq} M\tau + M_0\tau\leq 2M\tau. 
\end{align*}
If $L_{B,n}^i=2M\tau$, then $L_{B,n}^i- B^{\prime}(\Upsilon)\leq 2M\tau$.
Combining this with \eqref{inq1} gives \eqref{M_Cong_ineq1}.
\end{proof}
\begin{lemma} \label{3.7}
Under the assumption of \Cref{theo:MS}, the coefficient functions in \eqref{eq:G-coeff} satisfy 
\[G^i_1\geq 2M\t\geq G^i_3, \text{ and } G^i_2\geq \epsilon M\t.\]
Moreover, with $\ell:=\min\{\inf b',\inf B'\}$, one has $G^i_1\geq 2\ell$.
\end{lemma}
 \begin{proof}
With $u^*\in (0,\vs)$ given in \ref{ass:Bphi}, we have the following cases:\\
\underline{If $s_n^{i-1},s_n< u^*$}:
\begin{figure}[h]
    \centering
\includegraphics[width=0.45\textwidth]{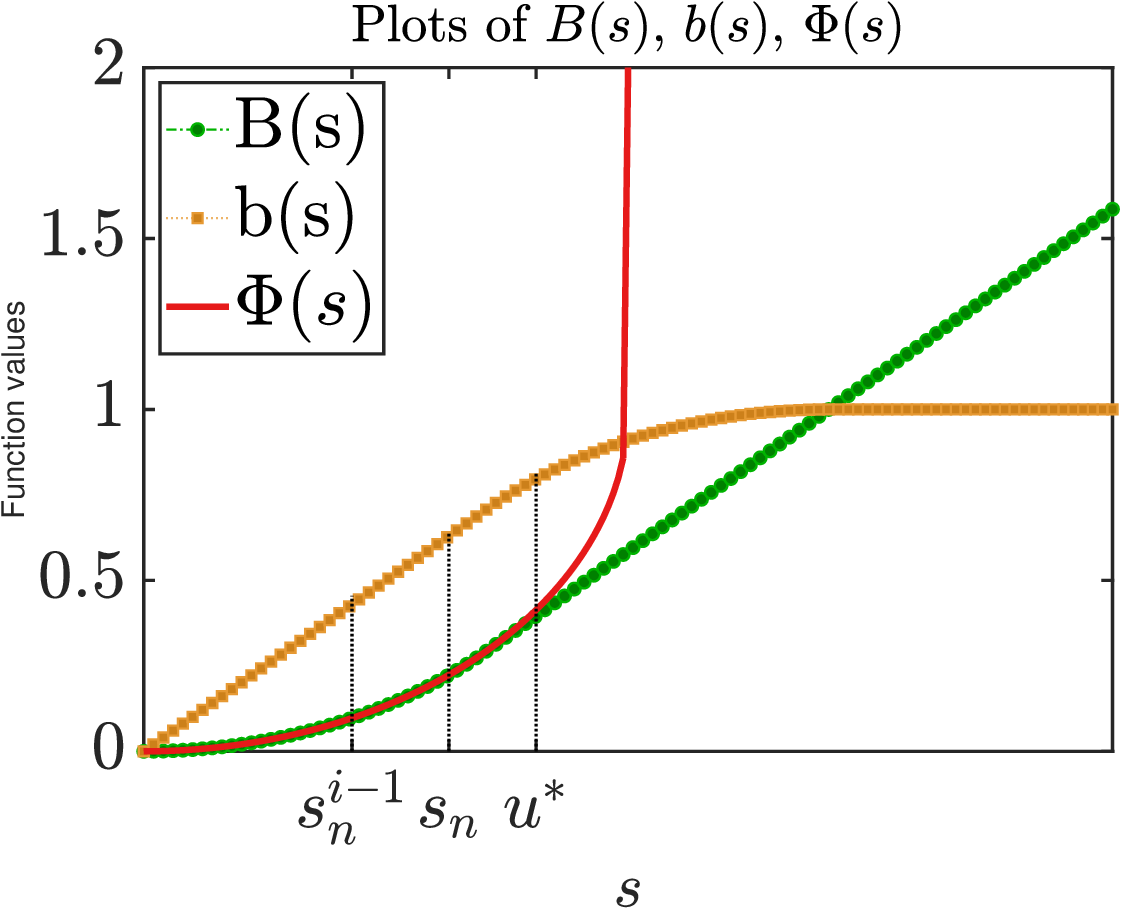}
\caption{The case $s^{i-1}_n,\,s_n<u^*$ in the proof of \Cref{M_ineq_LB}.}
    \label{fig:case1}
\end{figure}
In this case, the construction of $b$ in \Cref{lemma:decomspoition} gives $b'(s^{i-1}_n)=1$, also see \Cref{fig:case1}. Then, $L^i_{b,n}=\min(\max(1+M\t, 2M\t), 1+ \epsilon)=1+\epsilon$ since $\epsilon <(M-M_0)\t<M\t$. Moreover, $b[s_n^{i-1},s_n]=1$. Using this and the definition of $\ell$, one obtains
\begin{subequations}\label{eq:G_MS1}
    \begin{align}
        G_1^i&= b[s_n^{i-1},s_n]L^{i}_{B,n}+L^{i}_{b,n} B[s_n^{i-1},s_n]=L^{i}_{B,n}+(1+\epsilon)\, B[s_n^{i-1},s_n]\nonumber\\
        &\overset{\eqref{Mimp}}\geq \text{max}(2\ell+M\tau,2M\tau+\ell)\geq 2\max(\ell, M\t),\\
        G_2^i&= 2 L^{i}_{B,n} L^{i}_{b,n}-\big(b[s_n^{i-1},s_n]L^{i}_{B,n}+L^{i}_{b,n}  B[s_n^{i-1},s_n]\big)=(1+\epsilon)\left( L_{B,n}^i-\left(B[s_n^{i-1},s_n]\right)\right) +\epsilon  L_{B,n}^i\nonumber\\
&\overset{\eqref{Mimp}, \eqref{M_Cong_ineq1}}\geq (1+\epsilon)(M-M_0)\tau+\epsilon L^{i}_{B,n}\geq \epsilon (2M\tau).
    \end{align}
However, to estimate an upper bound for $G^i_3$ we need to further consider two cases, namely $B'(s_n^{i-1})< M\t$, and $B'(s_n^{i-1})\geq M\t$. In the former,  $L^i_{B,n}=2M\t$. Taking $\epsilon < \min \{1-\t L_F, (M - M_0 - L_F)\tau\}$ gets
\begin{align}
            G_3^i&= b[s_n^{i-1},s_n]L^{i}_{B,n}+L^{i}_{b,n} B[s_n^{i-1},s_n] -(2-\t L_F) b[s_n^{i-1},s_n]\,B[s_n^{i-1},s_n]\nonumber\\
        &= \bigg(\tau L_F-(1-\epsilon)\bigg) B[s_n^{i-1},s_n] + L^{i}_{B,n} \leq 2M\tau.  \label{eq:G3i_MS1}
\end{align}
On the other hand, if 
$B'(s_n^{i-1})\geq M\t$, then $L^i_{B,n}=B'(s_n^{i-1})+M\t$. Since $M>M_0 + L_F$, taking $\epsilon< (M-M_0-L_F)\t$ gives
\begin{align}
            G_3^i&= \bigg(\tau L_F-(1-\epsilon)\bigg) B[s_n^{i-1},s_n] + L^{i}_{B,n}= B'(s^{i-1}_n)-B[s_n^{i-1},s_n] + M\t + (\t L_F+\epsilon)B[s_n^{i-1},s_n]
            \nonumber\\
            &\overset{\eqref{eq:B-min-B}}\leq (M_0 + M)\t \overset{\eqref{eq:def_errors_c}}+ (\t L_F +\epsilon)\leq 2M\t,\label{eq:G3i_MS2}    
            \end{align}
    \end{subequations}
which proves \Cref{3.7} when $s^{i-1}_n,\, s_n<u^*$. \\
\underline{If $s_n^{i-1}, s_n > u^*$}: The proof follows similar arguments, with $B[s_n^{i-1},s_n]=1$ and $L^{i}_{B,n}=1+\epsilon$.\\ 
\underline{If $s_n^{i-1} < u^*<s_n$}:
\begin{figure}[h]
\centering
\includegraphics[width=0.45\textwidth]{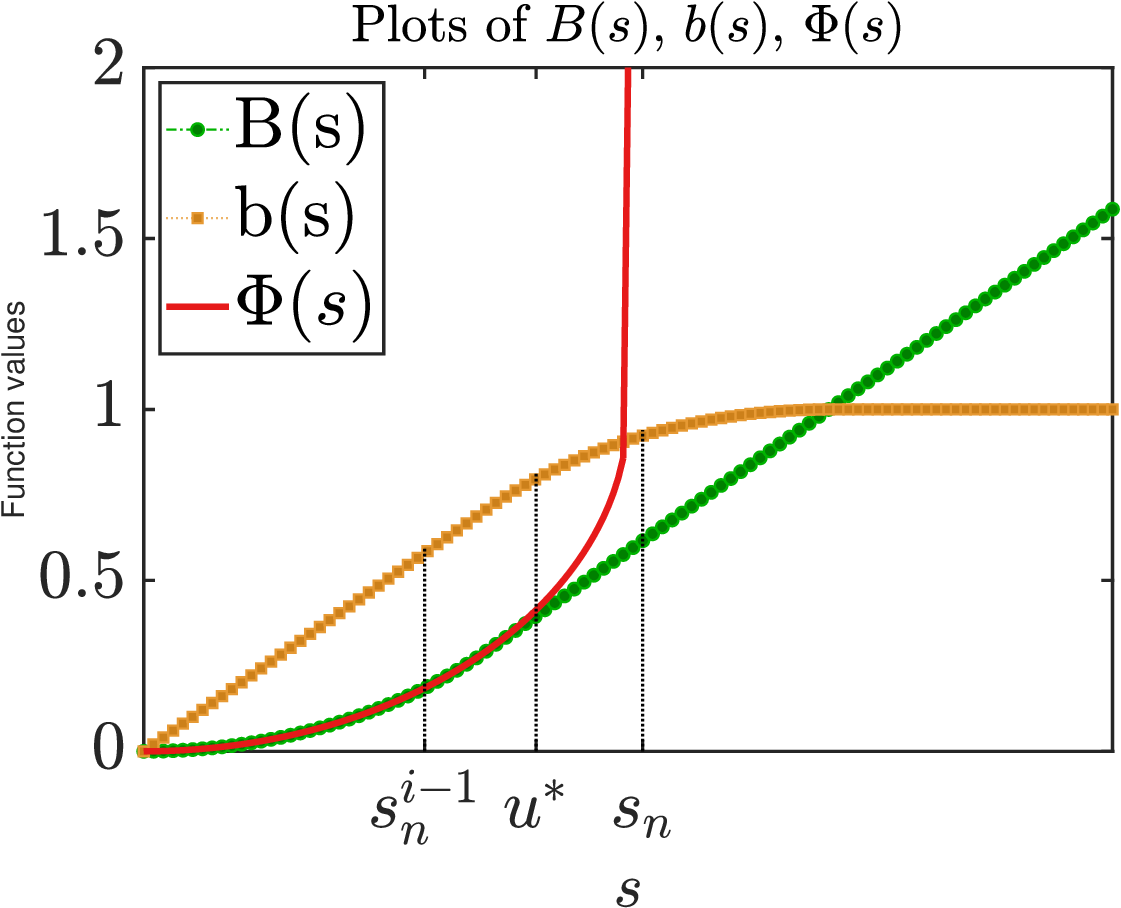}
\caption{The case $s^{i-1}_n<u^*<s_n$ in the proof of \Cref{M_ineq_LB}.}
\label{fig:case3}
\end{figure}
By \eqref{boundedness} one has
    $u^*-\Lambda \tau\leq s^{i-1}_n<u^*<s_n\leq u^*+\Lambda \tau.$
The construction of $b$ and $B$ in \Cref{lemma:decomspoition} (also see \Cref{fig:case3}) gives $B'(s_n)=1$ and $b'(s^{i-1}_n)=1$. Moreover,
\begin{align*}
    |B'(s^{i-1}_n)-B'(s_n)|\leq \snorm{B^{\prime}}_{\rm Lip} |s^{i-1}_n-s_n| \overset{\eqref{boundedness}}{\leq}\snorm{B^{\prime}}_{\rm Lip}\Lambda\tau \leq M_0\tau, 
\end{align*}
since $M_0 := \Lambda \max\{\snorm{b'}_{\rm Lip},\snorm{B'}_{\rm Lip}\}$. 
The inequality above along with $B'(s_n)=1$ give the following bounds, 
\begin{equation}\label{eq:Mobserves}
    1 - M_0 \tau \leq B^\prime(s_n^{i-1}), \  b[s^{i-1}_n,s_n], \ B[s^{i-1}_n,s_n] \leq 1 .
\end{equation}
Since $0\leq\epsilon<(M-M_0-L_F)\t$, one has that $b'(s_n^{i-1}) + M\tau \geq 1 + (M - M_0) \tau \geq 1 + \epsilon$, so $L^i_{b,n} = 1+\epsilon$ and, analogously,  $L^i_{B,n} = 1+\epsilon$. This gives 
\begin{subequations}\label{eq:G_MS2}
    \begin{align}
        G_1^i&= (1+\epsilon)\left(b[s_n^{i-1},s_n]+ B[s_n^{i-1},s_n]\right) 
        \geq 2\max\{\ell, M\tau\}.
\end{align}
In the last inequality, $G_1^i \geq 2 \ell$ follows from the definition of $\ell$, since $\epsilon \geq 0$. Further, by \eqref{eq:Mobserves}, $G_1^i \geq 2(1 + \epsilon)(1 - M_0 \t)$ and, if $\t \leq 1/(M + M_0)$, one gets that $G_1^i \geq 2 M \t$. We estimate $G_2^i$, $G_3^i$ as
        \begin{align}
        G_2^i&= (1+\epsilon)\left(2(1+\epsilon)-\left(b[s_n^{i-1},s_n]+B[s_n^{i-1},s_n]\right)\right)\geq 2\epsilon (1+\epsilon)\geq 2\epsilon M \tau, \\[1em]
        G_3^i&=b[s_n^{i-1},s_n](L^i_{B,n}-B[s_n^{i-1},s_n])+B[s_n^{i-1},s_n](L^i_{b,n}-b[s_n^{i-1},s_n]) + \t L_F b[s_n^{i-1},s_n] B[s_n^{i-1},s_n]\nonumber\\
        &\overset{\eqref{eq:Mobserves}}\leq 2 (1+\epsilon -(1-M_0\t)) + \t L_F 
        < 2M\t.
             \label{eq:G3i_MS}
     \end{align}
\end{subequations}
 In the last inequality, we used the inequaltities  $M > M_0 + L_F$ and $\epsilon \leq (M-M_0-L_F)\t$. \\
 \underline{If $s_n < u^*<s_n^{i-1}$}: This case is completely analogous to the one before. 
\\
With this, one can take $C=2M\tau$ and $\xi = 2\epsilon M \tau$ in \Cref{lemma:criteria} to obtain the convergence of MS in the doubly degenerate case, as stated in the first part of \Cref{theo:MS}. 

We continue the proof of \Cref{theo:MS} and consider the single degenerate case,  when $\inf B'=\ell_B>0$. The proof is similar to \Cref{sec:proofLscheme}. In this case, $0<\ell_B\leq B[s^{i-1}_n,s_n]\leq 1$.
Then, from \eqref{eq:b1} multiplying with $e^i_s$, we obtain by rearranging
\begin{align}
     e^i_w\, e^i_s&\overset{\eqref{eq:def_errors_c}}=(B[s^{i-1}_n,s_n]\, e^{i-1}_s +L^{i}_{B,n} (e^i_s-e^{i-1}_s))e^i_s\nonumber\\
&\overset{\eqref{eq:equality}}=\frac{L^{i}_{B,n}+ B[s^{i-1}_n,s_n]} {2}|e^i_s|^2 -\frac{L^{i}_{B,n}-B[s^{i-1}_n,s_n]}{2}|e^{i-1}_s|^2 +\frac{L^{i}_{B,n}-B[s^{i-1}_n,s_n]}{2}|e^i_s-e^{i-1}_s|^2.\nonumber
 \end{align}
We estimate the right hand side using Young's inequality \eqref{young} and the inequalities $0\leq L_{B,n}^i-B[s_n^{i-1},s_n]\leq 2M\tau\leq L^i_{B,n}$ proven in \Cref{M_ineq_LB}, which gives
\begin{align*}
    \left(2M\t +\ell_B\right)\|e^i_s\|^2 -2M\t\|e^{i-1}_s\|^2 \leq \frac{1}{2\ell_B}\|e^i_w\|^2 + \frac{\ell_B}{2}\|e^{i}_s\|^2.
\end{align*}
Employing the Poincare inequality \eqref{eq:P} gives 
\begin{align}
    \left(2M\t +\frac{\ell_B}{2}\right)\|e^i_s\|^2 -2M\t\|e^{i-1}_s\|^2 \leq \frac{1}{2\ell_B}\|e^i_w\|^2\overset{\eqref{eq:P}}\leq \frac{C^2_\Om h_\Om^2}{2\ell_B}\|\del e^i_w\|^2.
\end{align}
After multiplication by $\tau \ell_B/(C_\Omega^2 h_\Omega^2)$ and adding the result to \eqref{eq:main-ineq}, using $C=2M\t$ in \Cref{lemma:criteria} gives,
\begin{align}
    &\left[2M\t \left(1 + \frac{\t\ell_B}{C^2_\Om h_\Om^2}\right) + \frac{\t \ell_B^2}{2 C^2_\Om h_\Om^2}\right]\|e^i_s\|^2+ 2\epsilon M\t  \|e^i_s-e^{i-1}_s\|^2 + \frac{\t}{2}\|\del e^i_w\|^2\nonumber \\
    &\hskip 2em \leq 2M\t \left(1 + \frac{\t\ell_B}{C^2_\Om h_\Om^2}\right)\|e^{i-1}_s\|^2
\end{align}
Since $0<\ell_B,\,\epsilon\leq 1$, this yields \eqref{eq:linear-conv-MS} with  $\vr:=C^2_\Om h_\Om^2/(4 MC^2_\Om h_\Om^2+4M\t\ell_B+\ell_B^2)$ and $\Theta:=1/(4 MC^2_\Om h_\Om^2+4M\t\ell_B+\ell_B^2)$.

To conclude the proof of \Cref{theo:MS} we consider now 
the non-degenerate case. Since $G^i_1\geq 2\ell > 0$, \eqref{eq:main-ineq} immediately gives  \eqref{eq:lin-conv-MS}. 
 \end{proof}  
\begin{remark}
Note that for the convergence proof of LS, no additional assumption is made, which makes LS more general compared to MS. Next to this, since the  $L$-factors can be taken as constants for all time steps and iterations, the operators encountered in all iterations will remain the same, which can be used to design efficient algebraic solvers. However, this generality comes with the cost of a notably slower convergence. On the contrary, the assumptions employed for MS are based on mathematical reasoning and are likely to apply in most situations. In particular, the restriction on the time-step size $\t$ is mild and is not impacted by the spatial discretization or mesh. Moreover, the contraction rate of MS is positively influenced by $\t$ under these assumptions, in the sense that the smaller $\t$ is, the closer the rate is to 0. In practical terms, MS emerges as a significantly more competitive iterative solver in comparison to LS.
\end{remark}

\section{Adaptive estimation of linearization error}\label{sec:Adaptive}
Having proved the convergence of LS and MS, we focus now on the latter and turn our attention to the choice of the parameter $M$. As shown in \cite{mitra2019modified,smeets2024robust}, the value of the parameter $M$ plays a crucial role in determining the convergence speed of the MS. A larger $M$ value guarantees unconditional convergence of the scheme, whereas, a smaller value of $M$ makes the scheme closer to the NS which converges quadratically.

In particular, below we use a posteriori error estimation to show that this scheme can achieve unconditional convergence and, in many cases, outperform Newton's method. This is inspired by \cite{mitra2023guaranteed} where a precise identification of the linearization error was construed, and by   \cite{stokke2023adaptive}, where this identification was used in designing an adaptive linearization algorithm. Here we develop an adaptive M-scheme which chooses a quasi-optimal value of $M$ to expedite convergence.

We derive a posteriori estimates for the residual and linearization errors involving the space 
\begin{equation}\label{eq:V}
    \calV:=L^2(\Om)\times H^{1}_0(\Om) . 
\end{equation}
\subsection{Residual and linearization error}
\begin{definition}[Residual]
    Let $L^{i}_{b,n}:\Om \to (0,\infty)$ be a coefficient function that is bounded from above and below by positive constants. The residual $\calR_n^{i}:\calV\to \calV^{*}$ corresponding to Problem \ref{prob:time-disc} is defined as follows. Given $(s,w)\in \calV$, $\calR_n^{i}(s, w):\calV \to \R$ takes for any pair $(\psi,\varphi)\in \calV$ the value 
\begin{align}\label{residual}
\langle \calR^{i}_n((s,w)),(\psi,\varphi) \rangle &=\left( b(s)-u_{n-1}, \varphi\right)+ \t(\del w, \del \varphi) -\t(\bm{F}(b(s)),\del \f)-\t\langle f,\varphi\rangle \nonumber\\
&\quad +(L^{i}_{b,n}(B(s)-w),\psi). 
\end{align}\label{def:residual}
 \end{definition}
\vspace{-1em}
 Observe that $\calR^{i}_n((s,w))=0$ in $\calV^*$ if and only if $s=s_n$ and $w=B(s_n)=w_n$. Following the framework developed in \cite{mitra2023guaranteed}  to find the solution of $\calR_n^{i}=0$ based on iterative linearization, we can formulate the double-splitting scheme, i.e., Problem \ref{prob:3}, alternatively as follows.
\subsection{Alternative formulation of the double-splitting scheme}
Let $s^{i-1}_n\in L^2(\Om)$ be given, and $L^i_{b,n}, L^i_{B,n}:\Om \to \R$ be coefficient functions bounded above and below by positive constants, and computed using $s^{i-1}_n$. Consider the following bilinear form $\calB_n^{i} \calV\times \calV \to \R$, 
\begin{align}\label{eq:ip}
    \calB_n^{i}((s,w),(\psi,\f)):= (L^i_{b,n} s,\f) + \t (\del w,\del \f) + \left(L^i_{b,n}\left(L^i_{B,n} s-w \right),\psi\right)
\end{align}
Observe that $\calB_n^{i}$ satisfies the coercivity condition
\begin{align}\label{eq:B-coerce}
    \calB_n^{i}((s,w),(s,w))&= (L^i_{b,n} s, w) + \t (\del w, \del w) + \left(L^i_{b,n} \left(L^i_{B,n} s - w \right), s \right) \nonumber\\[1em]
    &= \int_\Om \left( L^i_{b,n} L^i_{B,n} |s|^2 + \t |\del w|^2 \right) \geq 0.
\end{align}
Since $L^i_{b,n}, L^i_{B,n}$ are bounded away from 0, we can define an iteration-dependent norm on $\calV$,
\begin{subequations}\label{eq:norm}
\begin{align}\label{eq:norm1}
    \norm{(s,w)}_{1,i}:= \calB_n^{i}((s,w),(s,w))^{\frac{1}{2}}= \left[ \int_\Om \left( L^i_{b,n} L^i_{B,n}|s|^2 + \t |\del w|^2 \right)\right]^{\frac{1}{2}}. 
\end{align}
The corresponding dual norm for a linear operator $\ell\in \calV^*$ is 
\begin{align}\label{eq:norm-1}
    \norm{\ell}_{-1,i}:= \sup_{(\psi,\f)\in \calV} \frac{\ell((\psi,\f))}{\norm{(\psi,\f)}_{1,i}}.
\end{align}
\end{subequations}

Eliminating $u_n^i$ from equation $\eqref{eq:linearization_gen.a}$ and $\eqref{eq:linearization_gen.b}$, we can represent the iterations in terms of the bilinear form $\calB_n^{i}$ and residual $\calR^{i}_n$ with unknowns $s_n^i$ and $w_n^i$:
\begin{problem}[Alternative formulation Problem \ref{prob:3}]\label{prob:alter}
    Let $s^{0}_n=s_{n-1}\in L^2(\Om)$ be given.  For some $i\in \N$, let $s^{i-1}_n\in L^2(\Om)$ be known. Find $(s^{i}_n,w^{i}_n)\in \calV$ solving
   \begin{align}\label{eq:reformulated-w}
\calB_n^{i}((s^{i}_n-s^{i-1}_n,w^{i}_n-w^{i-1}_n),(\psi,\f))=-\langle \calR^{i}_n((s^{i-1}_n,w^{i-1}_n)),(\psi,\f) \rangle_{\calV^*\times \calV}.
    \end{align}
for all $(\psi,\f)\in \calV$,    and update
    \begin{align}\label{eq:reformulated-Update}
 u^{i}_n=b(s^{i-1}_n) +    L^{i}_{b,n}(s^{i}_n-s^{i-1}_n) \quad\in L^2(\Om).
    \end{align}
\end{problem}

\begin{lemma}[Well-posedness of Problem \ref{prob:alter} and equivalence to Problem \ref{prob:3}] Let $s^0_n\in L^2(\Om)$ be given, $L^i_{b,n}, L^i_{B,n}:\Om \to \R$ be coefficient functions bounded uniformly above and below by positive numbers with respect to $i\in \N$. Then $\{(s^i_n,u^i_n,w^i_n)\}_{i\in \N} \subset \calZ$ solving Problem \ref{prob:alter} is well-posed and also solves Problem \ref{prob:3}.\label{lemma:equivalence}
\end{lemma}

\begin{proof}
For $(s^{i-1}_n,w^{i-1}_n)\in \calV$, the right-hand side is a linear functional for all  $(\psi,\f)\in \calV$, and $\calB_n^{i}$ is a coercive bilinear form as seen in \eqref{eq:B-coerce}. Since $L^i_{b,n}, L^i_{B,n}:\Om \to \R$ are bounded above and below by positive numbers,  $\calB^i_n$ is also Lipschitz continuous: for a constant $L_{\calB,n}^i> 0$,
\begin{align}\label{eq:B-Lip}
    |\calB^i_n((s,w),(\psi,\f))|\leq L_{\calB,n}^{i} \norm{(s,w)}_{1,i} \norm{(\psi,\f)}_{1,i}.
\end{align}
Hence, by Lax-Milgram lemma, a unique $(s^{i}_n,w^{i}_n)\in \calV$ exists. Using the definitions of $\calR^{i}_n$ and $\calB_n^{i}$ in \eqref{eq:reformulated-w}, cancelling the common terms on both sides, and rearranging, it is straightforward to verify that $(s^i_n,u^i_n,w^i_n)$ solves Problem \ref{prob:3}.
\end{proof}
In \cite{mitra2023guaranteed} it was argued that \eqref{eq:reformulated-w} represents a general form that a linearization scheme must have. In fact, due to the reasons stated below, the linearization error was identified there as
\begin{align}\label{eq:lin-err}
    \linErr{i}:=  \norm{(s^{i}_n-s^{i-1}_n,w^i_n-w^{i-1}_n)}_{1,i}.
\end{align}


\begin{lemma}[Identification of linearization error] Let the residual $\calR^{i}_n:\calV\to \calV^{*}$ be as in \Cref{def:residual}, $s^{i-1}_n\in \calV$, and the norms $\norm{\cdot}_{\pm 1,i}$ be defined in \eqref{eq:ip}. Let, $(u^i_n, s^i_n, w^i_n)\in \calZ$ be obtained through solving Problem \ref{prob:alter}. Then, the linearization error $\linErr{i}$, defined in \eqref{eq:lin-err}, is equivalent to the dual norm of the residual $\norm{\calR_n(s^{i-1}_n)}_{-1,s^{i-1}_n}$, i.e., for $L_{\calB,n}^i>0$  in \eqref{eq:B-Lip},
    \[\linErr{i} \leq \norm{\calR_n^i((s^{i-1}_n,w^{i-1}_n))}_{-1,i} \leq L_{\calB,n}^i\,\linErr{i}.\]
    Consequently, $\calR^i_n((s^{i-1}_n,w^{i-1}_n))\to 0$ in $\calV^*$ if and only if $\linErr{i}\to 0$.\label{lemma:linearization-error}
\end{lemma}

\begin{proof}
    Observe that, introducing $\d s^i_n:=s^{i}_n-s^{i-1}_n$ and $\d w^i_n:=w^{i}_n-w^{i-1}_n$,
    \begin{align}
        &\norm{\calR^{i}_n((s^{i-1}_n,w^{i-1}_n))}_{-1,i}\overset{\eqref{eq:norm}}=\sup_{(\psi,\f)\in \calV} \frac{\langle \calR^{i}_n((s^{i-1}_n,w^{i-1}_n)),(\psi,\f) \rangle_{\!_{\calV^*\times \calV}}}{\norm{(\psi,\f)}_{1,i}}\nonumber \\
        &\qquad\overset{\eqref{eq:reformulated-w}}=\sup_{(\psi,\f)\in \calV} \frac{-\calB_n^i((\d s^i_n,\d w^i_n),(\psi,\f))}{\norm{(\psi,\f)}_{1,i}}\overset{\eqref{eq:norm}}\geq \norm{(\d s^i_n,\d w^i_n)}_{1,i}\overset{\eqref{eq:lin-err}}=\linErr{i}.
    \end{align}
On the other hand, continuing the first line of the above relation,
\begin{align}
    \norm{\calR^{i}_n((s^{i-1}_n,w^{i-1}_n))}_{-1,i}\overset{\eqref{eq:B-Lip}}\leq L_{\calB,n}^i \norm{(\d s^i_n,\d w^i_n)}_{1,i}=L_{\calB,n}^i \linErr{i}.
\end{align}
Since $L^i_{b,n}$ and $L^i_{B,n}$ are uniformly bounded with respect to $i$, $\norm{\calR^{i}_n}_{-1,i}$ is uniformly equivalent  to $\norm{\calR^{i}_n}_{\calV^*}$. Hence, $\linErr{i}\to 0$ if and only if $\calR^{i}_n\to 0$ in $\calV^*$. 
\end{proof}

\subsection{A posteriori estimates of the linearization error of the M-scheme}
In this section, we derive a posteriori estimate for the linearization error. We will denote the linearization error of the $i^{\rm th}$ iteration of the MS corresponding to an $M>0$ by $ \linErr{i,M}$. Then the following holds,

\begin{theorem}[A posteriori estimate of the linearization error for a M-scheme step]\label{theo:a-posteriori}
    For $i>1$, let $\{(s^j_n,u_n^j,w_n^j)\}_{j=1}^{i-1}\subset \calZ$ be the solution to Problem \ref{prob:3} for some choice of $L^{j}_{b,n},\, L^{j}_{B,n}:\Om\to \R^+$ functions. Let $(s^i_n,u^i_n, w^i_n)\in \calZ$ be obtained through solving Problem \ref{prob:alter} with $L^i_{b,n}, \, L^i_{B,n}$ determined by the M-scheme \eqref{Mimp} with a fixed $\epsilon>0$ and a particular $M>0$. Let $\linErr{i, M}$ denote the linearization error defined in \eqref{eq:lin-err} corresponding to this choice of $M$. Introduce the estimators
    \begin{align}\label{eq:def_eta1}
      \eta_{\!_{\,{\rm lin},n,\pm}}^{{i},M}:&=  \Bigg(\left \| \left(\tfrac{L^{i}_{b,n}}{L^{i}_{B,n}} \right)^{\frac{1}{2}}(w_n^{i-1}-B(s^{i-1}_n)) \pm \left(\tfrac{L^{i}_{B,n}}{L^{i}_{b,n}} \right)^{\frac{1}{2}}(u^{i-1}_n- b(s^{i-1}_n))\right\|^2\nonumber\\
      &\qquad +\t\left\|\bm{F}(b(s^{i-1}_n))-\bm{F}(b(s^{i-2}_n))\right\|^2 \Bigg)^{\frac{1}{2}}.
    \end{align}
Then, one has    
 \begin{align}\label{eq:a-posteriori}
    \max\left(0, \frac{1}{2}\left(\eta_{\!_{\,{\rm lin},n,+}}^{i,M}-\eta_{\!_{\,{\rm lin},n,-}}^{i,M}\right)\right)\leq \linErr{i,M}\leq  \linEst{i,M}:=\frac{1}{2}\left(\eta_{\!_{\,{\rm lin},n,+}}^{i,M}+\eta_{\!_{\,{\rm lin},n,-}}^{i,M}\right).
\end{align}  
    \end{theorem}
\begin{remark}[Linearization estimator $\linEst{i,M}$ and the lower bound on error] Observe that, without needing to compute $(s^i_n,u^i_n,w^i_n)$, \eqref{eq:a-posteriori} still gives a fully computable estimate of the linearization error $\linErr{i,M}$ if M-scheme with a particular $M>0$ value is used in the $i^{\rm th}$ iteration. Hence,  $\linEst{i,M}$ can be used to choose the optimal value of $M>0$ which minimizes the linearization error. On the other hand, \eqref{eq:a-posteriori} also provides a lower bound on the linearization error $\linErr{i,M}$. However, the positivity of this lower bound cannot be guaranteed. 
\end{remark}
  
\begin{proof}[Proof of \Cref{theo:a-posteriori}]
We use again the shorthand $\delta s_n^{i} = s_n^{i} - s_n^{i-1}$ and $\delta w_n^{i} = w_n^{i} - w_n^{i-1}$. Subtracting equations \eqref{eq:linearization_gen.a} for iterations $({i+1})$ and $i$ and inserting the test function $\f=\delta w_n^{i}$, one has
\begin{align}
    \left(\delta u_n^{i}, \delta w_n^{i}\right)+\t \|\del \d w^{i}_n\|^2 =\t (\bm{F}(b(s_n^{i-1})) - \bm{F}(b(s_n^{i-2})),\del \delta w_n^{i}).
\end{align}
Using above, observe from (\ref{eq:norm}) and (\ref{eq:lin-err}) that
\begin{align}
  \left(\linErr{i,M} \right)^2 &=  \int_\Om \left( {L^{i}_{B,n}}{L^{i}_{b,n}}|\delta s_n^{i}|^2 + \t |\del \delta w_n^{i}|^2 \right) \nonumber\\
  &=  \left( {L^{i}_{B,n}}\delta s_n^{i},{L^{i}_{b,n}}\delta s_n^{i}\right)- \left(\delta u_n^{i}, \delta w_n^{i}\right)-\t (\bm{F}(b(s_n^{i-1})) - \bm{F}(b(s_n^{i-2})),\del \delta w_n^{i}).\label{eq:error}
\end{align}   
Now, we expand the first two terms to obtain
\begin{align*}
&\left( {L^{i}_{B,n}}\delta s_n^{i},{L^{i}_{b,n}}\delta s_n^{i}\right)- \left( u_n^{i}-b(s^{i-1}_n)-(u^{i-1}_n-b(s^{i-1}_n)), w_n^{i}-B(s^{i-1}_n)-(w^{i-1}_n -B(s^{i-1}_n)) \right)\nonumber\\
  &\overset{\eqref{eq:linearization_gen}}=\left( {L^{i}_{B,n}}\delta s_n^{i},{L^{i}_{b,n}}\delta s_n^{i}\right)- \left( {L^{i}_{b,n}}\delta s_n^{i}-(u^{i-1}_n-b(s^{i-1}_n)), {L^{i}_{B,n}}\delta s_n^{i}-(w^{i-1}_n -B(s^{i-1}_n)) \right)\\
  &=(u^{i-1}_n-b(s^{i-1}_n), L^i_{B,n} \d s^i_n  )+ (L^i_{b,n} \d s^i_n , w^{i-1}_n -B(s^{i-1}_n))- (u^{i-1}_n-b(s^{i-1}_n), w^{i-1}_n -B(s^{i-1}_n))\\
  &=\left({L^{i}_{b,n}} (w_n^{i-1}-B(s^{i-1}_n)) + {L^{i}_{B,n}} (u^{i-1}_n- b(s^{i-1}_n)), \d s^i_n \right).
\end{align*}
Inserting this back into \eqref{eq:error} we have
\begin{align}
  \left(\linErr{i,M} \right)^2 &=\left(\left(\tfrac{L^{i}_{b,n}}{L^{i}_{B,n}} \right)^{\frac{1}{2}}(w_n^{i-1}-B(s^{i-1}_n)) + \left(\tfrac{L^{i}_{B,n}}{L^{i}_{b,n}} \right)^{\frac{1}{2}}(u^{i-1}_n- b(s^{i-1}_n)), (L^{i}_{b,n}L^{i}_{B,n})^{\frac{1}{2}} \d s^i_n \right)\nonumber\\
  &\qquad -\t (\bm{F}(b(s_n^{i-1})) - \bm{F}(b(s_n^{i-2})),\del \delta w_n^{i})- (u^{i-1}_n-b(s^{i-1}_n), w^{i-1}_n -B(s^{i-1}_n))\nonumber\\
  & \leq \eta_{\!_{\,{\rm lin},n,+}}^{i,M}\left( \|\left( {L^{i}_{b,n}}{L^{i}_{B,n}} \right)^{\frac{1}{2}} \delta s_n^{i}\|^2 + \t\|\del \delta w_n^{i}\|^2\right)^{\frac{1}{2}}- (u^{i-1}_n-b(s^{i-1}_n), w^{i-1}_n -B(s^{i-1}_n)).\nonumber
\end{align} 
In the above, the Cauchy-Schwarz inequality along with the definition of $\eta_{\!_{\,{\rm lin},n,+}}^{i,M}$ has been used. Hence, from \eqref{eq:lin-err}, we get
\begin{align}\label{eq:express}
    \left(\linErr{i,M}\right)^2&\leq \linErr{i,M} \eta_{\!_{\,{\rm lin},n,+}}^{i,M}  -  (u^{i-1}_n-b(s^{i-1}_n), w^{i-1}_n -B(s^{i-1}_n)).
\end{align}
Hence, we get that
\begin{align*}
    &4\left(\linErr{i,M}-\frac{1}{2}\eta_{\!_{\,{\rm lin},n,+}}^{i,M}\right)^2 \leq \big(\eta_{\!_{\,{\rm lin},n,+}}^{i,M}\big)^2-4\big(u^{i-1}_n-b(s^{i-1}_n), w^{i-1}_n -B(s^{i-1}_n)\big) \nonumber\\
    & = \left \| \left(\tfrac{L^{i}_{b,n}}{L^{i}_{B,n}} \right)^{\frac{1}{2}}(w_n^{i-1}-B(s^{i-1}_n)) + \left(\tfrac{L^{i}_{B,n}}{L^{i}_{b,n}} \right)^{\frac{1}{2}}(u^{i-1}_n- b(s^{i-1}_n))\right\|^2-4\big(u^{i-1}_n-b(s^{i-1}_n), w^{i-1}_n -B(s^{i-1}_n)\big)\nonumber\\
      & \qquad +\t\left\|\bm{F}(b(s^{i-1}_n))-\bm{F}(b(s^{i-2}_n))\right\|^2\nonumber\\
     &=  \left \| \left(\tfrac{L^{i}_{b,n}}{L^{i}_{B,n}} \right)^{\frac{1}{2}}(w_n^{i-1}-B(s^{i-1}_n)) - \left(\tfrac{L^{i}_{B,n}}{L^{i}_{b,n}} \right)^{\frac{1}{2}}(u^{i-1}_n- b(s^{i-1}_n))\right\|^2+\t\left\|\bm{F}(b(s^{i-1}_n))-\bm{F}(b(s^{i-2}_n))\right\|^2\\
     &\overset{\eqref{eq:def_eta1}}=[\eta_{\!_{\,{\rm lin},n,-}}^{i,M}]^2
\end{align*}
Taking the square root and rearranging, we finally get \eqref{eq:a-posteriori}.
\end{proof}

\subsection{$M$-Adaptive algorithm}
Based on the above estimate, the algorithm elaborating the flow-chart in \Cref{fig:flow-chart} reads: 
 \begin{algorithm}[H]
  \caption{$M$-Adaptive algorithm}
  \label{alg:M-Adap}
  \begin{algorithmic}
  \REQUIRE $n\geq 1$, $s_{n-1}\in L^2(\Om)$ given
\ENSURE $s^0_n=s_{n-1}$, $M = 1$, and stopping criteria $\epsilon_{\rm stop}\ll 1$ 
\FOR{$i=1, 2, \dots$}
\STATE \textbf{Solve} \eqref{eq:linearization_gen} 
\STATE \textbf{Update} error $ \linErr{i}:=  \norm{(s^{i}_n-s^{i-1}_n,w^i_n-w^{i-1}_n)}_{1,i}.$
      \IF{error $>\epsilon_{\rm stop}$ and $i>1$}\vspace{0.5em}
           \FOR{$j=-10,-9, \dots,-2$}  
           \IF{$\eta^{i+1,10^j}_{{\rm lin},n}\leq \linErr{i}$}
           \STATE break 
           \ENDIF
           \ENDFOR
        \STATE Set $M=10^j$.
\vspace{0.5em}
    \ELSIF{error $<\epsilon_{\rm stop}$}
        \STATE Set $s_n = s_n^i$ and $w_n = w_n^i$
        \STATE break
    \ENDIF
\ENDFOR
  \end{algorithmic}\label{algo:Adaptive}
\end{algorithm}

\section{Numerical results}\label{section5}

In this section, we investigate the proposed iterative schemes numerically.
For solving the linear elliptic PDEs corresponding to each iteration, we will use a two-point flux approximation finite volume scheme with rectangular grids having mesh size $h>0$. For the M-schemes, since $L^{i}_{B,n}\geq 2M\t>0$ in all of $\Om$, we solve Problem \ref{prob:alter} since in this formulation $w^i_n$ can be solved first and $s^i_{n}$ updated subsequently. On the other hand, for Newton scheme, we solve Problem \ref{prob:3} since $L^i_{B,n}=0$ occurs in a subdomain, and thus, the two formulations are no longer equivalent. 
The code is based on Matlab and is available on GitHub\footnote{Link to the GitHub repository: \href{https://github.com/ayeshajaved00/Doubly-Degenerate-Non-Linear-Advection-Diffusion-Reaction-equation}{https://github.com/ayeshajaved00/Doubly-Degenerate-Non-Linear-Advection-Diffusion-Reaction-equation}}.

We consider four different test cases with increasing complexity:
\vspace{-.5em}
\begin{enumerate}[label=(\roman*)]
\setlength{\itemsep}{0pt}
    \item The porous medium equation ($\Phi(u)=u^m$).
    \item A double degenerate toy-model where $\Phi$ is multivalued at $\vs=1$.
    \item The biofilm growth model: $\Phi'$ vanishes at $0$ and $\Phi$ becomes infinite at $\vs=1$.
    \item The Richards equation with van Genuchten parametrization for unsaturated flow through soil (double degenerate and with nonlinear advection).
\end{enumerate}
\vspace{-.5em}
The last two test cases are examples of double degenerate models in real-life applications.

We investigate the above problems in one and two space dimensions (referred to as 1D and 2D cases henceforth) with the corresponding numerical domains being $(-10,10)$, and $(-10,10)^2$ respectively.
For all four test cases, we have opted for the Barenblatt solution \cite{vazquez2007porous}
\begin{align}\label{eq:Barenblatt}
  u_{\,\!_{\rm BB}}(\bm{x},t) = (1 + t)^{-\nu} \left[\max\left(\gamma - \frac{\nu(m-1)|\bm{x}|^2}{2dm(t+1)^{2\nu/d}}, 0\right)\right]^{1/(m-1)} \;\;\text{ with }\;\; \nu = \frac{1}{(m-1 + \frac{2}{d})},
\end{align}
 at $t=0$ as our initial condition, see \Cref{fig:Barenblatt}. Here, $d$ is the space dimension, $m>1$ is a parameter, $\bm{x}$ the space variable, and $t$ time.  Since $u_{\,\!_{\rm BB}}$ is an exact solution of the porous medium equation with $\Phi(u)=u^m$, this choice allows us to verify our code. Furthermore, it also enables us to compare schemes since $u_{\,\!_{\rm BB}}$ possesses a sharp front and can be made to reach 1 by altering $\gamma$, see \Cref{fig:Barenblatt}. In the simulations, $m=6$ is used unless stated otherwise.
\begin{figure}[H]
\centering
    \begin{subfigure}[t]{0.48\textwidth}   \includegraphics[width=\textwidth]{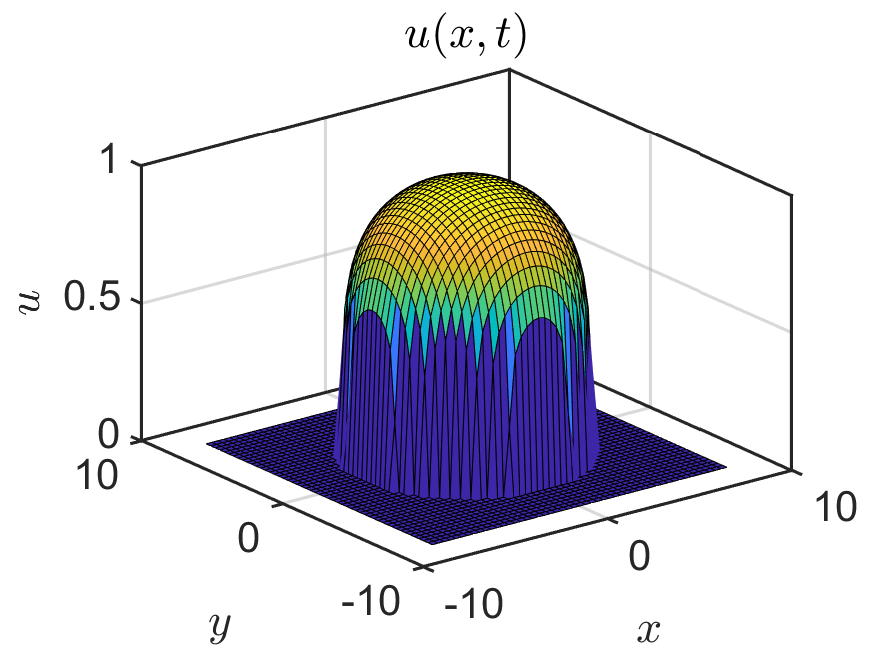}
    \end{subfigure}
    \begin{subfigure}[t]{0.48\textwidth}      \includegraphics[width=\textwidth]{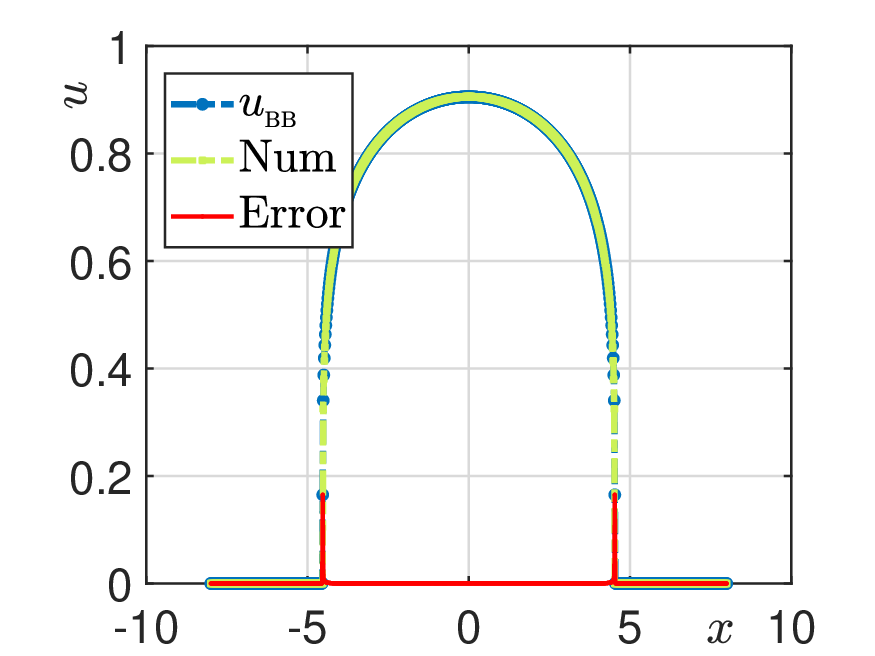}
    \end{subfigure}
\caption{\textbf{(left)} Barenblatt solution for $d=2$, $m = 6$, $\gamma=1$. \textbf{(right)} A comparison between the exact and numerical solutions for $h=0.016$, $\tau = 0.1$, $T=1.0$ at the cross-section $y=0$.}
\label{fig:Barenblatt}
\end{figure}

For all the test cases, the L-scheme was found to be much slower compared to the other schemes, although it converged in every instance. Hence, it is not represented in this section. For the fixed M-scheme, the value of $M=0.01$ is fixed unless stated otherwise. This value yields converge for all test cases considered. However, as shown in \cite{mitra2019modified,smeets2024robust}, the convergence behavior can really be improved by choosing an optimal value of $M$. The adaptive scheme solves the issue of finding the optimal $M$  by choosing it on the fly.  For a small, given tolerance \(\epsilon_{\rm stop} > 0\), we aim to ensure that the numerical scheme converges to a sufficiently accurate solution. To achieve this, we define a stopping criterion and terminate the iterations when the error measure $\linErr{i}$, introduced in \eqref{eq:lin-err}, satisfies
\begin{align}\label{eq:Stop}
   \linErr{i} := \left( \int_\Omega \left( L^i_{b,n} L^i_{B,n} |s^i_n - s^{i-1}_n|^2 \, + \tau\left| \nabla \left( w_n^i - w_n^{i-1} \right) \right|^2 \right)\,\right)^{\frac{1}{2}} \leq \epsilon_{\rm stop}.
\end{align}
To analyze the convergence behavior of the Newton scheme, M-scheme, and adaptive M-scheme with respect to discretization parameters, we present the averaged iteration counts across varying time-step sizes $\tau$ and mesh sizes $h$ for examples illustrated in Figures \ref{fig:PME_meshplots}, \ref{fig:DD_meshplot}, \ref{Biofilm_meshplot}, and \ref{Richards_meshplots}.
 Furthermore, to assess the convergence rates and order of convergence of the schemes, we consider an error $\mathcal{E}_{\mathrm{fix},n}^{i}$ similar to $\linErr{i}$ but in a fixed norm independent of $L^i_{b/B,n}$, i.e.,
 \begin{align}\label{eq:standard_L2}
     \mathcal{E}_{\mathrm{fix},n}^{i} := \left( \int_\Omega \left( |s^i_n - s^{i-1}_n|^2 \, + \tau \left| \nabla \left( w_n^i - w_n^{i-1} \right) \right|^2 \right)\right)^{\frac{1}{2}}.
\end{align}
The overall contraction rate $\a$ of the scheme at a given time-step $n\in \N$ is computed as the mean of $\alpha^i$, which are the ratios of $\mathcal{E}_{\mathrm{fix},n}^{i} $ between consecutive iterations:
\begin{equation}
\alpha^i := { \mathcal{E}_{\mathrm{fix},n}^{i}}\slash{ \mathcal{E}_{\mathrm{fix},n}^{i-1}},\qquad \forall\, i\in \N.
\label{eq:contraction}
\end{equation}
The mean is over the last three $\a^i$ values until criteria \eqref{eq:Stop} is satisfied. A smaller value of \( \alpha \) indicates faster convergence. 
The order of convergence for each scheme is defined as 
\begin{equation}
p:={\log(\alpha_i)}\slash {\log(\alpha_{i-1})},
\label{convergence_order}
\end{equation}
for the last iteration $i$ before meeting criteria \eqref{eq:Stop}.
The contraction rates and convergence orders, derived from \eqref{eq:contraction} and \eqref{convergence_order}, are shown on the left and right sides of \Cref{fig:PME_cont_order,fig:DD_cont_order,fig:Biofilm_cont_order,fig:Richars_Cont_order} respectively.
\subsection{The porous medium equation (PME)}\label{sec:PME}
First, we consider the porous medium equation
\begin{equation*}
\partial_t u = \Delta u^m \quad \text{ for }\quad  m>1
\end{equation*}
which corresponds to $\Phi(u)=u^m$, $f=0$, and $\bm{F}=\bm{0}$.
It shows degeneracy at $u = 0$ since $\Phi'$ vanishes. Observe that the Barenblatt solution $u_{\,\!_{\rm BB}}$ in \eqref{eq:Barenblatt} solves the PME exactly. Hence, it is used as a control for our code, see \Cref{fig:Barenblatt} (right).
\begin{figure}[h!]
    \centering
    \begin{subfigure}[b]{0.32\textwidth}
        \centering
        \subcaption{1D, $\tau = 10^{-1}$}
        \includegraphics[width=\textwidth]{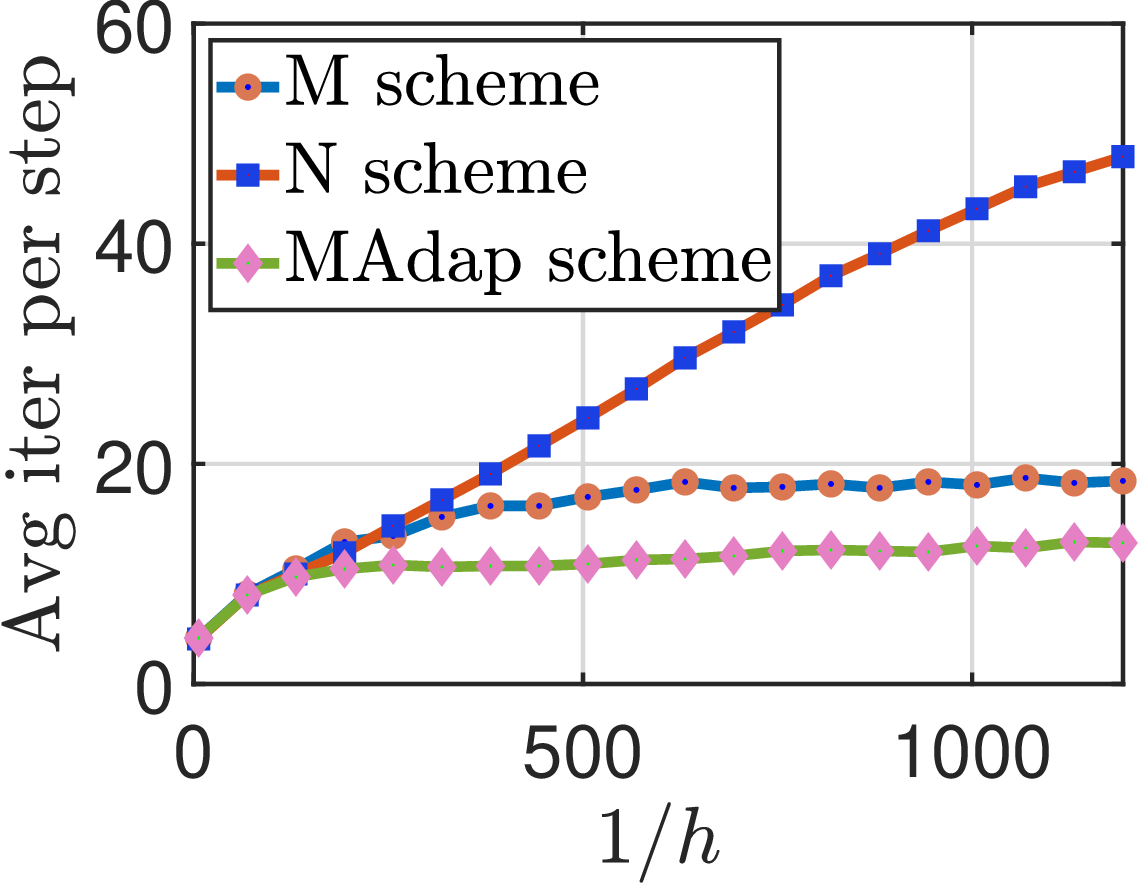}
    \end{subfigure}
    \hfill
    \begin{subfigure}[b]{0.32\textwidth}
        \centering
        \subcaption{1D, $\tau = 10^{-1.5}$}
\includegraphics[width=\textwidth]{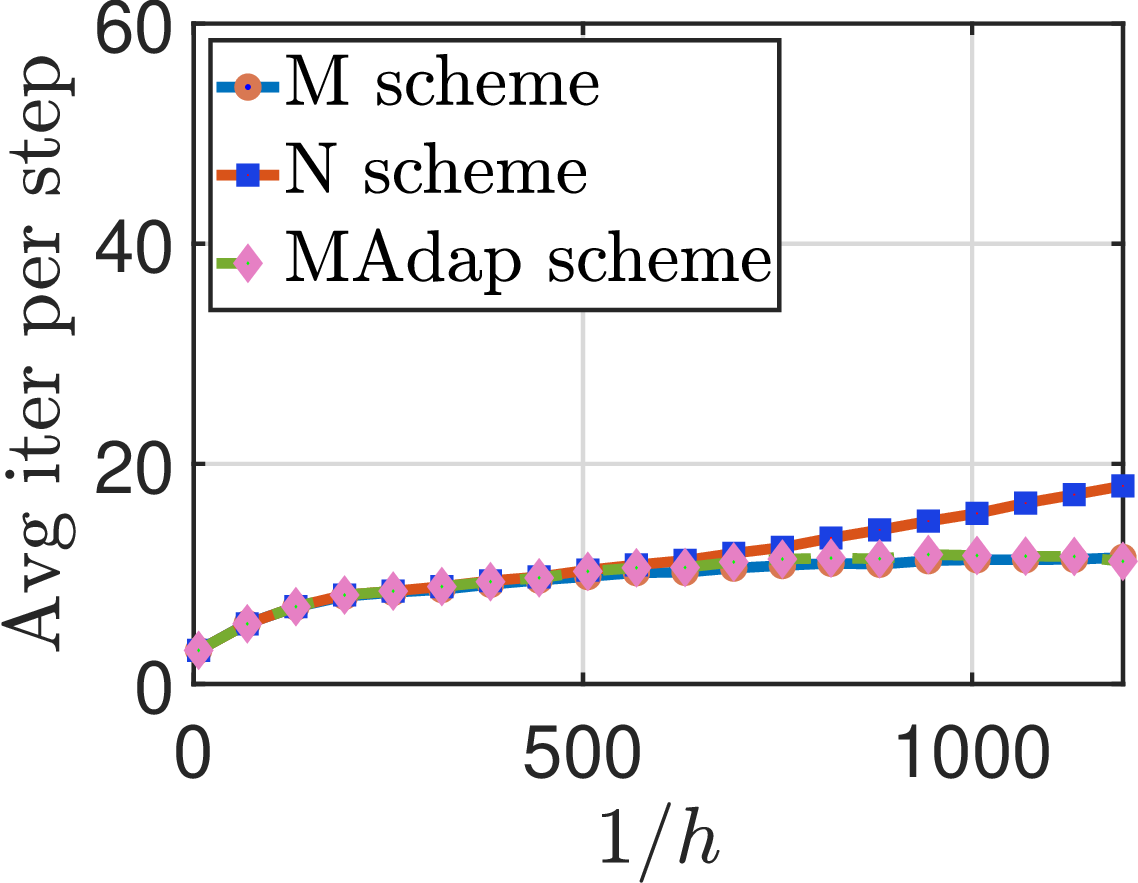}
    \end{subfigure}
    \hfill
    \begin{subfigure}[b]{0.32\textwidth}
        \centering
        \subcaption{1D, $\tau = 10^{-2}$}
        \includegraphics[width=\textwidth]{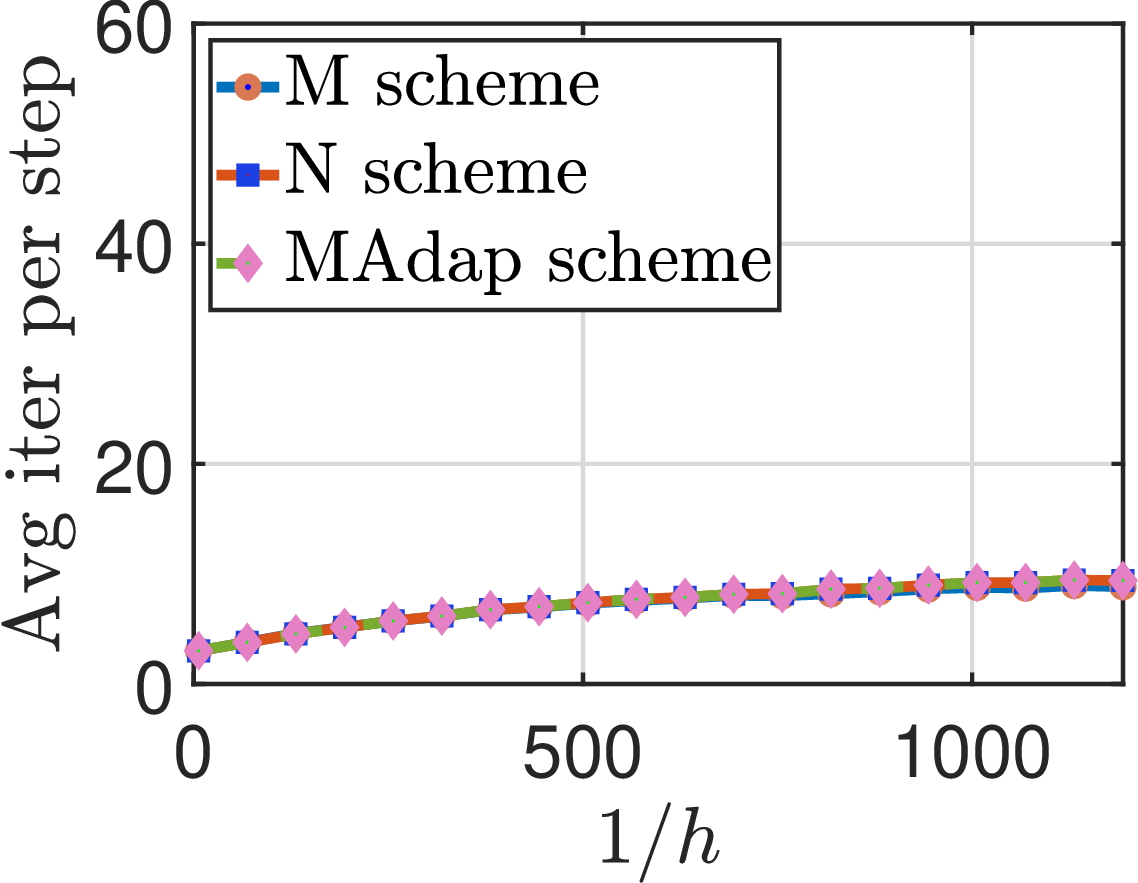}
    \end{subfigure}

    \begin{subfigure}[t]{0.32\textwidth}
        \centering
        \subcaption{2D, $\tau = 10^{-1}$}  \includegraphics[width=\textwidth]{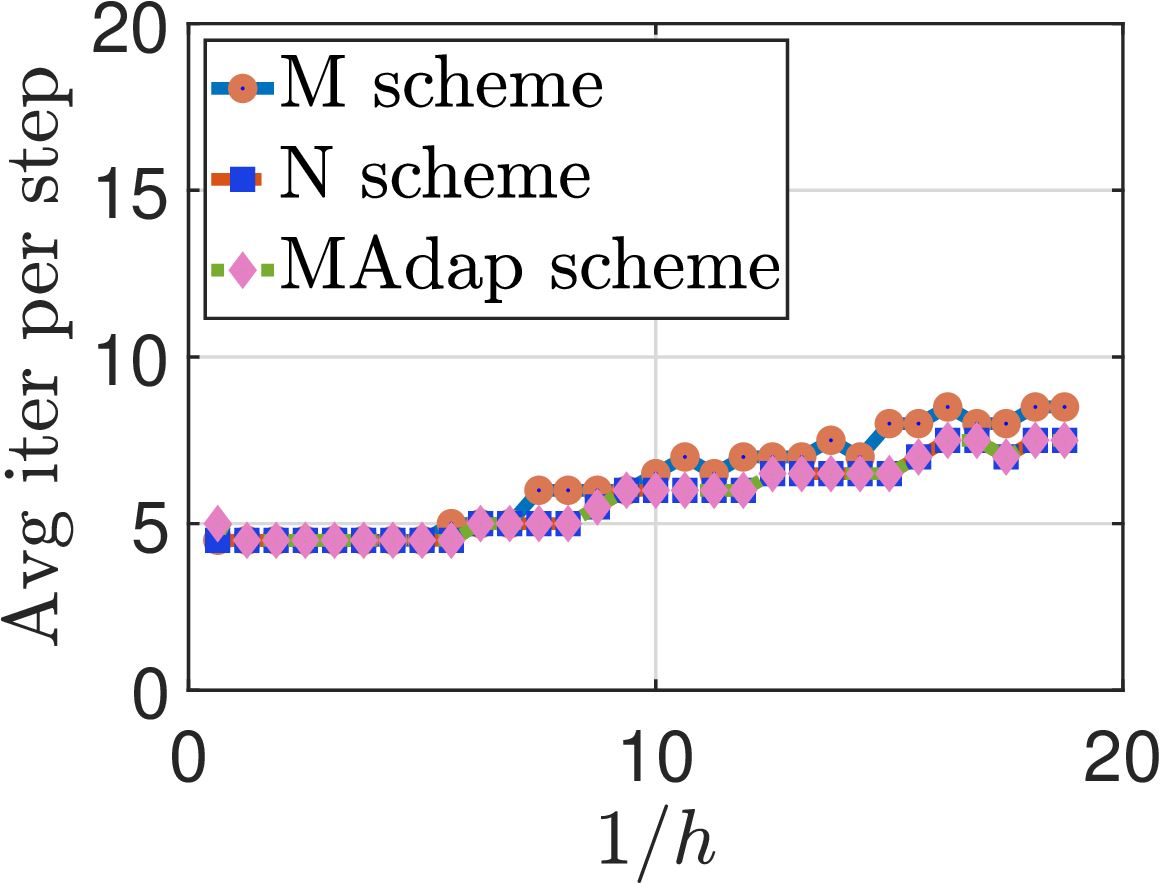}
    \end{subfigure}
    \hfill
    \begin{subfigure}[t]{0.32\textwidth}
        \centering
        \subcaption{2D, $\tau = 10^{-1.5}$}
\includegraphics[width=\textwidth]{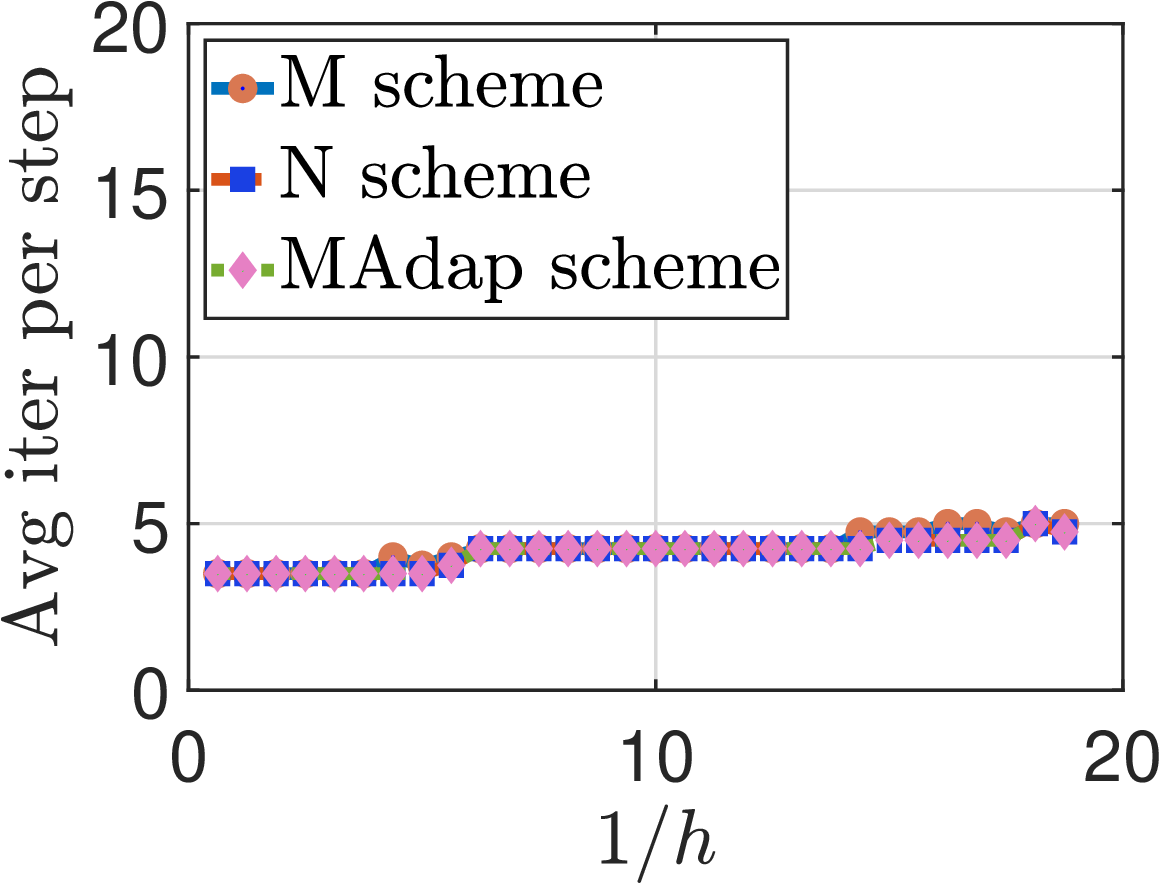}
    \end{subfigure}
    \hfill
    \begin{subfigure}[t]{0.32\textwidth}
        \centering
                \subcaption{2D, $\tau = 10^{-2}$}
        \includegraphics[width=\textwidth]{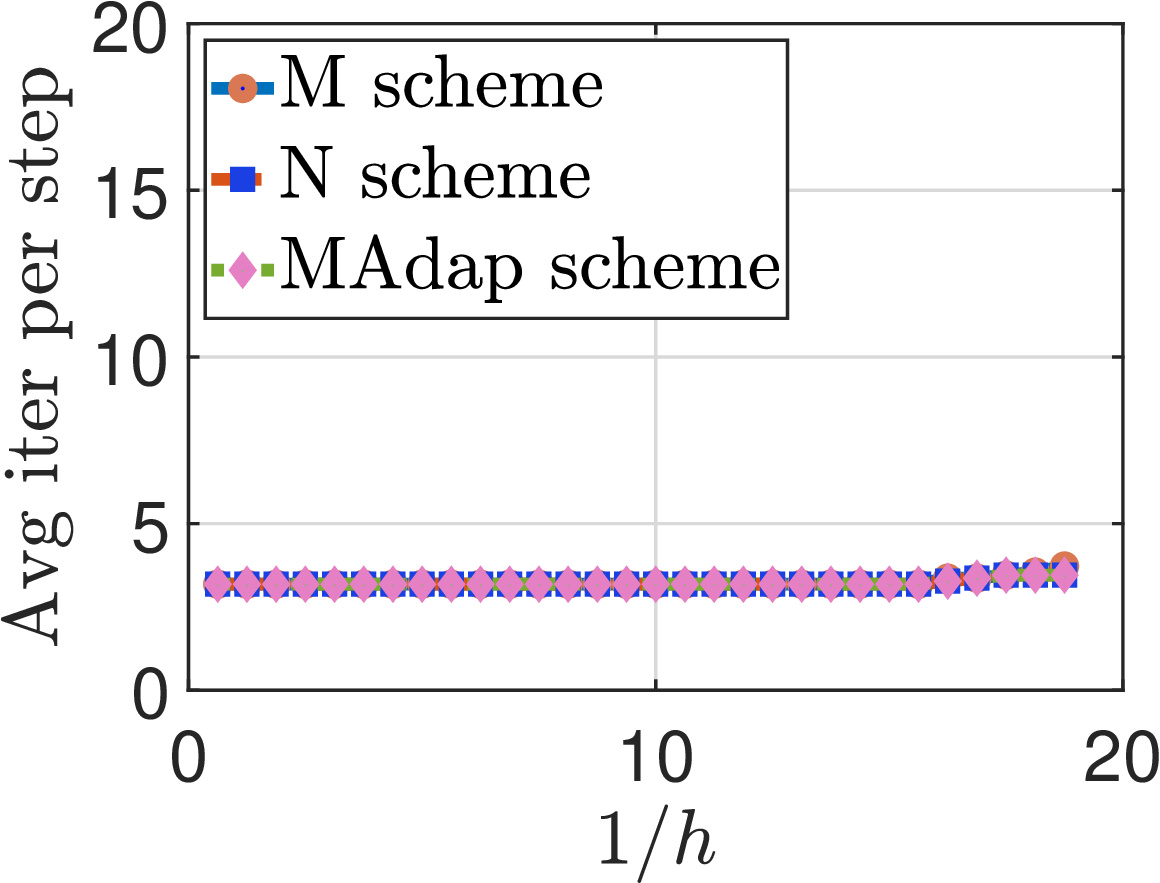}
    \end{subfigure}
        \caption{[\Cref{sec:PME}] Average iterations required per time-step for PME in 1D \textbf{(top row)} and 2D \textbf{(bottom row)} for varying mesh sizes. The stopping criterion is based on \eqref{eq:Stop}, with a tolerance of \( \epsilon_{\rm stop} = 10^{-6} \). Here, $T=1$ for 1D and $0.1$ for 2D.}
    \label{fig:PME_meshplots}
\end{figure}
The results of the performance of the Newton scheme, M-scheme, and adaptive M-scheme for 1D and 2D cases are shown in \Cref{fig:PME_meshplots}. We note that both in 1D and 2D, as $\tau$ gets smaller, the amount of iterations decreases. Moreover, the number of iterations required for the M-scheme remains consistent across different mesh sizes $(h)$. In contrast, the adaptive M-scheme exhibits superior performance in iteration count.  However, the key observation is that our proposed schemes outperform the Newton scheme in 1D for finer mesh sizes. The iterative schemes perform similarly to each other for smaller time-step sizes since the $M\t$ term becomes negligible.
\begin{figure}[h!]
\begin{subfigure}[t]{0.49\textwidth}
        \centering
  \includegraphics[width=\textwidth]{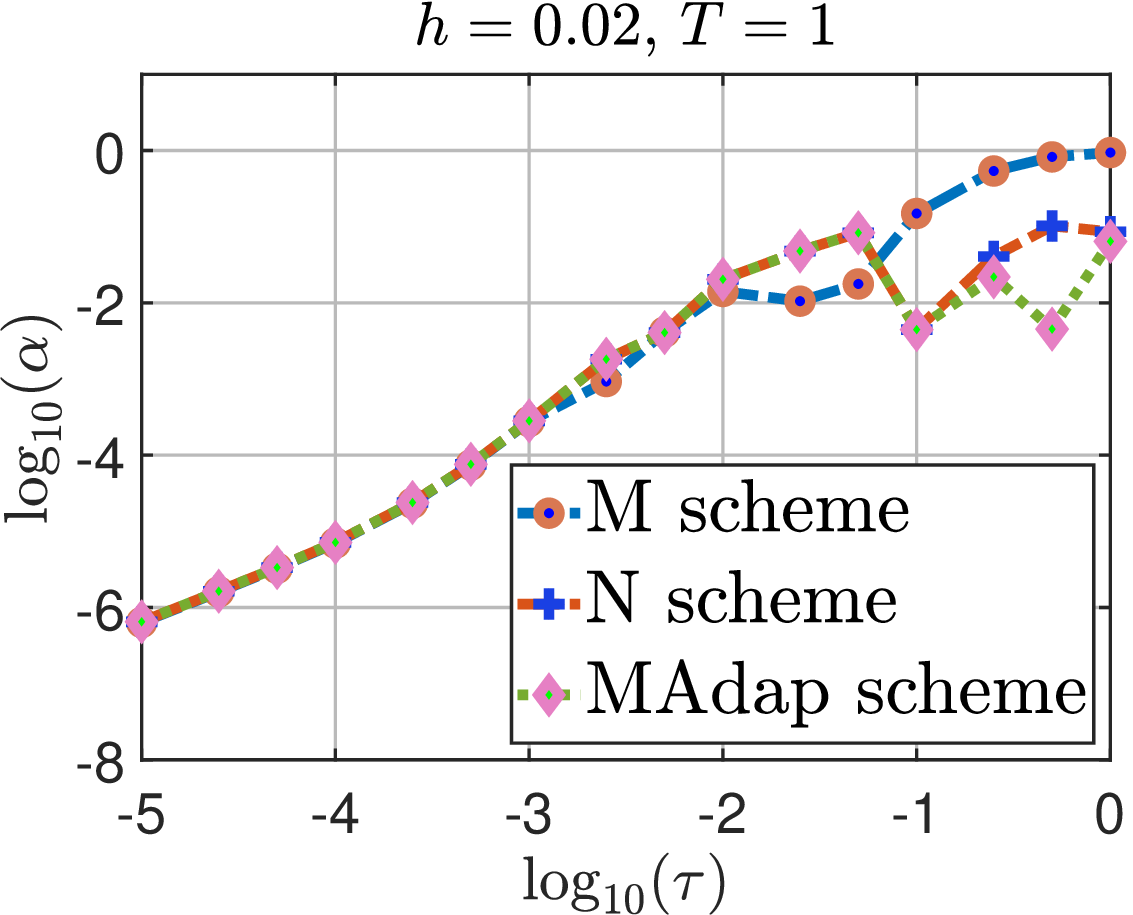}
\end{subfigure}
\hfill
\begin{subfigure}[t]{0.49\textwidth}
        \centering
  \includegraphics[width=\textwidth]{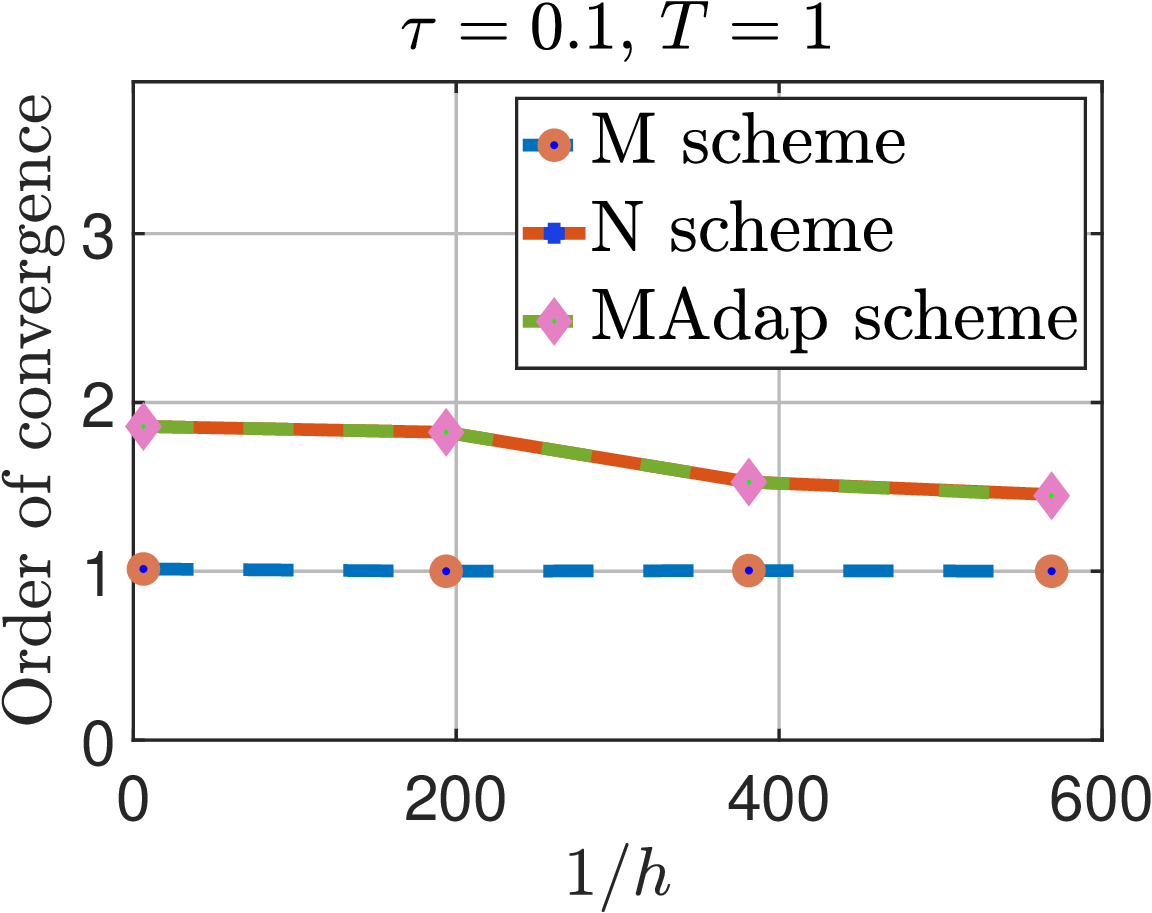}
\end{subfigure}
\caption{[\Cref{sec:PME}]\textbf{(left)} Average contraction rate $(\alpha)$ vs. time-step size $(\tau)$ for the 1D case. The stopping criterion here uses a tolerance of \( \epsilon_{\rm stop} = 10^{-10} \). \textbf{(right)} order of convergence of the iterative methods. }
    \label{fig:PME_cont_order}
    \end{figure}
    
\Cref{fig:PME_cont_order} (left) shows how the contraction rate, as defined in \eqref{eq:contraction}, varies with $\tau$ for different iterative schemes. It is observed that for small enough time-step sizes, $\a$ scales superlinearly with $\t$!! This is despite the non-degeneracy condition required for linear convergence not being satisfied. However, in 1D, there are only two points at the sharp front, and hence, at least linear scaling with $\t$ was expected. For more complicated problems, we will see this scaling being violated, see e.g. \Cref{sec:toy}. Another observation is that along with Newton, the adaptive scheme also shows quadratic convergence when the error is small, a fact validated by subsequent numerical results.  
\begin{figure}[h!]
    \centering
    \begin{subfigure}[b]{0.48\textwidth}
        \centering
         \subcaption{1D, $\tau = 0.1$}\includegraphics[width=\textwidth]{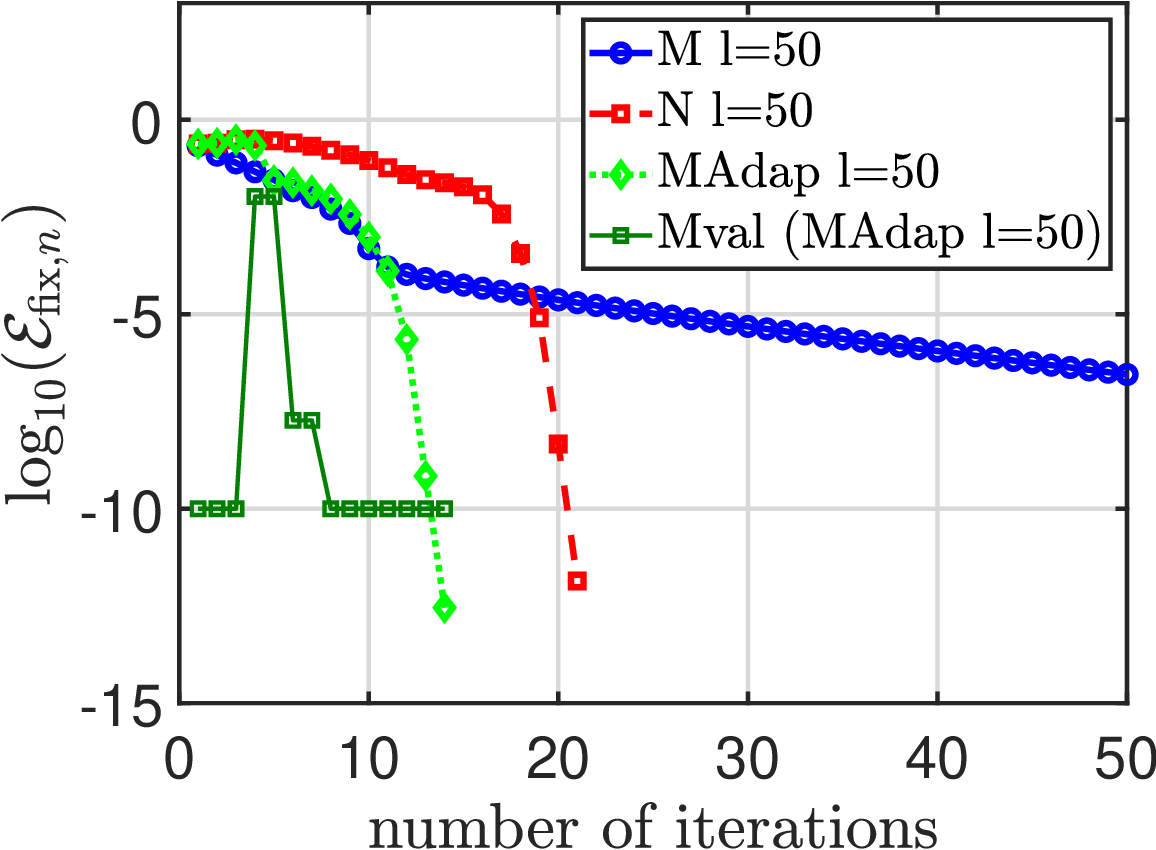}
    \end{subfigure}
    \begin{subfigure}[b]{0.48\textwidth}
        \centering
         \subcaption{1D, $\tau = 0.1$}\includegraphics[width=\textwidth]{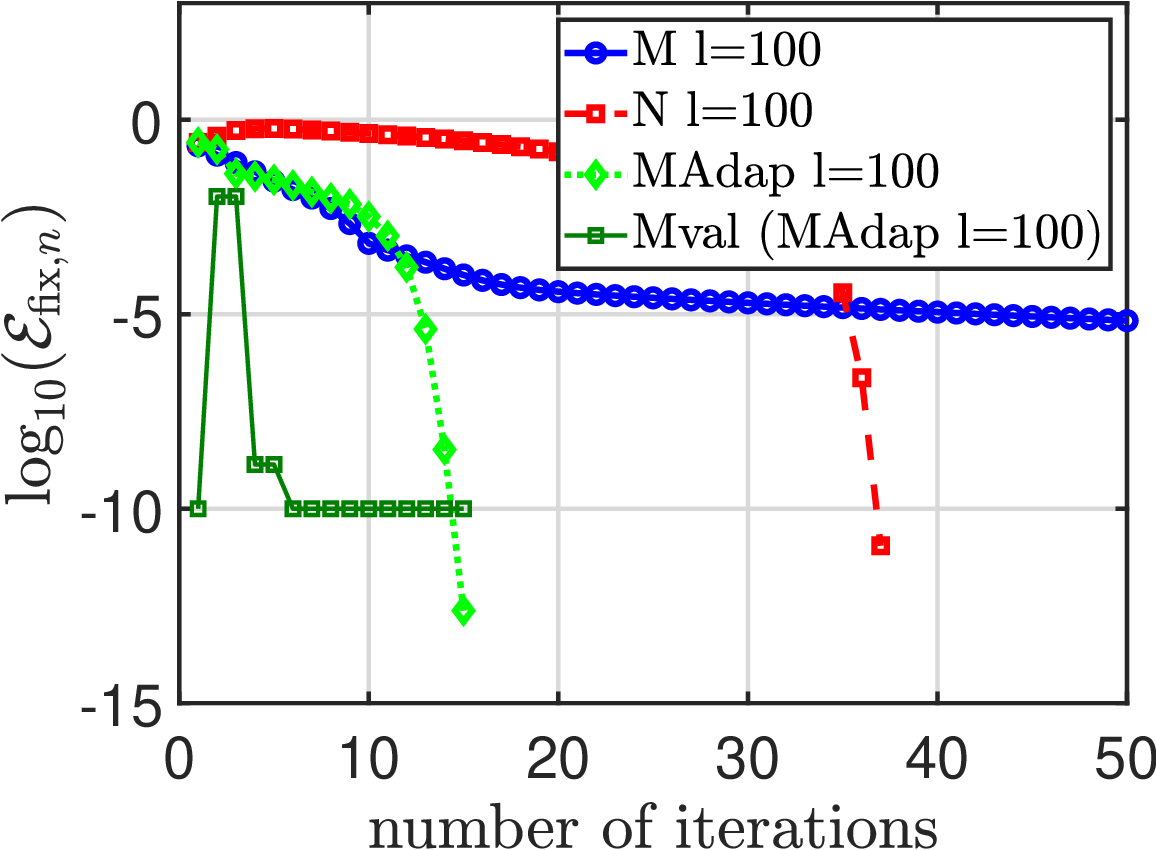}
    \end{subfigure}
    \caption{[\Cref{sec:PME}] Error $\mathcal{E}_{\mathrm{fix},n}^{i}$ vs. iteration $i$ for different iterative schemes for the first time-step ($\tau=0.1$) and mesh size $h=0.16/l$. 
    The Mval quantity shows how the $M$ varies with iteration for the adaptive scheme. }
        \label{fig:PME_ErrVsiter}
\end{figure}
\begin{figure}[h!]
    \centering
    \begin{subfigure}[b]{0.48\textwidth}
        \centering
         \includegraphics[width=\textwidth]{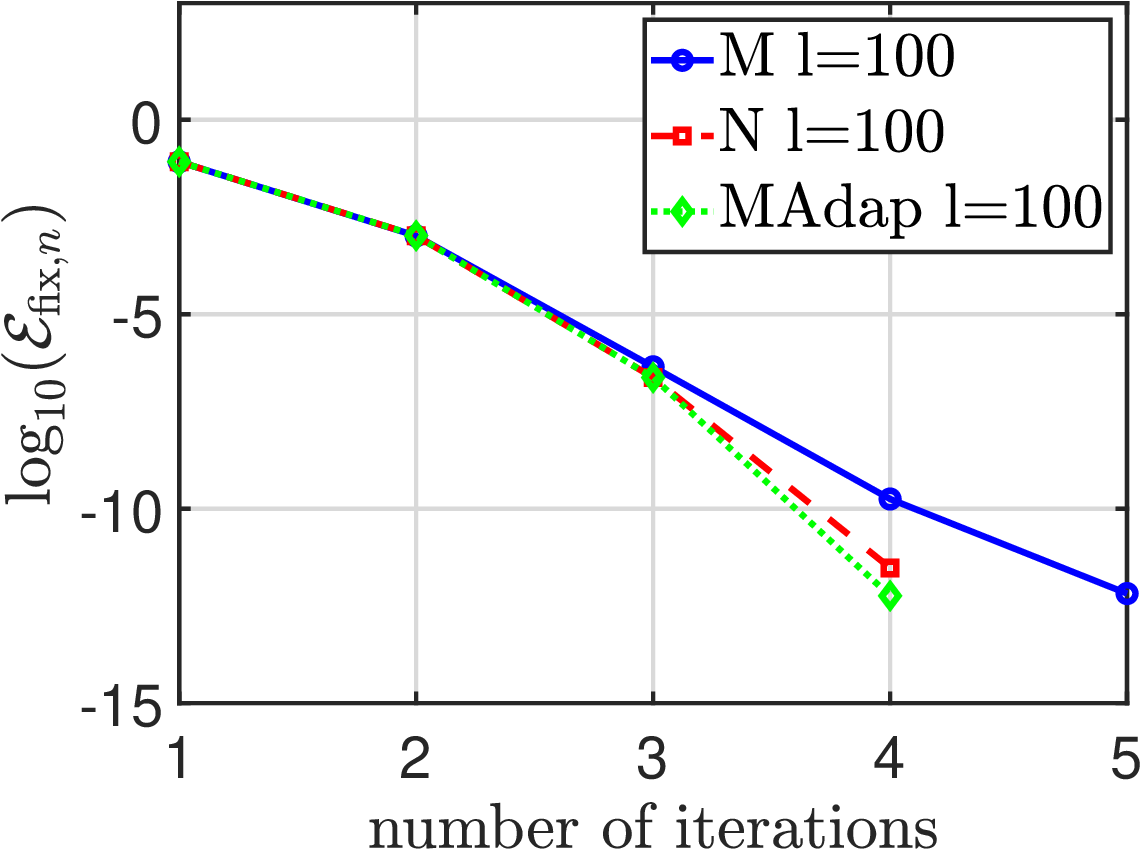}
         \subcaption{Nonlinear diffusion without degeneracy}
    \end{subfigure}
    \begin{subfigure}[b]{0.48\textwidth}
        \centering
        \includegraphics[width=\textwidth]{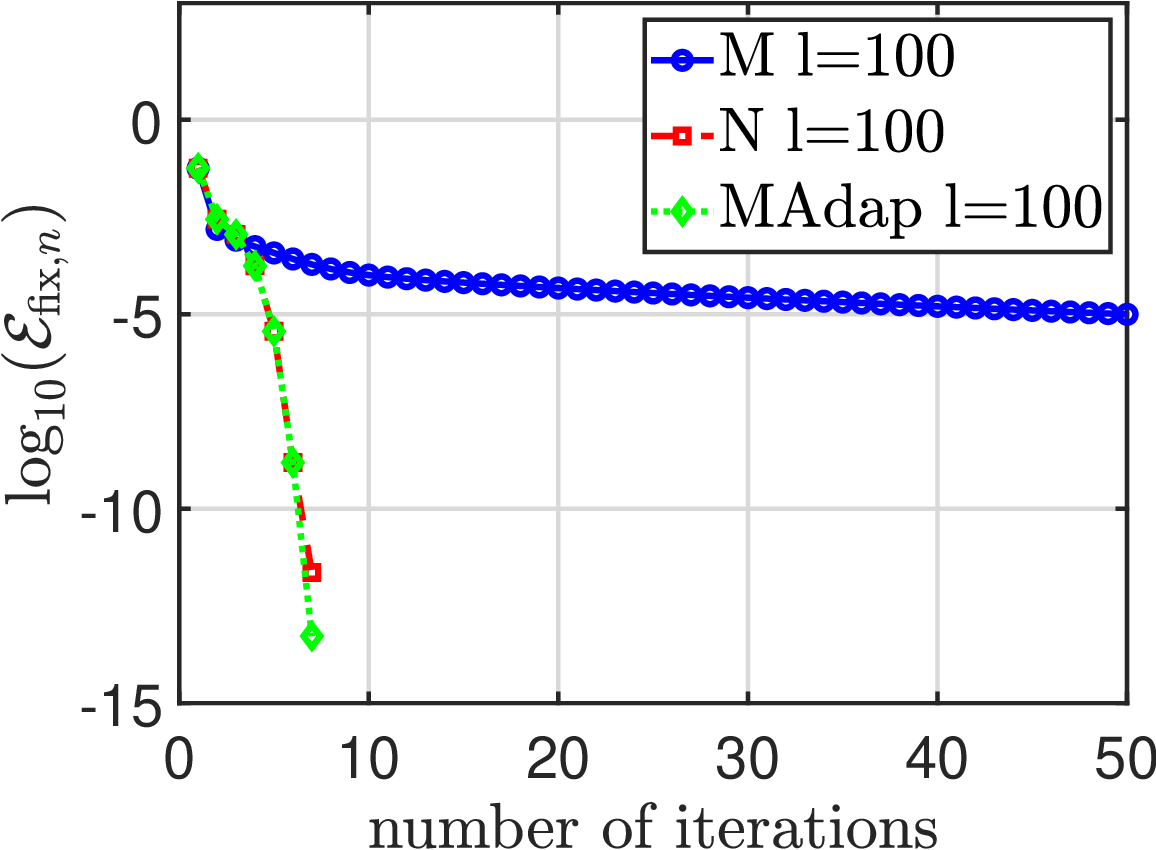}
        \subcaption{Degenerate diffusion}
    \end{subfigure}
    \caption{[\Cref{sec:PME}] Influence of diffusion function $B(s)$ on schemes behavior for the first time-step ($\tau=0.1$) and mesh size $h=0.16/l$: (left) With a nonlinear choice of $B(s)$, all schemes converge smoothly. (right) Introducing a degenerate $B(s)$ leads to plateauing of M-scheme.}
        \label{fig:PME_ErrVsiter_ModifyB(s)}
\end{figure}
\Cref{fig:PME_ErrVsiter} shows how the error $\mathcal{E}_{\mathrm{fix},n}^{i}$ decays with iterations for two mesh sizes $h$ and a given $\tau=0.1$. Newton's scheme exhibits quadratic convergence. However, as the mesh is refined, the Newton scheme requires significantly more iterations to achieve the same error level, highlighting its sensitivity to mesh refinement. 
In contrast, the fixed M-scheme asymptotically reaches a linear convergence regime when the errors are approximately of the order of the $M\t$ term, it demonstrates remarkably fast convergence in the absence of degenerate diffusion (see \Cref{fig:PME_ErrVsiter_ModifyB(s)}), reaching low error levels within just a few iterations. On the other hand, the adaptive scheme is slower than M-scheme in the beginning, but after some iterations the $M$-values become less and less, and the scheme converges quadratically. Thus, it reaches error levels below $10^{-10}$ faster, and it reaches every error level faster than Newton.

\subsection{A double degenerate toy-model}\label{sec:toy}
Now, we investigate a double degenerate toy-model where $\Phi$ becomes multivalue at $\vs=1$:
\begin{align}\label{eq:toy}
\frac{\partial u}{\partial t} = \Delta \Phi(u) +  \frac{1}{2}\, u, \quad \text{ where } \quad \Phi(u) =
\begin{cases}
0 & \text{if } u \leq 0, \\
1-\sqrt{1-u^2} & \text{if } 0 \leq u < 1,\\
 [1, \infty] & \text{if } u = 1.
\end{cases}
\end{align}
For this problem, using \eqref{eq:bBexpression}, the functions  $b,\, B$ can be expressed explicitly, i.e.,
\begin{align*}
    b(s):=\begin{cases}
        s &\text{ if } s\leq 1/\sqrt{2},\\
        \sqrt{1-(\sqrt{2}-s)^2} &\text{ if } 1/\sqrt{2}\leq s\leq \sqrt{2},\\
        1 &\text{ otherwise},
    \end{cases}
\quad
    B(s):=\begin{cases}
        0 &\text{ if } s\leq 0,\\
        1-\sqrt{(1-s^2)} &\text{ if } 0\leq s\leq 1/\sqrt{2},\\
        s+1-\sqrt{2} &\text{ otherwise}.
    \end{cases}
\end{align*}
\Cref{fig:doubledegenerate} (left) shows the functions $\Phi$, $b$, and $B$, whereas, the (right) plot shows the numerical solution for this case which has a plateau at 1. This example is designed to show the effect of the parabolic-elliptic degeneracy at $\vs=1$, and hence $\g=1.5$ is chosen in the initial condition \eqref{eq:Barenblatt}, and a reaction term of $\frac{1}{2}u$ is added to stabilize the plateau.
\begin{figure}[h!]
    \centering
    \begin{subfigure}[t]{0.45\textwidth}
        \centering
    \includegraphics[width=\textwidth]{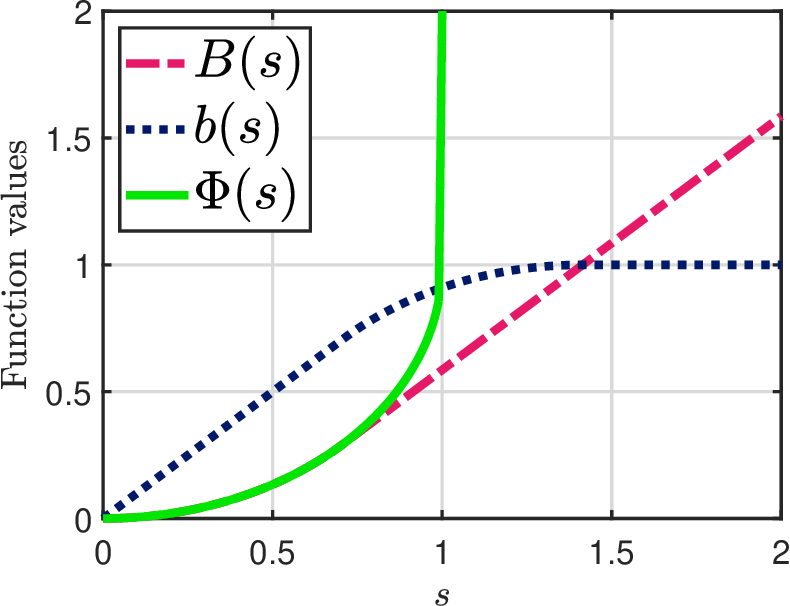}
    \end{subfigure}
    \begin{subfigure}[t]{0.5\textwidth}
        \centering
    \includegraphics[width=\textwidth]{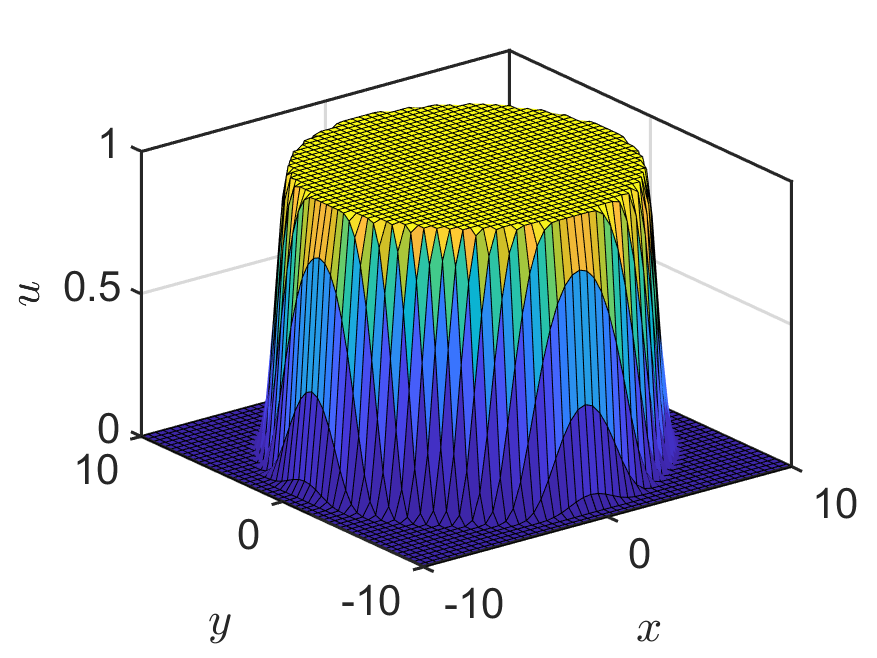}
    \end{subfigure}
    \caption{[\Cref{sec:toy}]\textbf{(left)} The functions $b$ and $B$ computed from $\Phi$ in \eqref{eq:toy}. \textbf{(right)} Numerical solution at $T = 1$ with a time-step size of $\tau = 0.1$. 
    }
    \label{fig:doubledegenerate}
\end{figure}

\begin{figure}[h!]
    \centering
    \begin{subfigure}[b]{0.32\textwidth}
        \centering
        \subcaption{1D, $\tau = 10^{-1}$}
        \includegraphics[width=\textwidth]{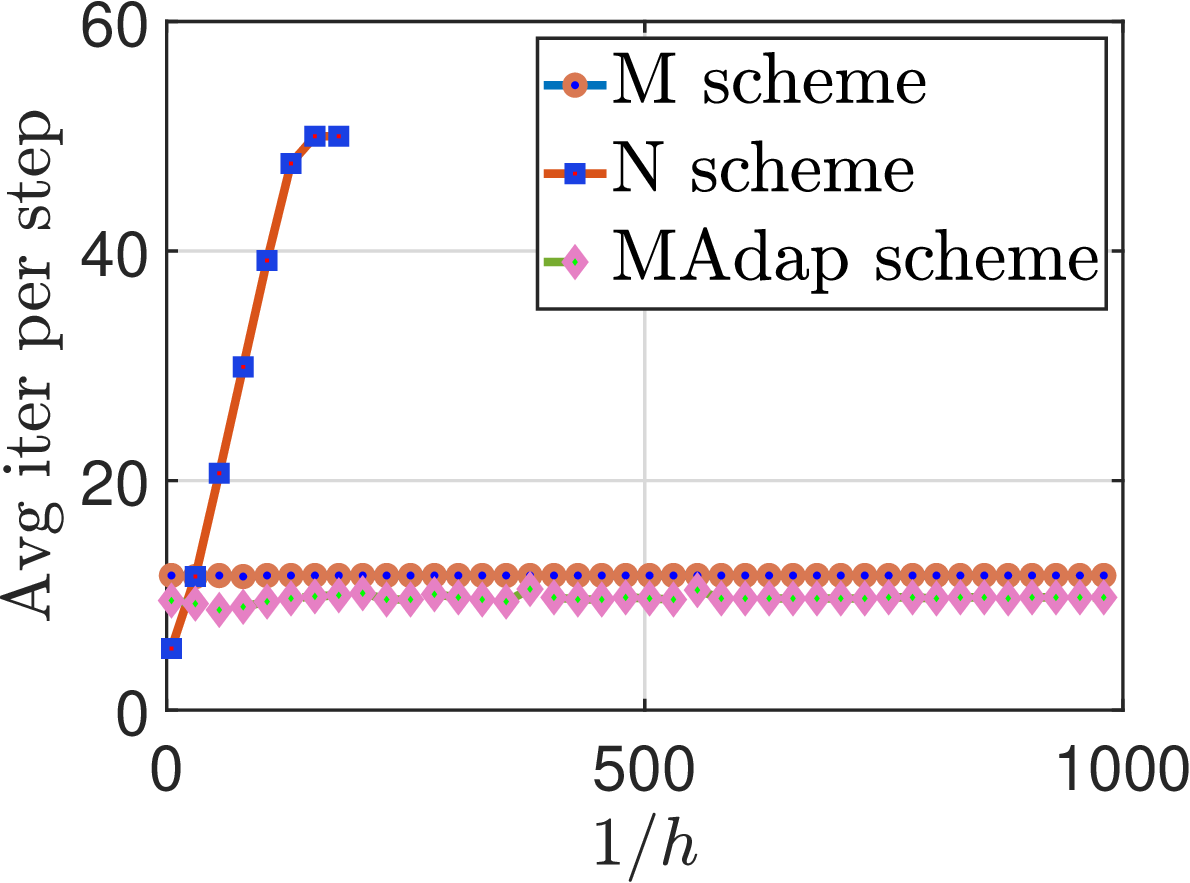}

    \end{subfigure}
    \hfill
    \begin{subfigure}[b]{0.32\textwidth}
        \centering
        \subcaption{1D, $\tau = 10^{-1.5}$}
        \includegraphics[width=\textwidth]{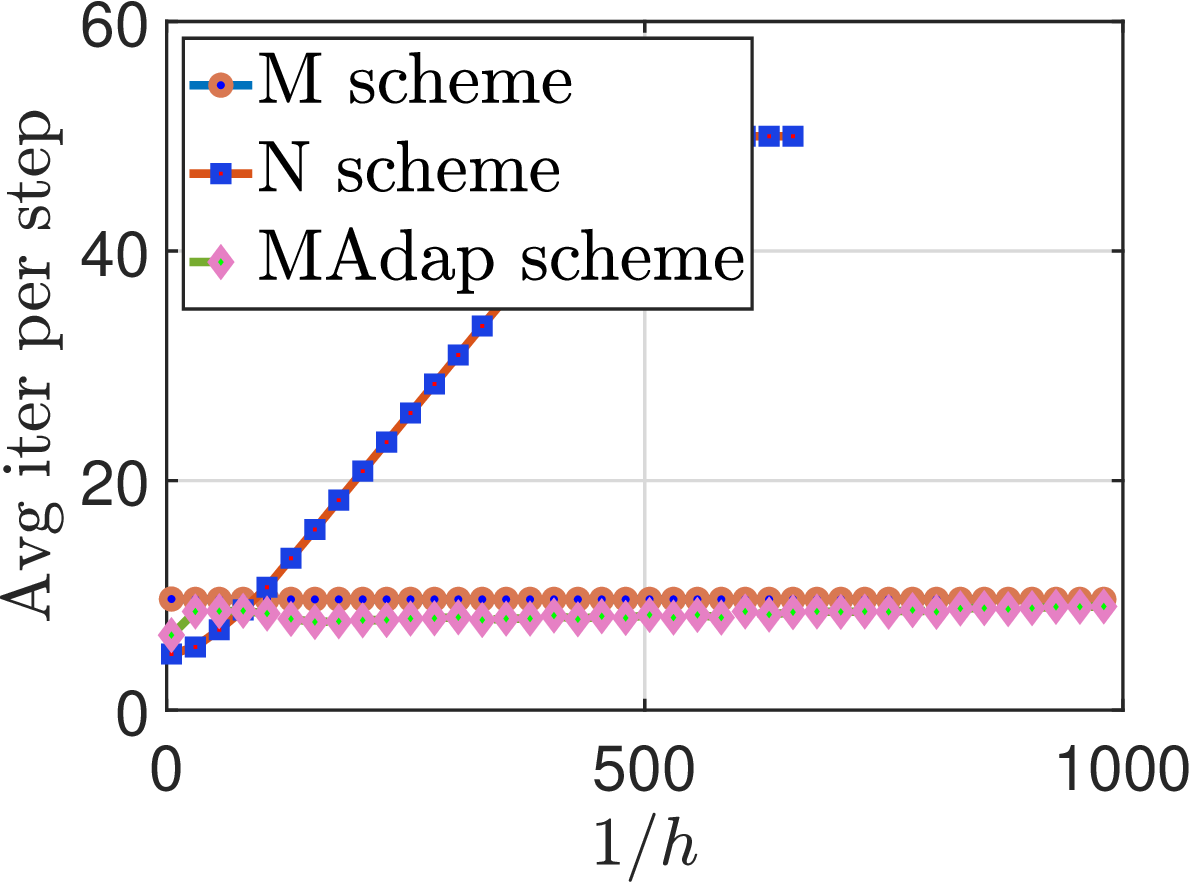}
    \end{subfigure}
    \hfill
    \begin{subfigure}[b]{0.32\textwidth}
        \centering
        \subcaption{1D, $\tau = 10^{-2}$}
        \includegraphics[width=\textwidth]{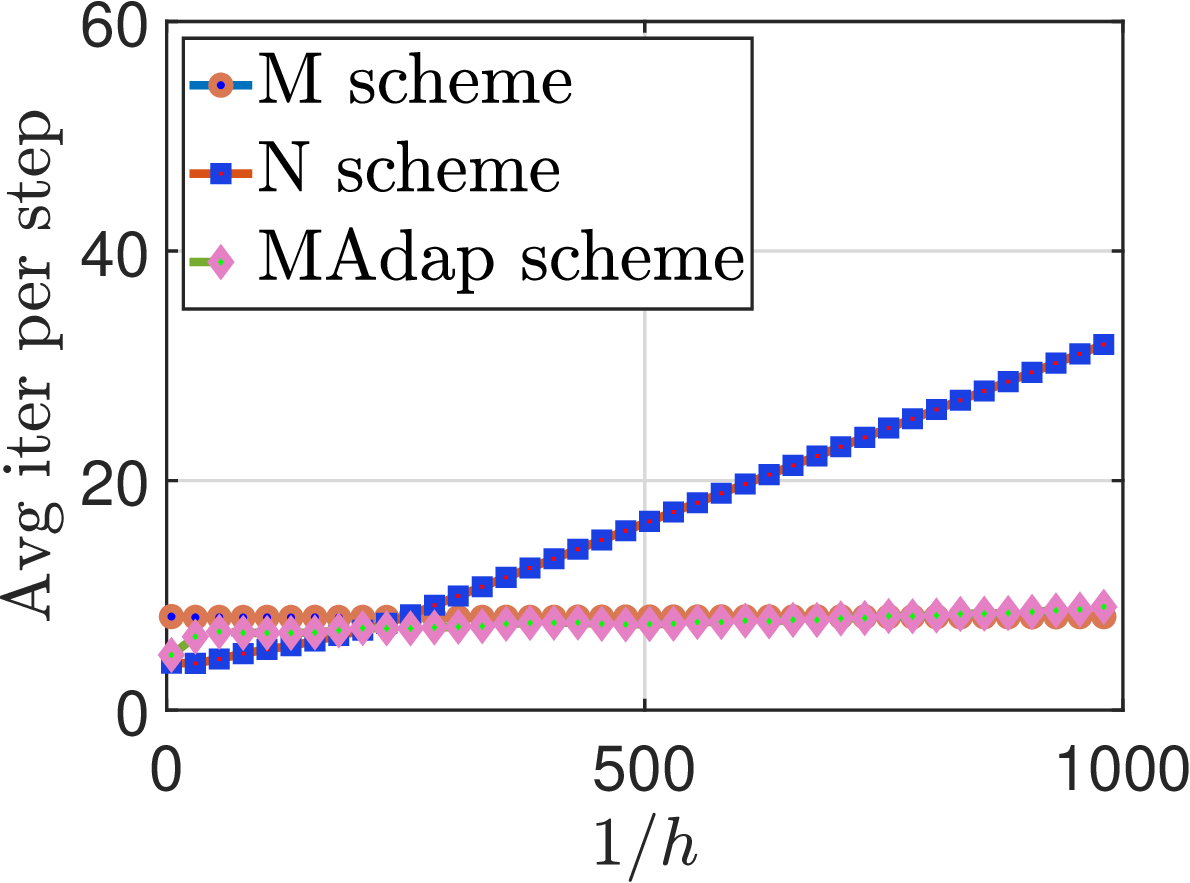}
    \end{subfigure}

    \begin{subfigure}[t]{0.32\textwidth}
        \centering
        \subcaption{2D, $\tau = 10^{-1}$}
        \includegraphics[width=\textwidth]{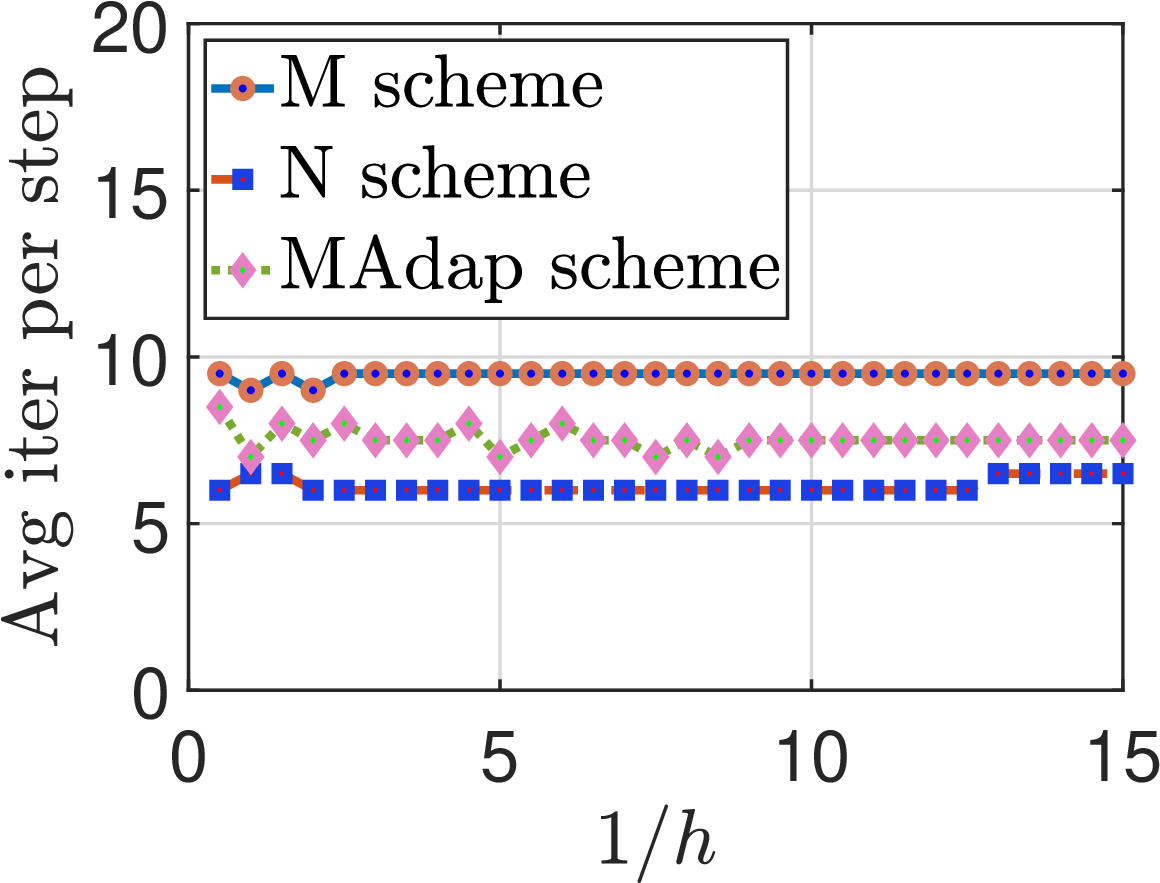}
    \end{subfigure}
    \hfill
    \begin{subfigure}[t]{0.32\textwidth}
        \centering
        \subcaption{2D, $\tau = 10^{-1.5}$}
        \includegraphics[width=\textwidth]{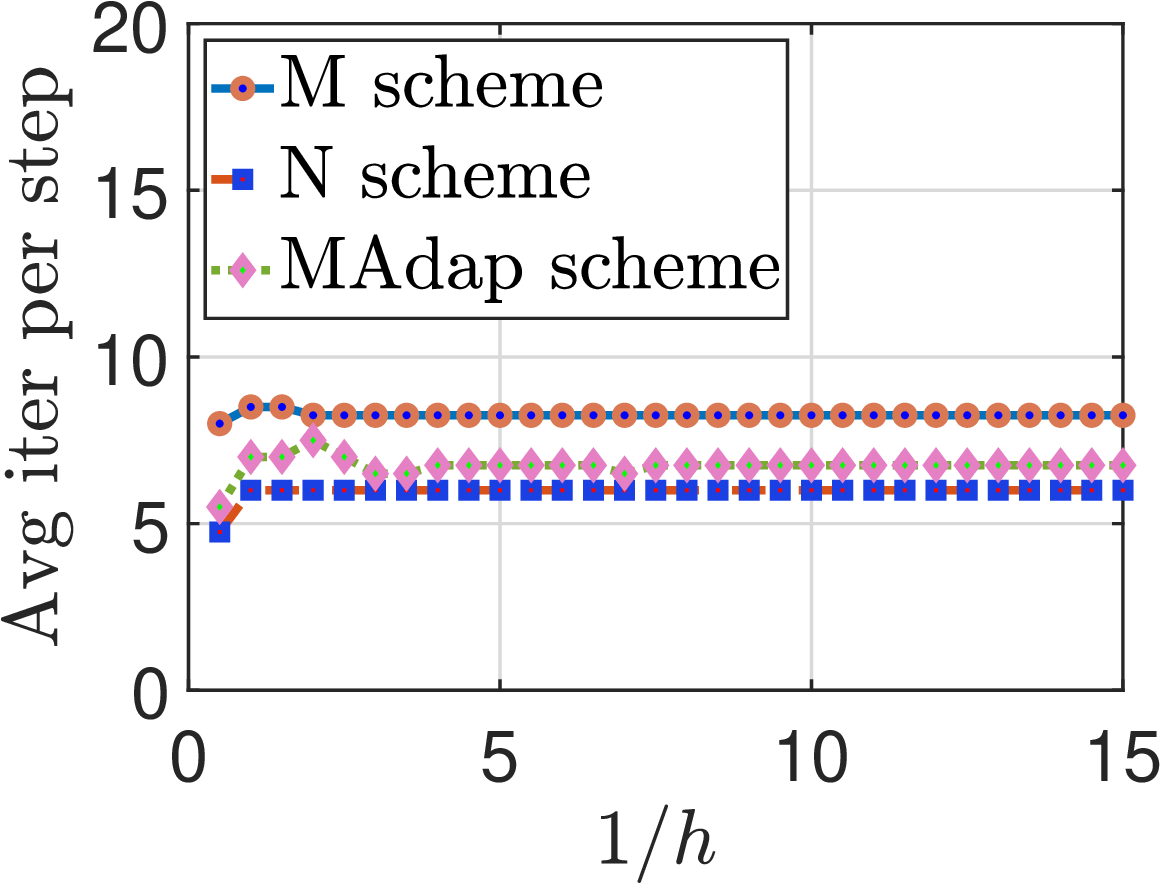}
    \end{subfigure}
    \hfill
    \begin{subfigure}[t]{0.32\textwidth}
        \centering
                \subcaption{2D, $\tau = 10^{-2}$}
        \includegraphics[width=\textwidth]{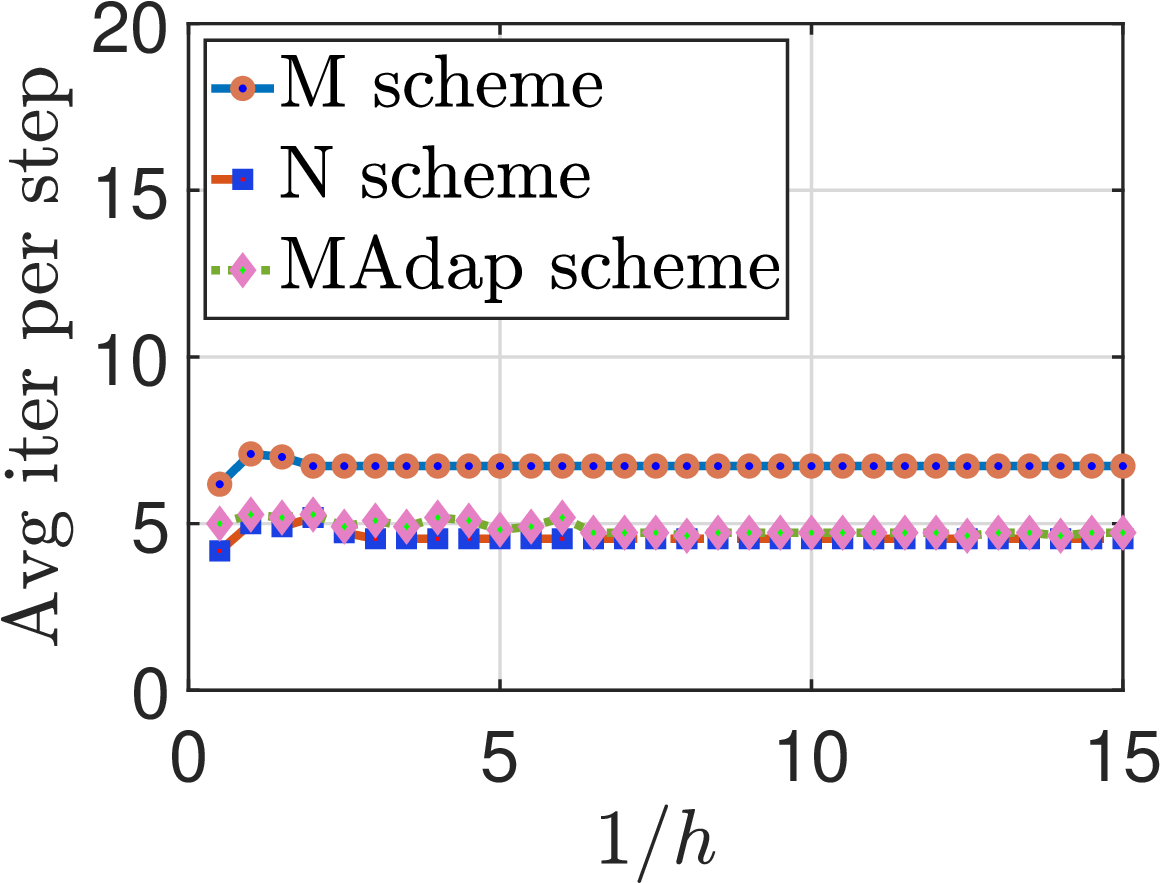}
    \end{subfigure}
     \caption{[\Cref{sec:toy}] Average iterations required per time-step for the double degenerate toy-model in $1D$ \textbf{(top row)} and 2D \textbf{(bottom row)} with varying mesh sizes $h$. The stopping criterion is based on \eqref{eq:Stop}, with a tolerance of \( \epsilon_{\rm stop} = 10^{-6} \). Here, $T=1$ for 1D and $0.1$ for 2D.}
    \label{fig:DD_meshplot}
\end{figure}
\Cref{fig:DD_meshplot} presents a comparison of the three schemes, highlighting that in 1D, although Newton is faster for coarser meshes, it is outperformed by both the adaptive and fixed M-schemes as the mesh is refined. In fact, in 1D, for mesh fine enough, the Newton starts diverging. On the contrary, the M-schemes exhibit mesh-independent behavior demonstrating their robustness to changes in mesh size. 
\begin{figure}[h]
\centering
    \begin{subfigure}[b]{0.49\textwidth}
        \centering     \includegraphics[width=\textwidth]{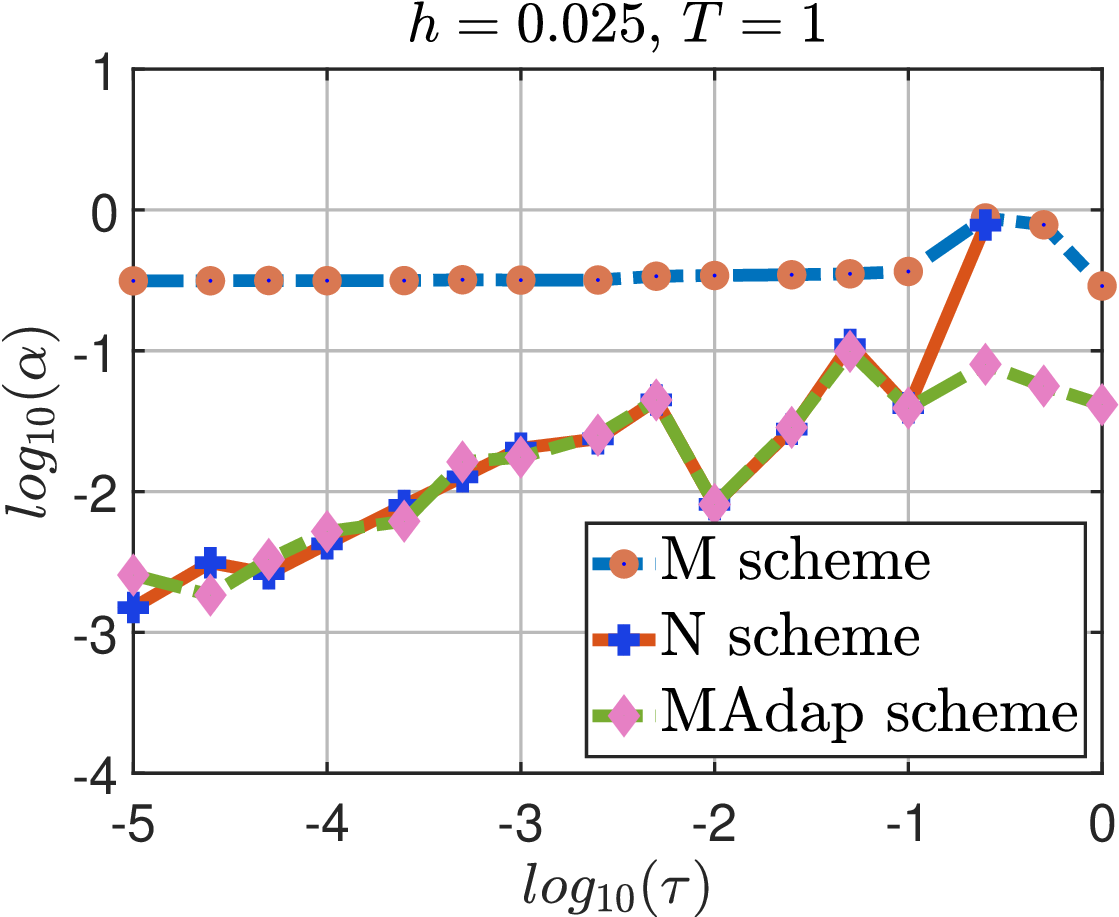}
\end{subfigure}
    \begin{subfigure}[b]{0.49\textwidth}
        \centering     \includegraphics[width=\textwidth]{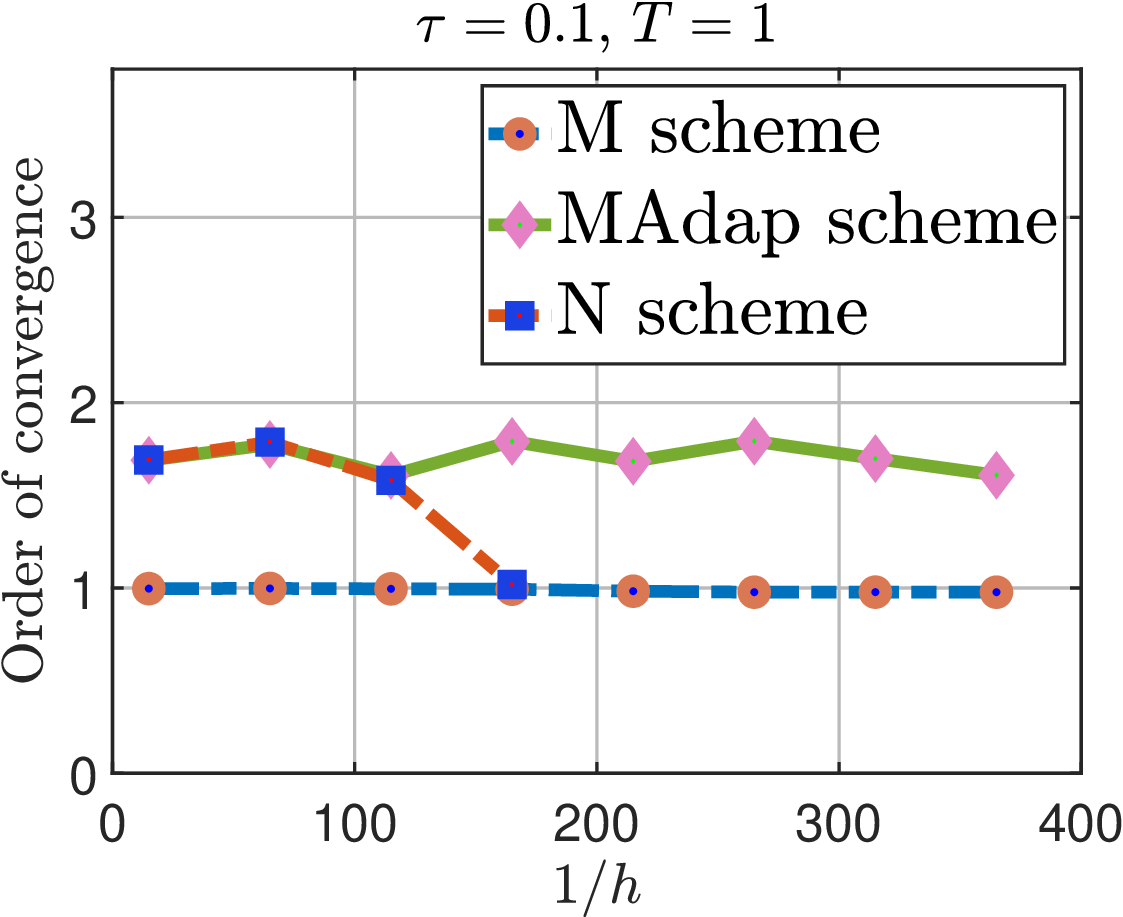}
\end{subfigure}    
\caption{[\Cref{sec:toy}] \textbf{(left)} Average contraction rate \((\alpha)\) vs. time-step size \((\tau)\) with mesh size \(h = 0.025\) for the 1D case. 
The stopping criterion here uses a tolerance of \( \epsilon_{\rm stop} = 10^{-10} \). 
\textbf{(right)} Order of convergence of the iterative methods.}
    \label{fig:DD_cont_order}
\end{figure}

\Cref{fig:DD_cont_order} (left) shows that the linear scaling of the convergence rate $\a$ with $\t$ is lost for the fixed M-scheme. We believe this is due to the presence of a large degenerate region at $\vs=1$, see \Cref{fig:doubledegenerate} (right). Thus, not all conditions specified in \Cref{theo:MS} for linear scaling are satisfied. The M-scheme maintains a nearly constant contraction rate $\alpha$ regardless of $\tau$. The adaptive M-scheme and Newton scheme, however, demonstrate a clear improvement in convergence for smaller $\tau$ and for coarser mesh. 
 \Cref{fig:DD_cont_order}(right) presents the order of convergence of the iterative methods computed from the last three iterations. The M-scheme exhibits linear convergence, as indicated by its consistent first-order behavior. The Newton scheme, shows a quasi-quadratic convergence behavior when reaching $\epsilon_{\rm stop}$, implying that the quadratic regime only becomes apparent much later. However, when \( h \) drops below \( 1/180 \), the Newton scheme diverges. The quasi-quadratic behavior is also shown by the adaptive scheme, but for the full range of mesh sizes.
 \begin{figure}[h!]
    \begin{subfigure}[b]{0.49\textwidth}
       \centering
        \subcaption{1D, $\tau = 0.1$}\includegraphics[width=\linewidth]{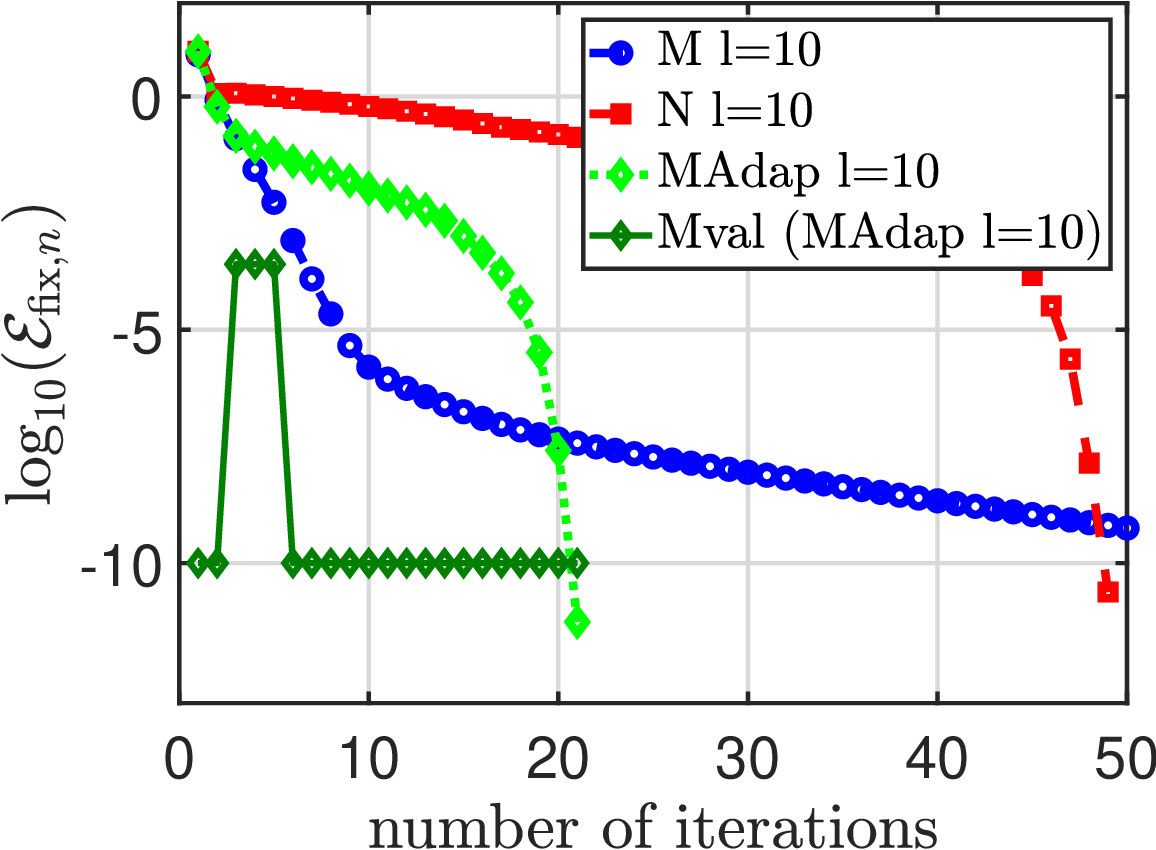}
   \end{subfigure}
    \begin{subfigure}[b]{0.49\textwidth}
        \centering
        \subcaption{1D, $\tau = 0.1$}        \includegraphics[width=\linewidth]{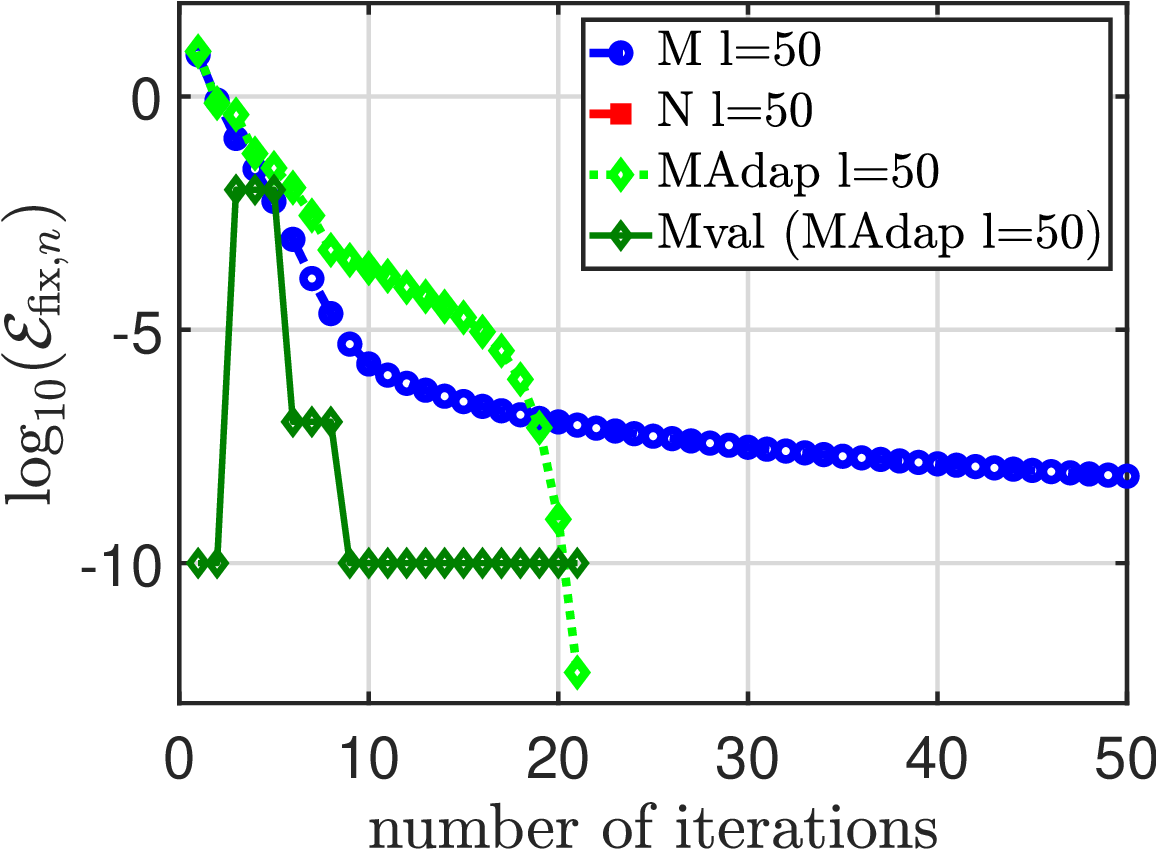}
   \end{subfigure}
    \caption{[\Cref{sec:toy}] Error $\mathcal{E}_{\mathrm{fix},n}^{i}$ vs. iteration $i$ for different iterative schemes for the first time-step ($\tau=0.1$) and mesh size $h=0.1/l$. 
    The Mval quantity shows how the $M$ varies with iteration for the adaptive scheme. }
\label{fig:DD_ErrVsiter}
\end{figure}

\Cref{fig:DD_ErrVsiter} shows how the error decreases with iteration for different iterative schemes. The M-scheme demonstrates rapid initial convergence showing its computational efficiency in such regimes. 
However, it hits a linear convergence regime after reaching a certain error level (approximately $10^{-6}$), which causes its error decay to plateau. In contrast, the adaptive M-scheme converges slowly in the beginning, but after some iterations, the $M$-values start decreasing and the adaptive scheme converges (quasi)-quadratically. Thus, it reaches lower error levels faster. Newton also shows (quasi)-quadratic behavior, for coarser mesh values; however, for finer mesh values, the error fails to decay and instead diverges. The M-schemes, on the other hand, are more stable in this respect.

\subsection{The biofilm equation}\label{sec:Bio}
Next, we consider an equation modeling the growth of biofilms \cite{van2002mathematical}, where the reaction term is of the Fisher type (logistic growth). It corresponds to $\Phi$ being singular at $\vs=1$ which represents the increasing tendency of the bacteria in the biofilm colony to spread when the maximum packing density is reached: 
\begin{align}
\frac{\partial u}{\partial t} = \Delta w +  \frac{1}{2}\, u(1 - u), \quad \text{ where } w =\Phi(u) \text{ and } \Phi(u)=\frac{u^m}{(1-u)^m}.
\end{align}
For the initial condition $\g=0.5$ and $m=6$ are chosen in \eqref{eq:Barenblatt}. The $\g=0.5$ value guarantees that the solution is reasonably far from the singularity at $t=0$, and the Fischer reaction term ensures that the singularity is never reached despite the biofilm growing. For the parameters chosen, $u^* = 0.36778$ in \ref{ass:Bphi} is computed.
\begin{figure}[h!]
    \centering
    \begin{subfigure}[t]{0.46\textwidth}
        \centering
        \includegraphics[width=\textwidth]{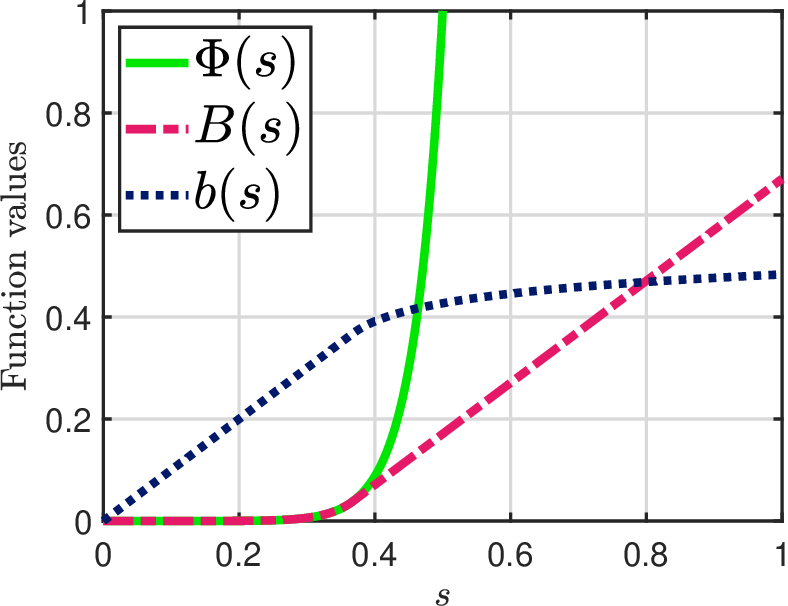}
    \end{subfigure}%
    \begin{subfigure}[t]{0.53\textwidth}
        \centering
        \includegraphics[width=\textwidth]{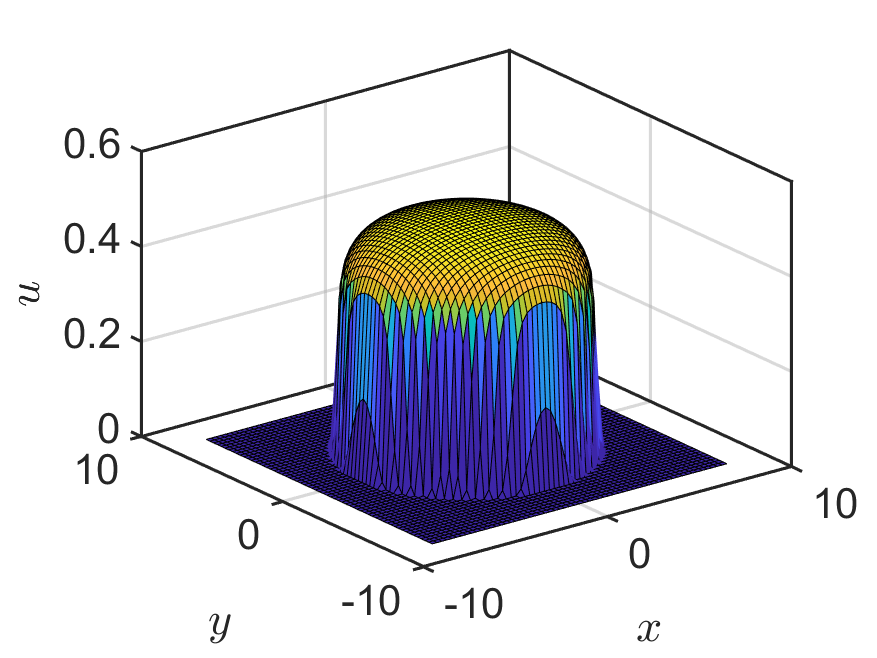}
    \end{subfigure}
    \caption{[\Cref{sec:Bio}] \textbf{(left)} The functions $b$ and $B$ computed from $\Phi$ in \eqref{eq:toy}. \textbf{(right)} Numerical solution at $T = 1$ with a time-step size of $\tau = 0.1$.}
    \label{Bio_Bb_Exact}
\end{figure}

\begin{figure}[h!]
    \centering
    \begin{subfigure}[t]{0.33\textwidth}
        \centering
                \subcaption{$\tau = 10^{-1}$}

        \includegraphics[width=\textwidth]{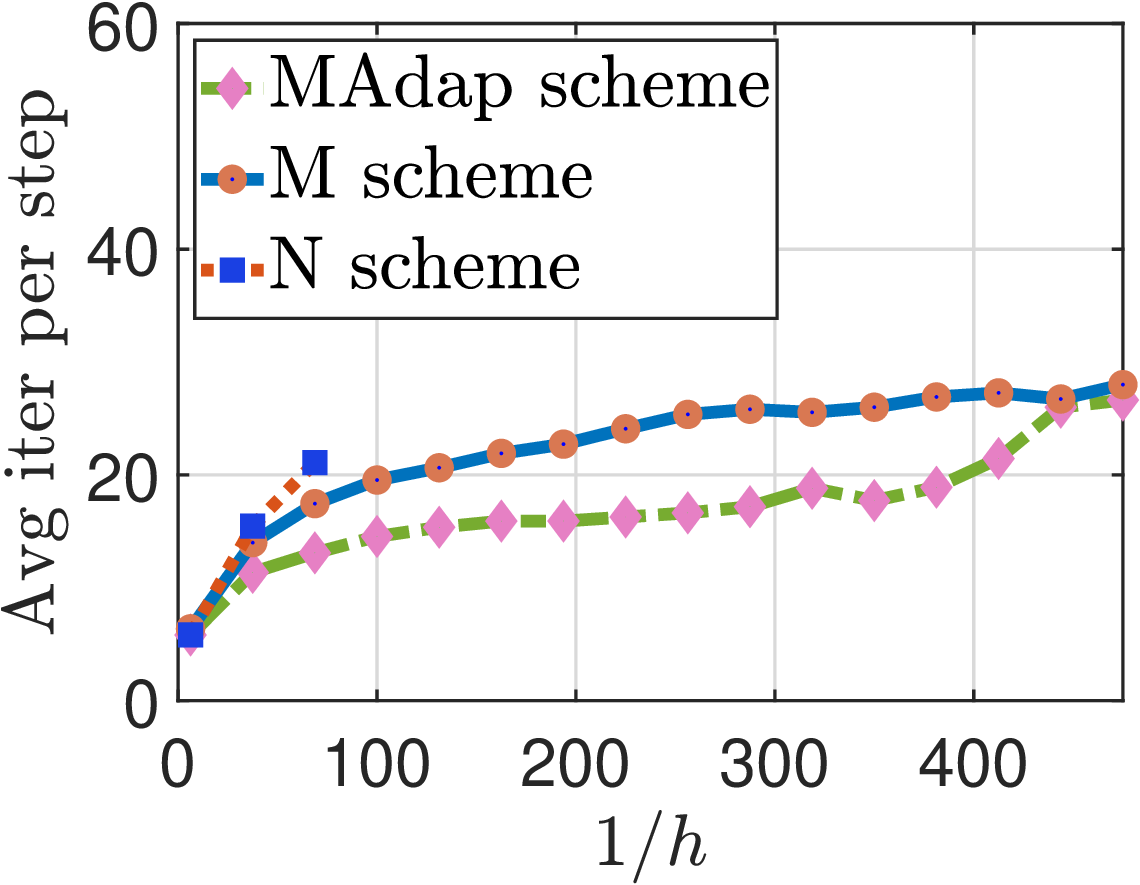}
    \end{subfigure}
    \hfill
    \begin{subfigure}[t]{0.32\textwidth}
        \centering
                \subcaption{$\tau = 10^{-1.5}$}

        \includegraphics[width=\textwidth]{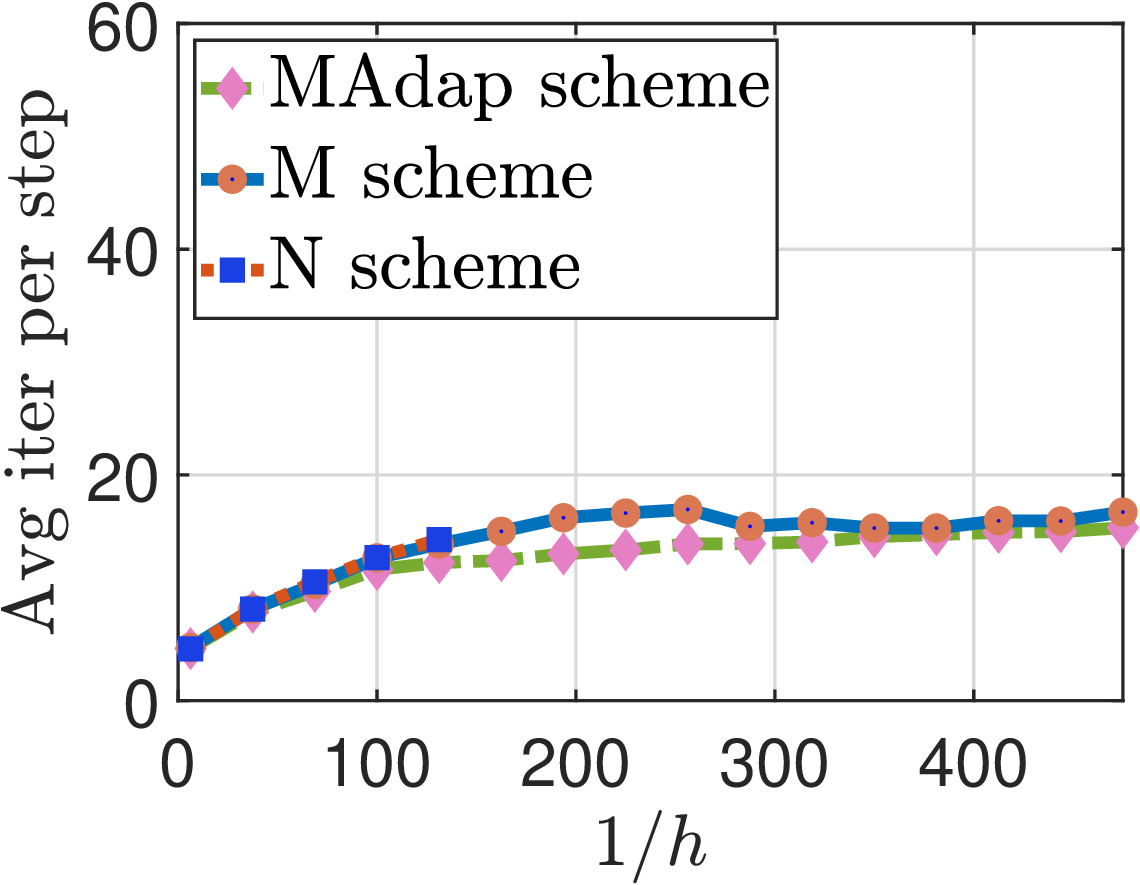}
    \end{subfigure}
    \hfill
    \begin{subfigure}[t]{0.33\textwidth}
        \centering
                \subcaption{$\tau = 10^{-2}$}

        \includegraphics[width=\textwidth]{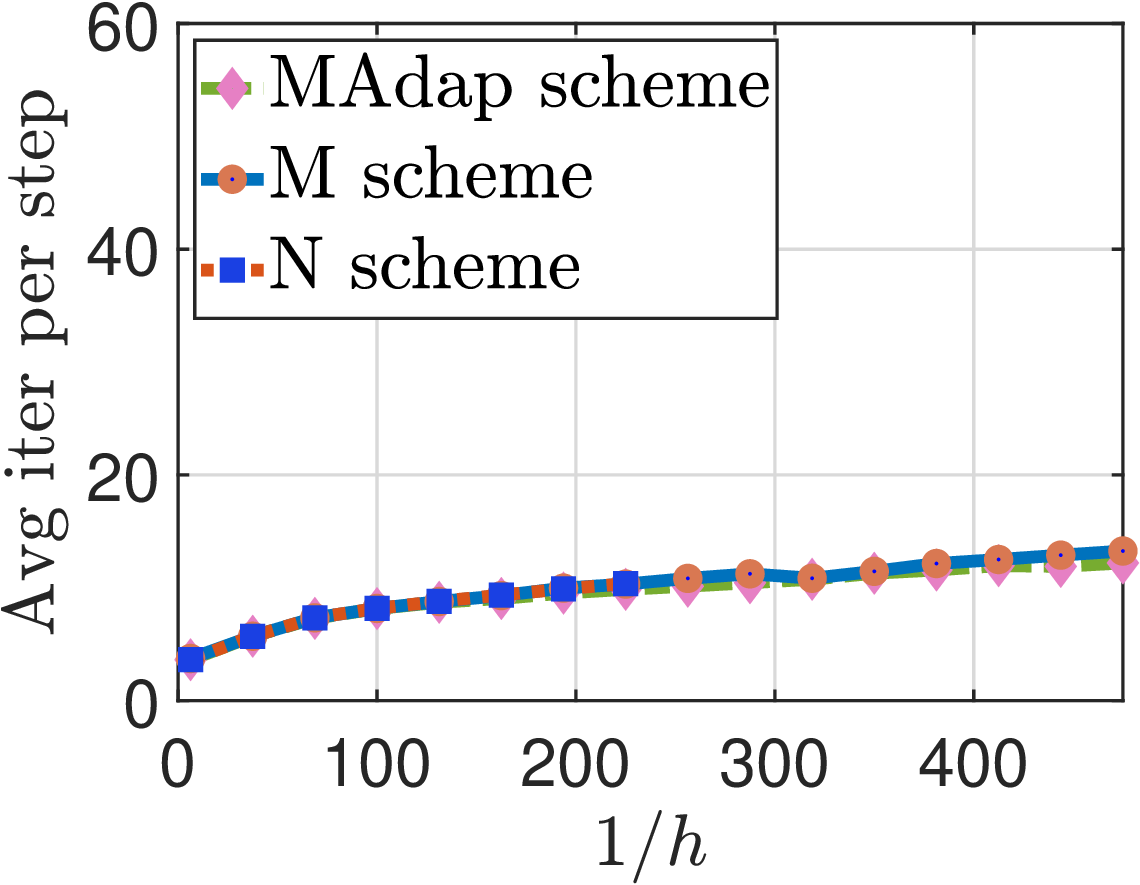}
    \end{subfigure}
    
    \begin{subfigure}[t]{0.32\textwidth}
        \centering
                \subcaption{$\tau = 10^{-1}$}

        \includegraphics[width=\textwidth]{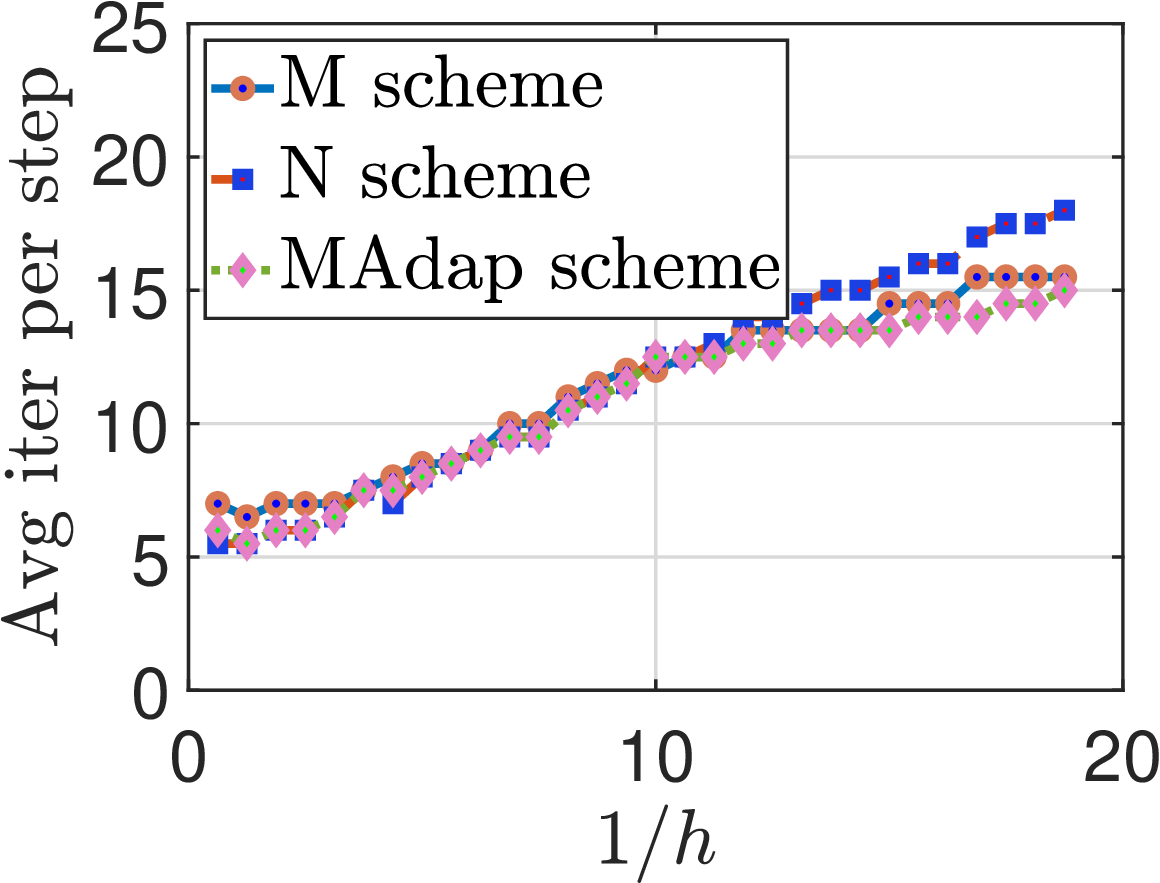}
    \end{subfigure}
    \hfill
    \begin{subfigure}[t]{0.32\textwidth}
        \centering
                \subcaption{$\tau = 10^{-1.5}$}

        \includegraphics[width=\textwidth]{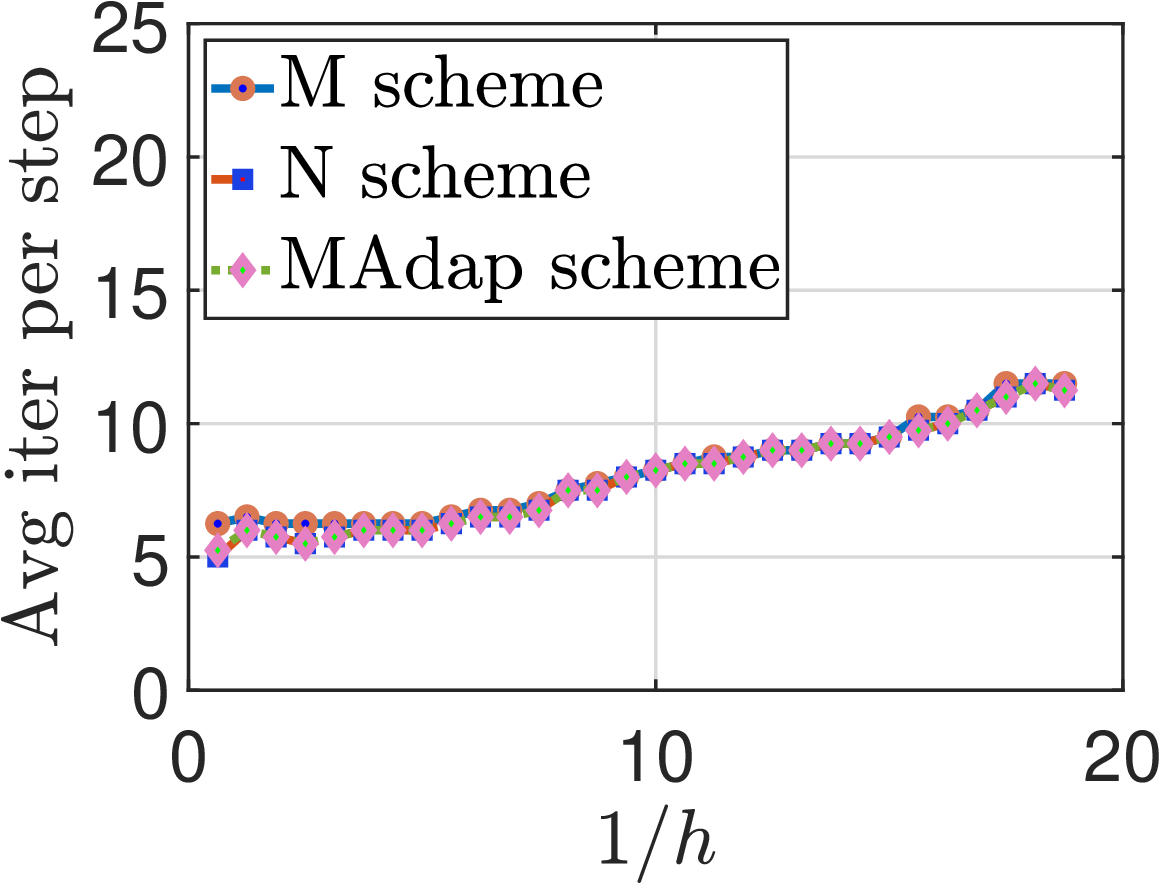}
    \end{subfigure}
    \hfill
    \begin{subfigure}[t]{0.32\textwidth}
        \centering
                \subcaption{$\tau = 10^{-2}$}

        \includegraphics[width=\textwidth]{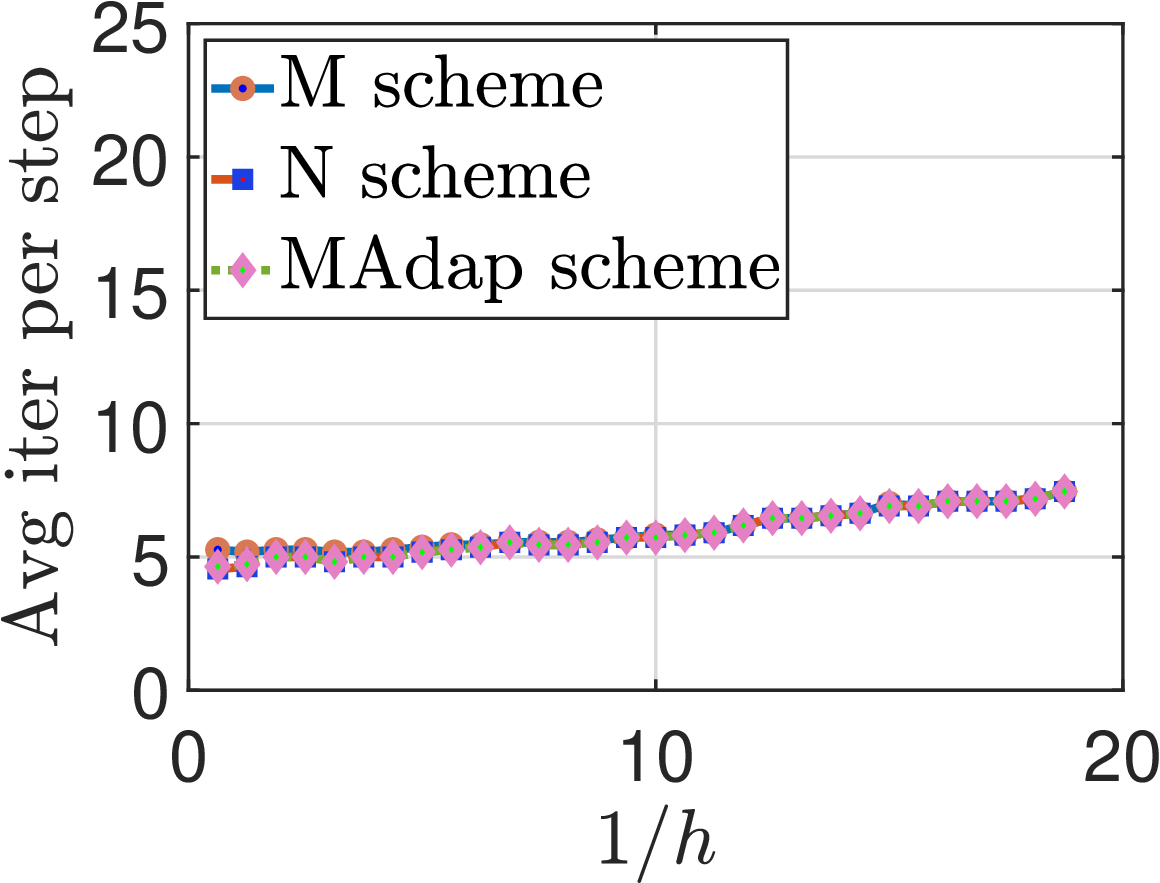}
    \end{subfigure}
    \caption{[\Cref{sec:Bio}]Average iterations required per time-step for the  biofilm equation in $1D$ \textbf{(top row)} and $2D$ \textbf{(bottom row)} with varying mesh sizes $h$. The stopping criterion is based on \eqref{eq:Stop} with a tolerance of \( \epsilon_{\rm stop} = 10^{-6} \). Here, $T=1$ for 1D and $0.1$ for 2D.}
    \label{Biofilm_meshplot}
\end{figure}
\Cref{Biofilm_meshplot} presents comparisons between the $1D$ and $2D$ results obtained for the biofilm model. In the 1D case, the Newton scheme converges for larger time steps (\(\tau\)) only on coarser meshes (\(h\)). Reducing \(\tau\) improves convergence on finer meshes, but when \(h < \frac{1}{250}\), divergence occurs even for small \(\tau\). Similarly, in the 2D scenario with a refined mesh, when $\tau = 0.1$, the Newton scheme required more iterations to achieve convergence. However, for smaller time steps, all schemes demonstrated comparable performance due to the $M\t$ term becoming small. The M-schemes converged in all cases, and the adaptive scheme required the least iterations in almost all cases.

\begin{figure}[h!]
    \centering
         \begin{subfigure}[b]{0.49\textwidth}
        \centering
    \includegraphics[width=\textwidth]{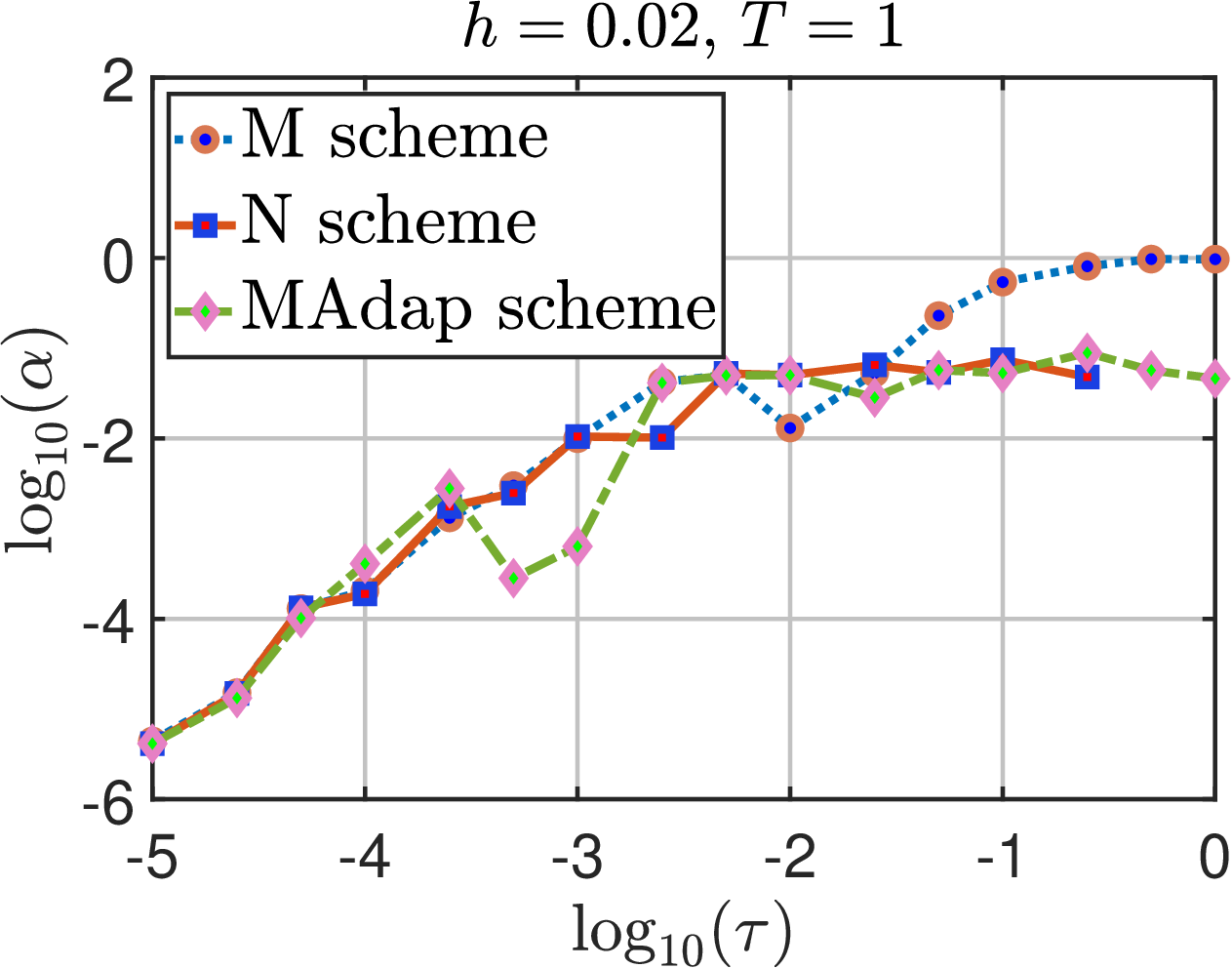}
    \end{subfigure}
    \begin{subfigure}[b]{0.49\textwidth}
        \centering
     \includegraphics[width=\textwidth]{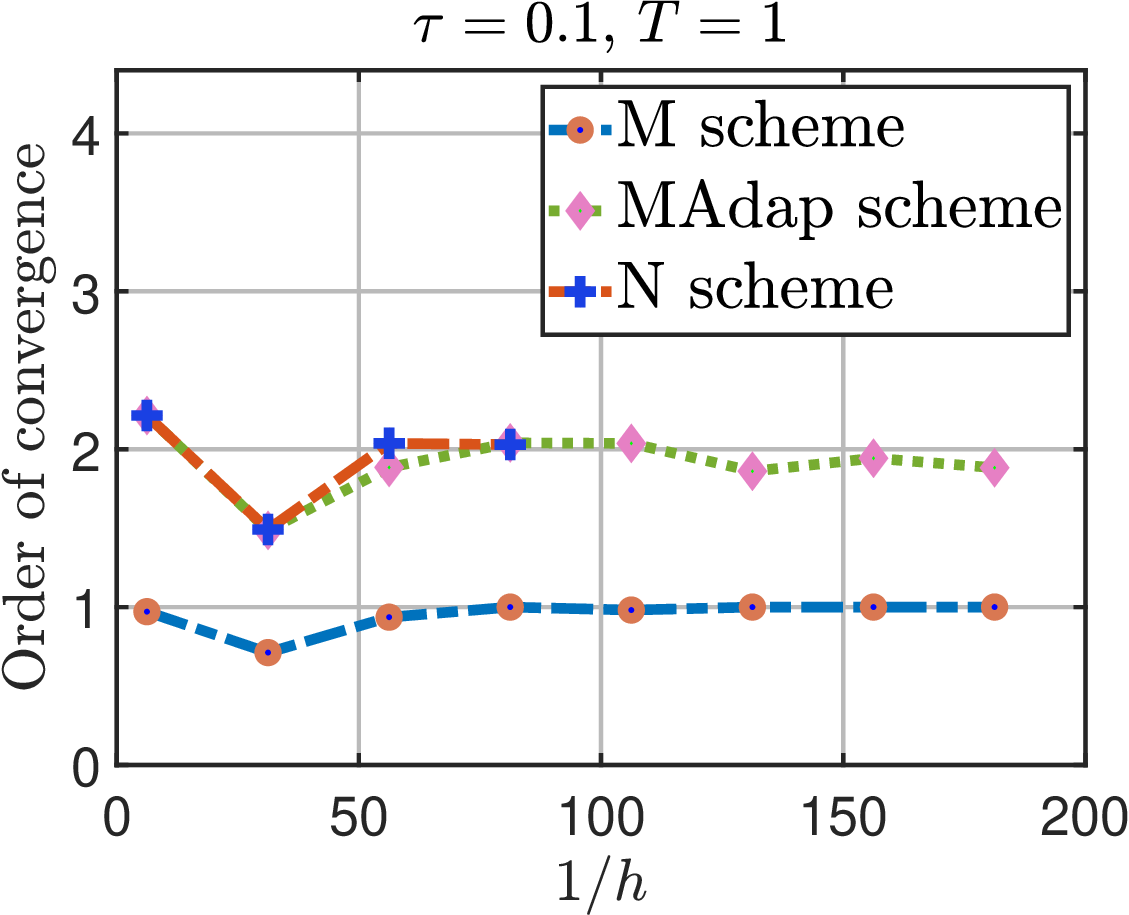}
    \end{subfigure}
   
    \caption{[\Cref{sec:Bio}] \textbf{(left)} Average contraction rate \((\alpha)\) vs. time-step size \((\tau)\) with mesh size \(h = 0.02\) for the 1D case. 
The stopping criterion here uses a tolerance of \( \epsilon_{\rm stop} = 10^{-10} \). 
\textbf{(right)} Order of convergence of the iterative methods.}
    \label{fig:Biofilm_cont_order}
\end{figure}

\Cref{fig:Biofilm_cont_order} (left) analyzes the contraction rates of different schemes for varying time-step sizes \( \tau \). For small time steps, the contraction rate $\a$ is seen again to increase superlinearly with $\t$. \Cref{fig:Biofilm_cont_order} (right) shows that in the asymptotic limit, the M-scheme is indeed linear, whereas, the adaptive scheme is quadratic. As before, this is supported by \Cref{fig:Bio_ErrVsiter}. For smaller error levels, the fixed M-scheme enters a linear convergence regime, whereas, the adaptive scheme enters a quadratic regime. Newton also shows quadratic convergence in the case that it converges, i.e. the (left) case, albeit the convergence is much slower than the adaptive scheme.

\begin{figure}[h!]
    \centering
    \begin{subfigure}[b]{0.48\textwidth}
        \centering
       \subcaption{1D, $\tau = 0.1$}\includegraphics[width=\textwidth]{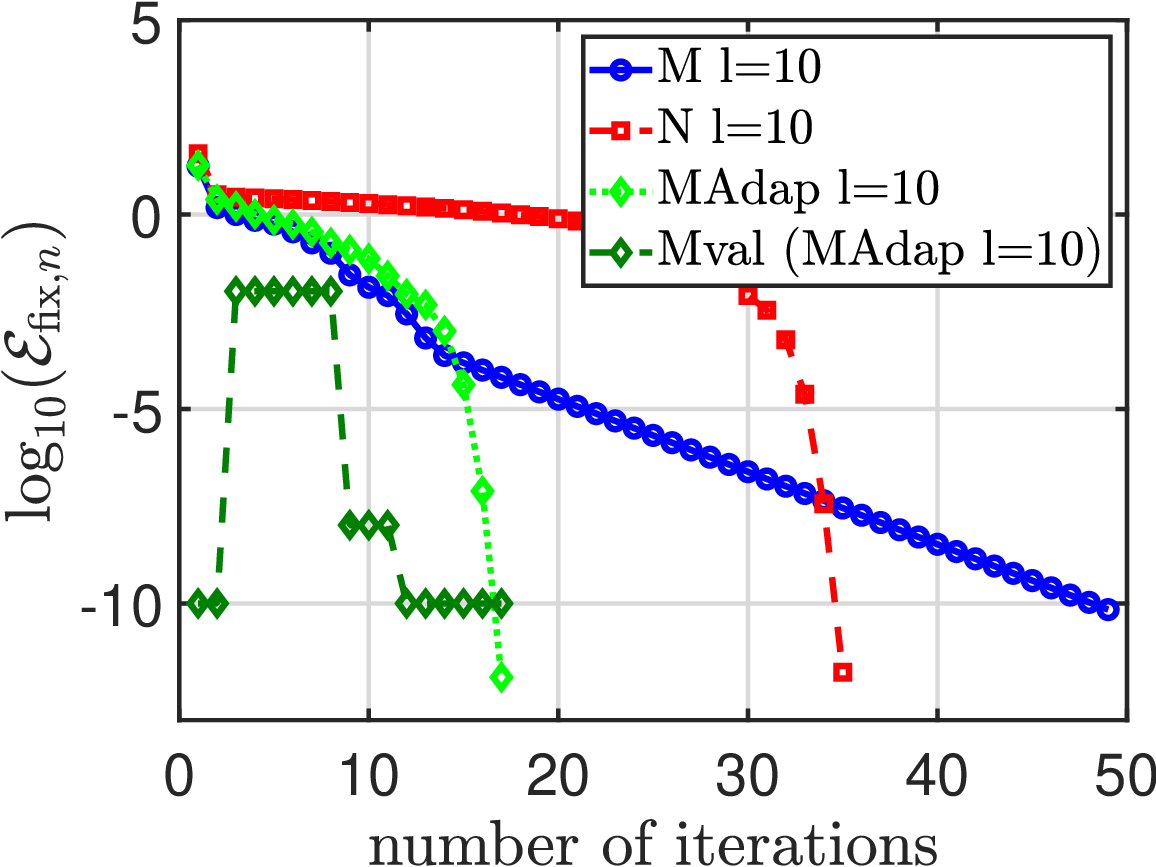}
    \end{subfigure}
    \begin{subfigure}[b]{0.48\textwidth}
        \centering
       \subcaption{1D, $\tau = 0.1$}\includegraphics[width=\textwidth]{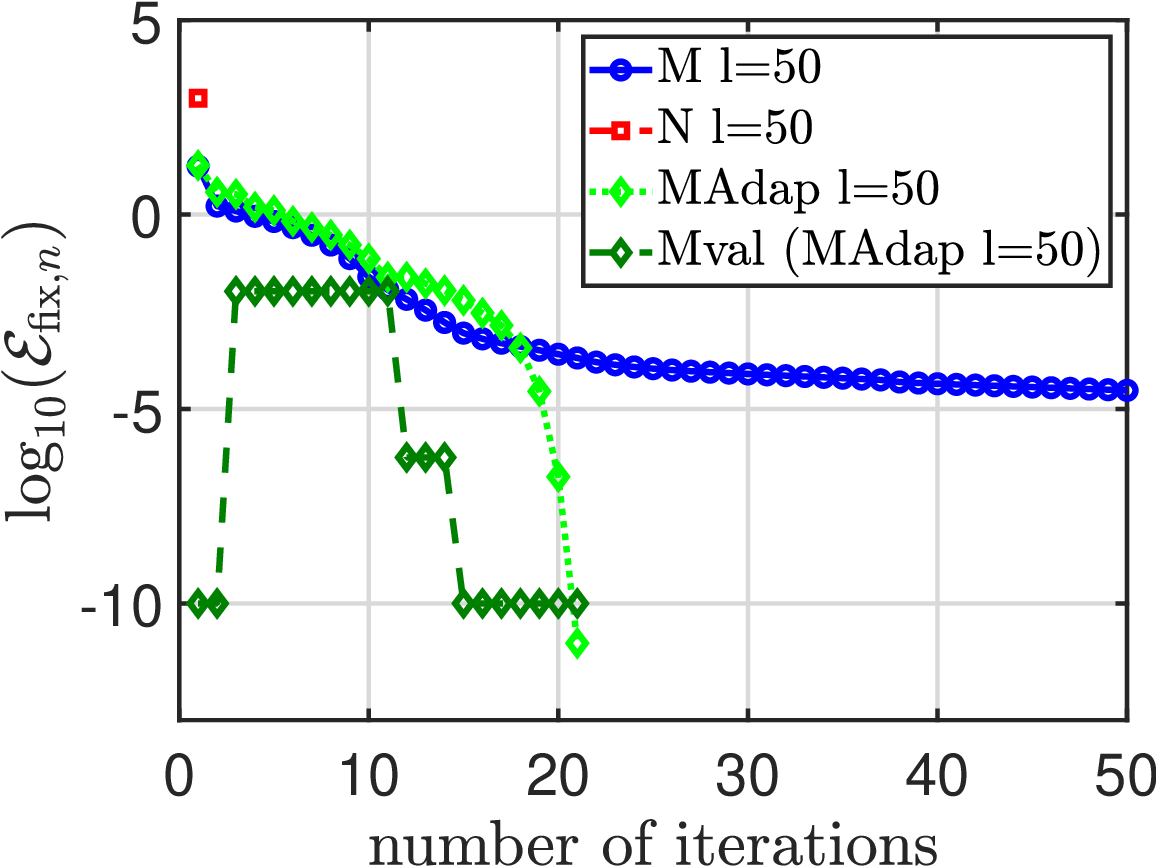}
    \end{subfigure}
    \caption{[\Cref{sec:Bio}] Error $\mathcal{E}_{\mathrm{fix},n}^{i}$ vs. iteration $i$ for different iterative schemes for the first time-step ($\tau=0.1$) and mesh size $h=0.16/l$. 
    The Mval quantity shows how the $M$ varies with iteration for the adaptive scheme.}
    \label{fig:Bio_ErrVsiter}
\end{figure}

 \subsection{The Richards equation}\label{sec:Richards}
Finally, we consider the Richards equation, which is widely used in groundwater modelling. In terms of the capillary pressure $p$, the non-dimensional Richards equation we would solve is:
\begin{align}
\frac{\partial (S(p))}{\partial t} = \nabla \cdot \left( \kappa(S(p))  (\nabla p-\hat{\bm{g}}) \right)+ C S(p)
\end{align}
Here, \(\hat{\bm{g}}\) represents the unit vector along the direction of gravity which we have taken to be the $y$-direction.
The Richards equation involves nonlinearities in all the terms. The saturation function \(S(p)\) is increasing, and the permeability function \(\kappa(S(p))\) takes non-negative values. 
The saturation function \(S(p)\) and the permeability function \(\kappa(s)\) are modeled using the Van Genuchten parametrization \cite{van1980closed} as expressed in a nondimensional setting for \(\lambda \in (0, 1)\). In this work, we consider $\lambda = 0.8$:
\begin{align}
S(p) = \left(1 + (1 - p)^{\frac{1}{1 - \lambda}}\right)^{ - \lambda},
\quad
\kappa(s) = \sqrt{s} \left(1 - \left(1 -s^\frac{1}{\lambda}\right)^\lambda\right)^2
\end{align}
We use the parametrization discussed in \Cref{introduction}, in which $u =S(p)$ and $\Phi(S(p))$ is defined as: 
\begin{align*}
\Phi(S(p)) = \int_{0}^{p} \kappa(S(q)) \, dq.
\end{align*}
There is no analytical expression known for $\Phi$. Thus, the integral is evaluated numerically in an extremely fine grid, and for arbitrary values of the argument, it is recovered using interpolation between the tabulated points. The functions $b(s)$ and $B(s)$ derived from $\Phi$ using \eqref{eq:bBexpression} are constructed subsequently by numerical differentiation and integration.  The resultant functions are plotted in  \Cref{fig:Richards_Bb_Exact} (left). The (right) plot illustrates a converged discrete numerical solution obtained using the M-scheme with $M = 0.01$. The dissymmetry in the solution stems from the nonlinear advection term.

\begin{figure}[h!]
    \centering
    \begin{subfigure}[t]{0.44\textwidth}
        \centering
        \includegraphics[width=\textwidth]{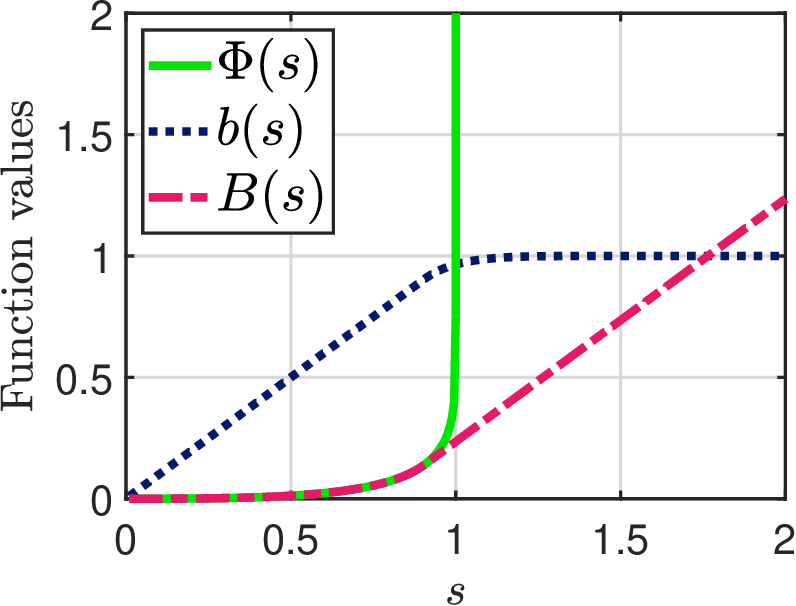}
    \end{subfigure}
    \begin{subfigure}[t]{0.55\textwidth}
        \centering
        \includegraphics[width=\textwidth]{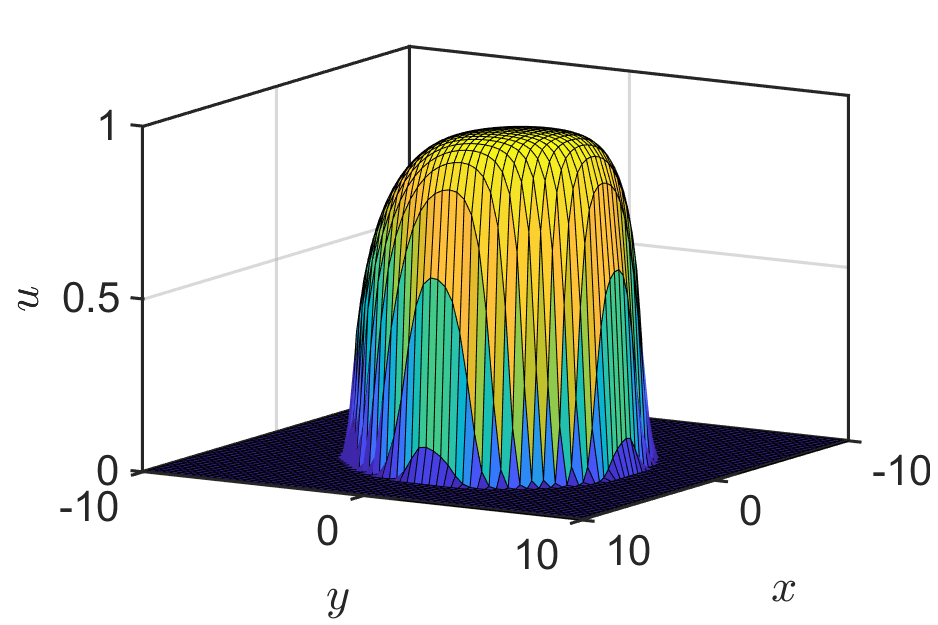}
    \end{subfigure}
    \caption{[\Cref{sec:Richards}] \textbf{(left)} The functions $b$ and $B$ computed from $\Phi$ in \eqref{eq:toy}. \textbf{(right)} Numerical solution at $T = 1$ with a time-step size of $\tau = 0.1$.}
    \label{fig:Richards_Bb_Exact}
\end{figure}

In \Cref{Richards_meshplots}, the average number of iterations for different choices of the mesh size ($h$) with time-step size $(\tau)$ are presented. As expected, the M-scheme is robust and converges in each scenario. Unfortunately, particularly in 1D, the Newton scheme with $M = 0$ does not converge if the time-step size is increased or the mesh is refined. In 2D, since we could not look into very fine mesh sizes, Newton outperformed the M-schemes. However, as $\t$ got smaller, the difference became less since the extra $M\t$ term became less important.

 \begin{figure}[h!]
    \centering
    \begin{subfigure}[t]{0.32\textwidth}
        \centering
                \subcaption{$\tau = 10^{-1}$}
     \includegraphics[width=\textwidth]{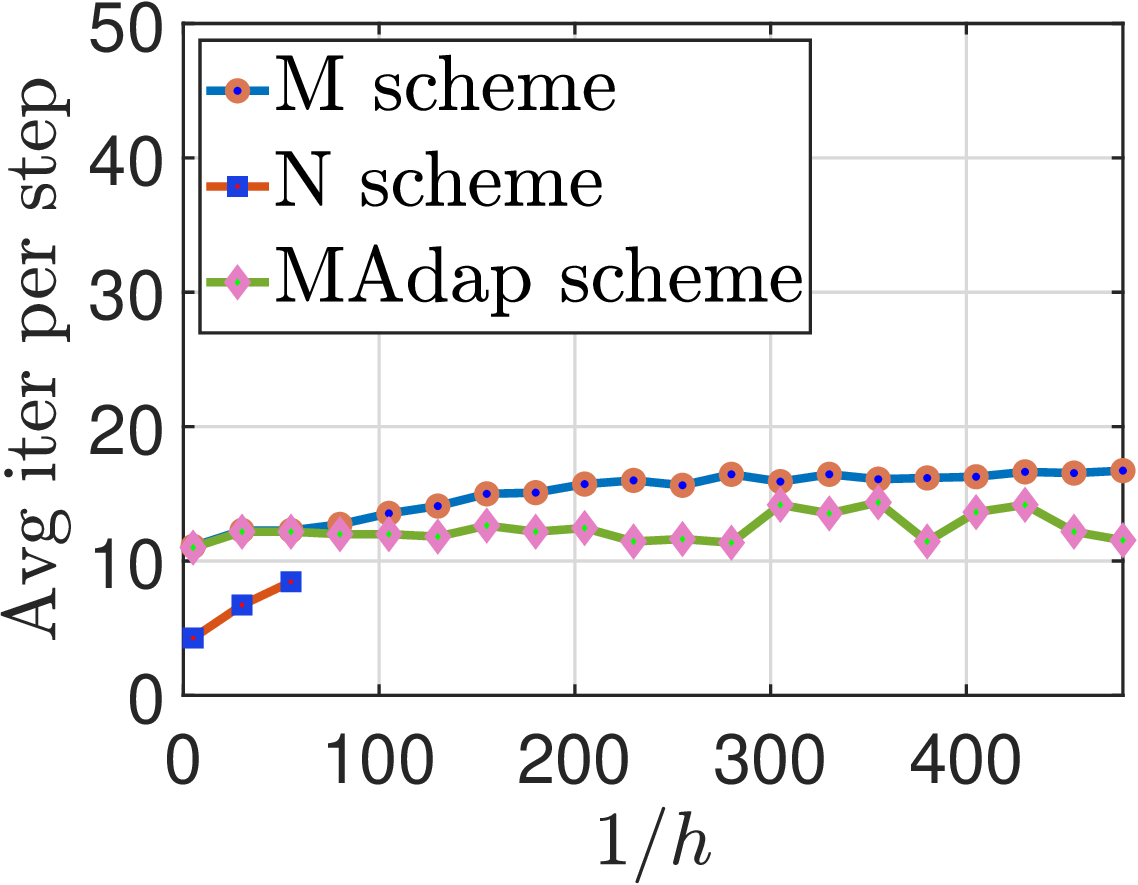}
    \end{subfigure}
    \hfill
    \begin{subfigure}[t]{0.32\textwidth}
        \centering
                \subcaption{$\tau = 10^{-1.5}$}       \includegraphics[width=\textwidth]{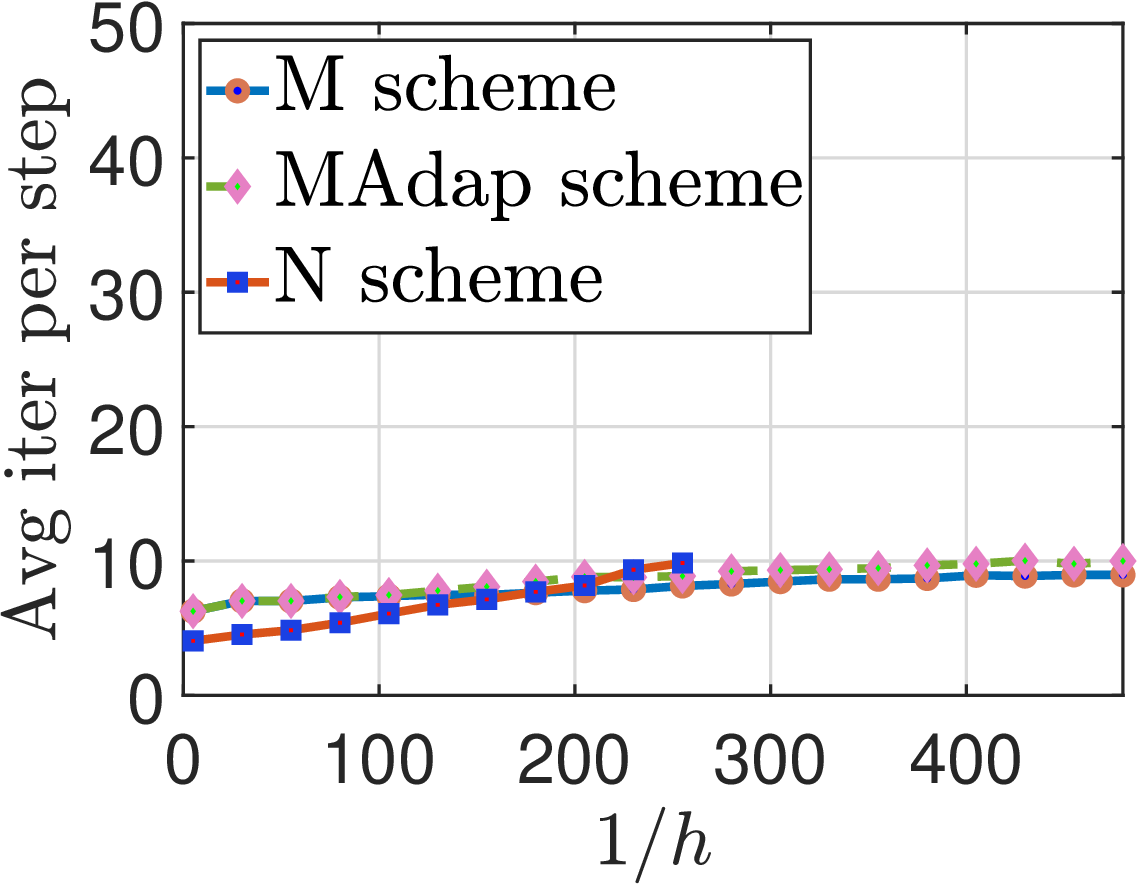}
    \end{subfigure}
    \hfill
    \begin{subfigure}[t]{0.32\textwidth}
        \centering
                \subcaption{$\tau = 10^{-2}$}        \includegraphics[width=\textwidth]{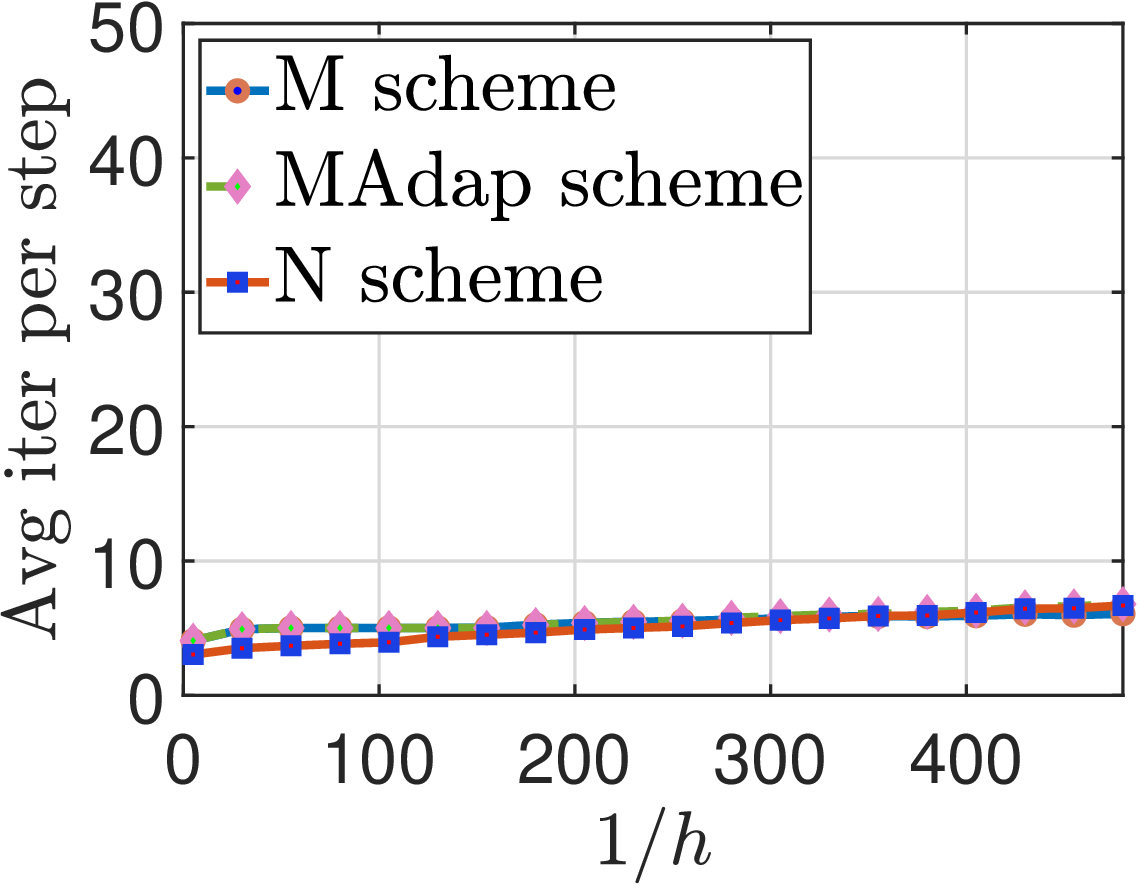}
    \end{subfigure}
    \begin{subfigure}[t]{0.32\textwidth}
        \centering
                \subcaption{$\tau = 10^{-1}$}       \includegraphics[width=\textwidth]{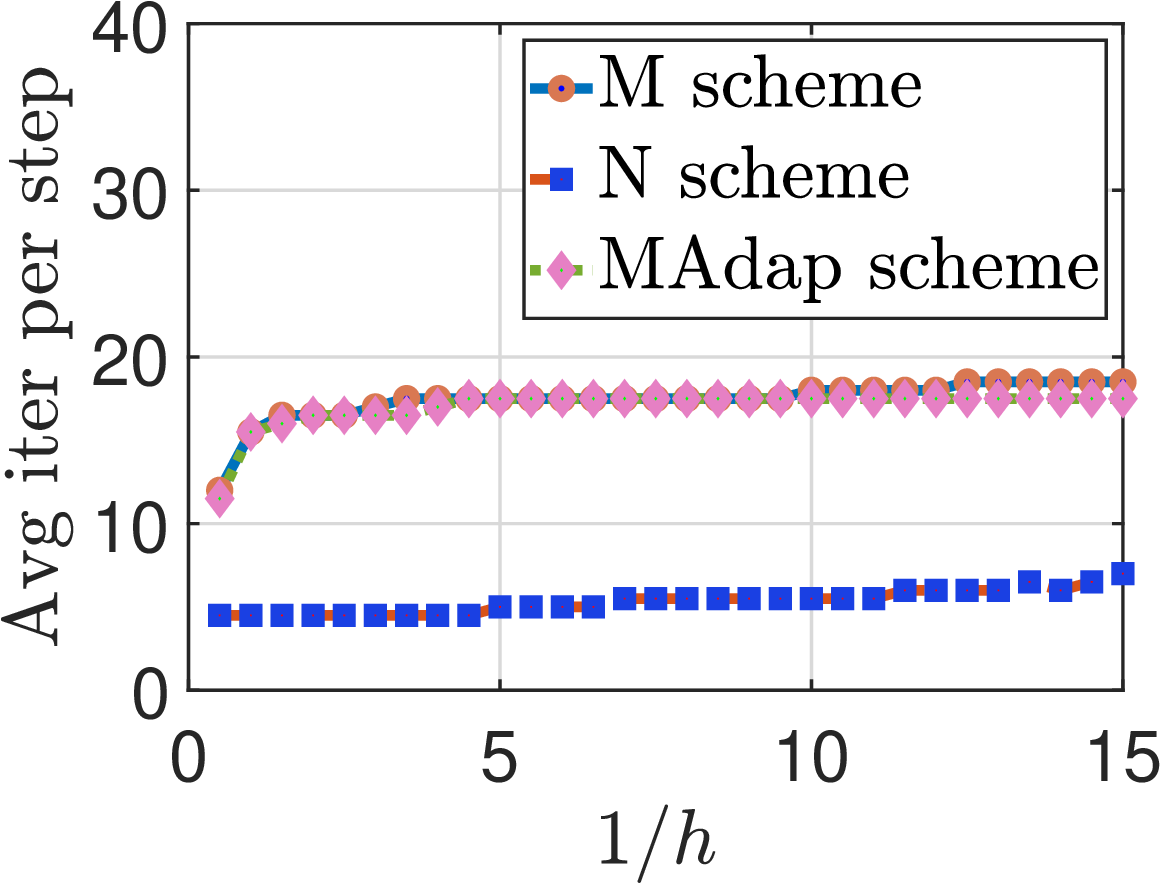}
    \end{subfigure}
    \hfill
    \begin{subfigure}[t]{0.32\textwidth}
        \centering
                \subcaption{$\tau = 10^{-1.5}$}
    \includegraphics[width=\textwidth]{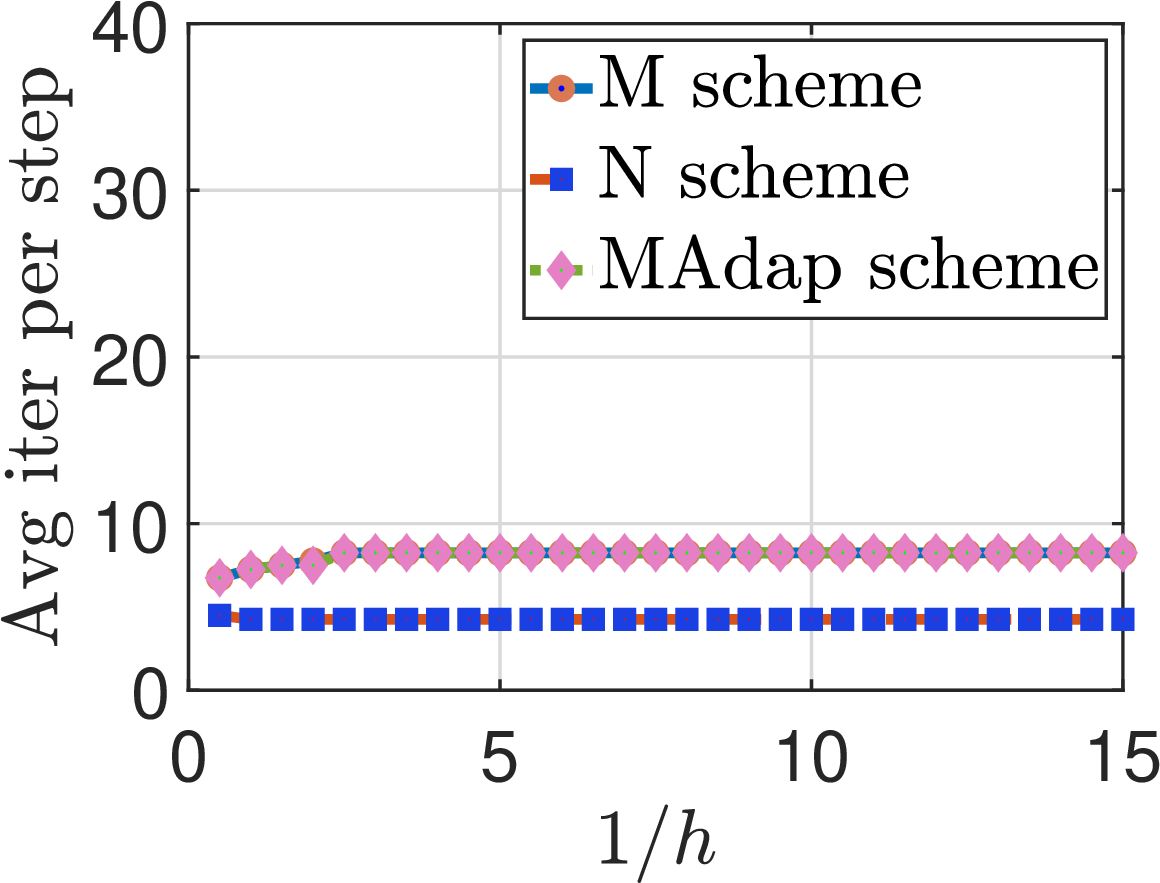}
    \end{subfigure}
    \hfill
    \begin{subfigure}[t]{0.32\textwidth}
        \centering
                \subcaption{$\tau = 10^{-2}$}
\includegraphics[width=\textwidth]{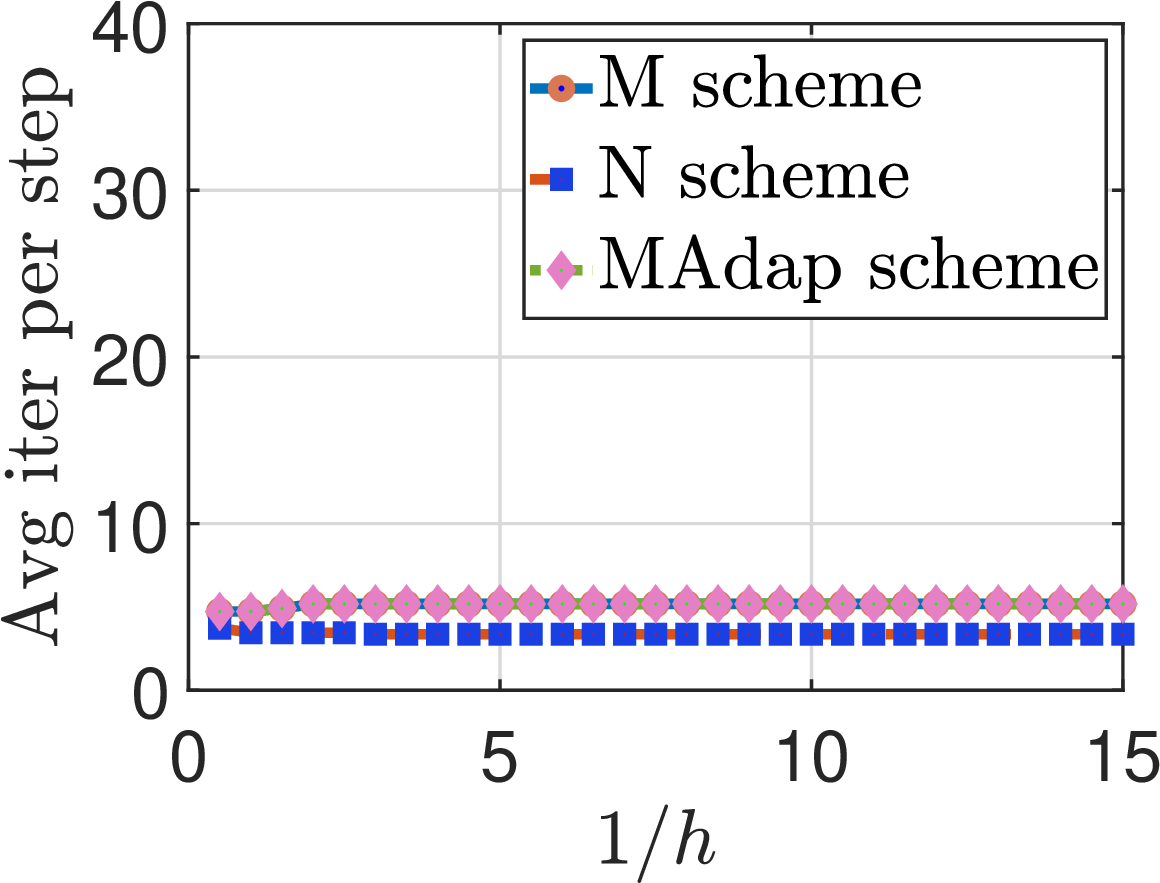}
    \end{subfigure}   \caption{[\Cref{sec:Richards}] Average iterations required per time-step for the  Richards equation in $1D$ \textbf{(top row)} and $2D$ \textbf{(bottom row)} with varying mesh size $h$.  The stopping criterion is based on \eqref{eq:Stop}, with a tolerance of \( \epsilon_{\rm stop} = 10^{-6} \). Here, $T=1$ for 1D and $0.1$ for 2D.}
    \label{Richards_meshplots}
\end{figure}
\begin{figure}[h!]
    \centering
    \begin{subfigure}[t]{0.49\textwidth}
            \centering
        \includegraphics[width=\textwidth]{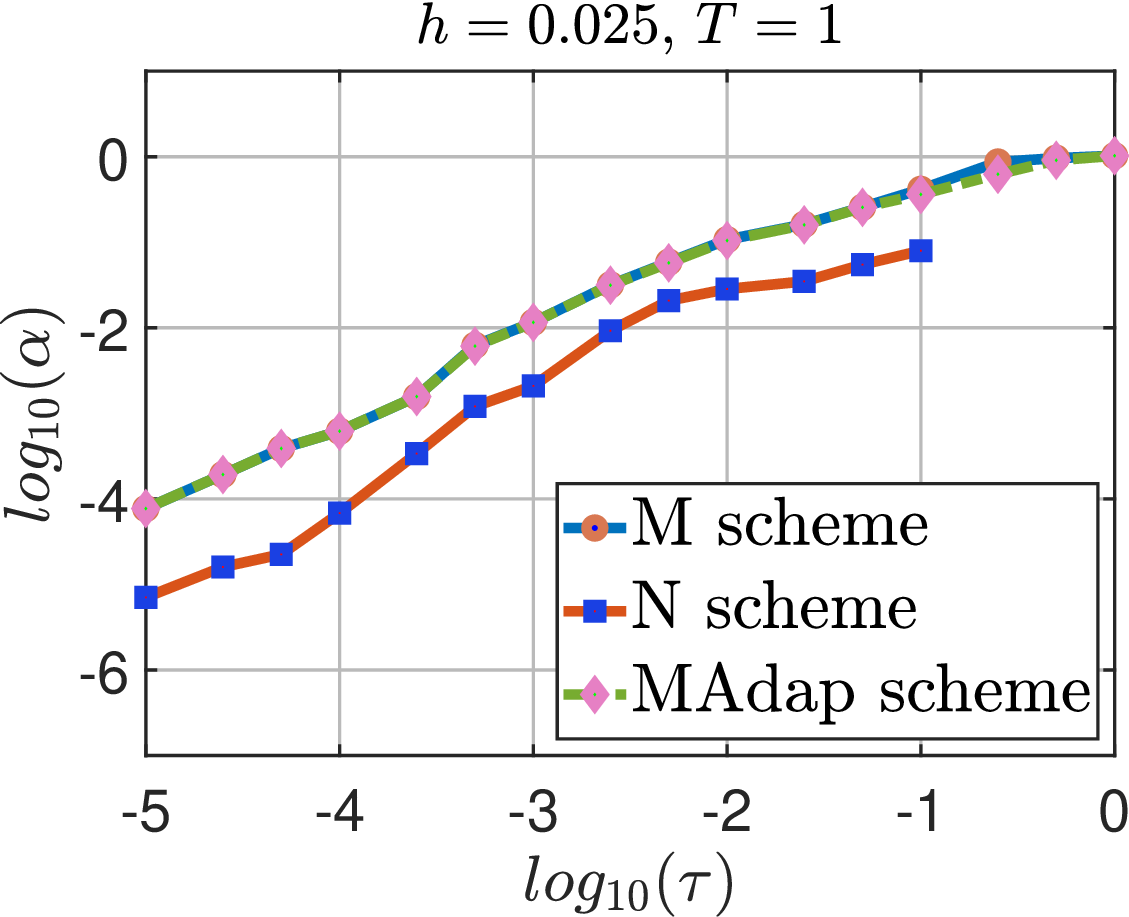}
    \end{subfigure}
     \hfill
    \begin{subfigure}[t]{0.49\textwidth}
       \centering
         \includegraphics[width=\textwidth]{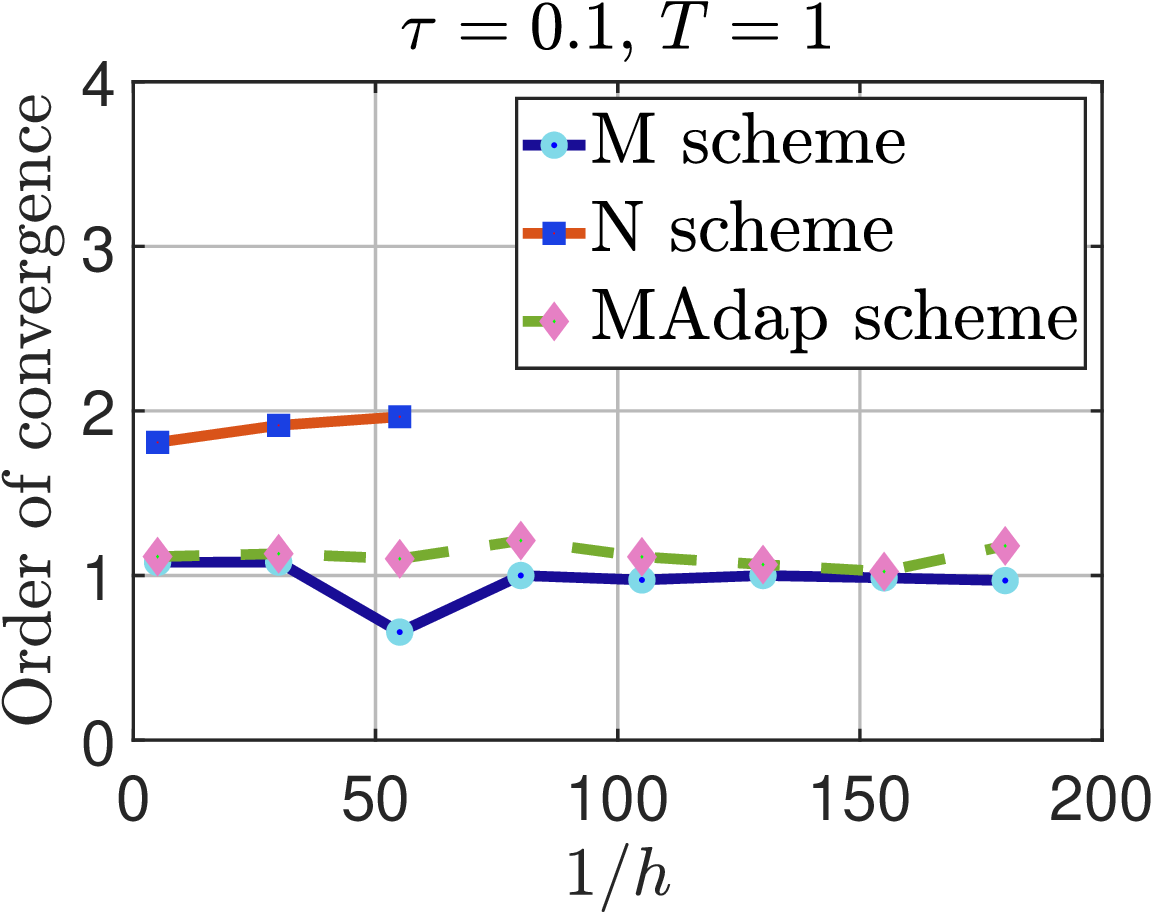}
    \end{subfigure}
    \caption{[\Cref{sec:Richards}] \textbf{(left)} Average contraction rate \((\alpha)\) vs. time-step size \((\tau)\) with mesh size \(h = 0.025\) for the 1D case. 
The stopping criterion here uses a tolerance of \( \epsilon_{\rm stop} = 10^{-10} \). 
\textbf{(right)} Order of convergence of the iterative methods. }
    \label{fig:Richars_Cont_order}
\end{figure}
\Cref{fig:Richars_Cont_order} illustrates both the contraction rate and the order of convergence. The contraction rate $\a$ truly scales linearly with $\t$ for smaller time-step sizes for the M-schemes. The results also indicate that the M-scheme and the adaptive M-scheme exhibit linear convergence, whereas Newton’s method, when it converges, achieves quadratic convergence. \Cref{fig:Richards_ErrVsiter} shows the error decay with iterations for the different schemes. Newton does not converge in both cases. Both the M-schemes show linear convergence until the error level $\epsilon_{\rm stop}=10^{-10}$. However, the adaptive scheme has a steeper descent. 


 \begin{figure}[h!]
    \centering
    \begin{subfigure}[t]{0.48\textwidth}
        \centering
        \subcaption{1D, $\tau = 0.1$}\includegraphics[width=\linewidth]{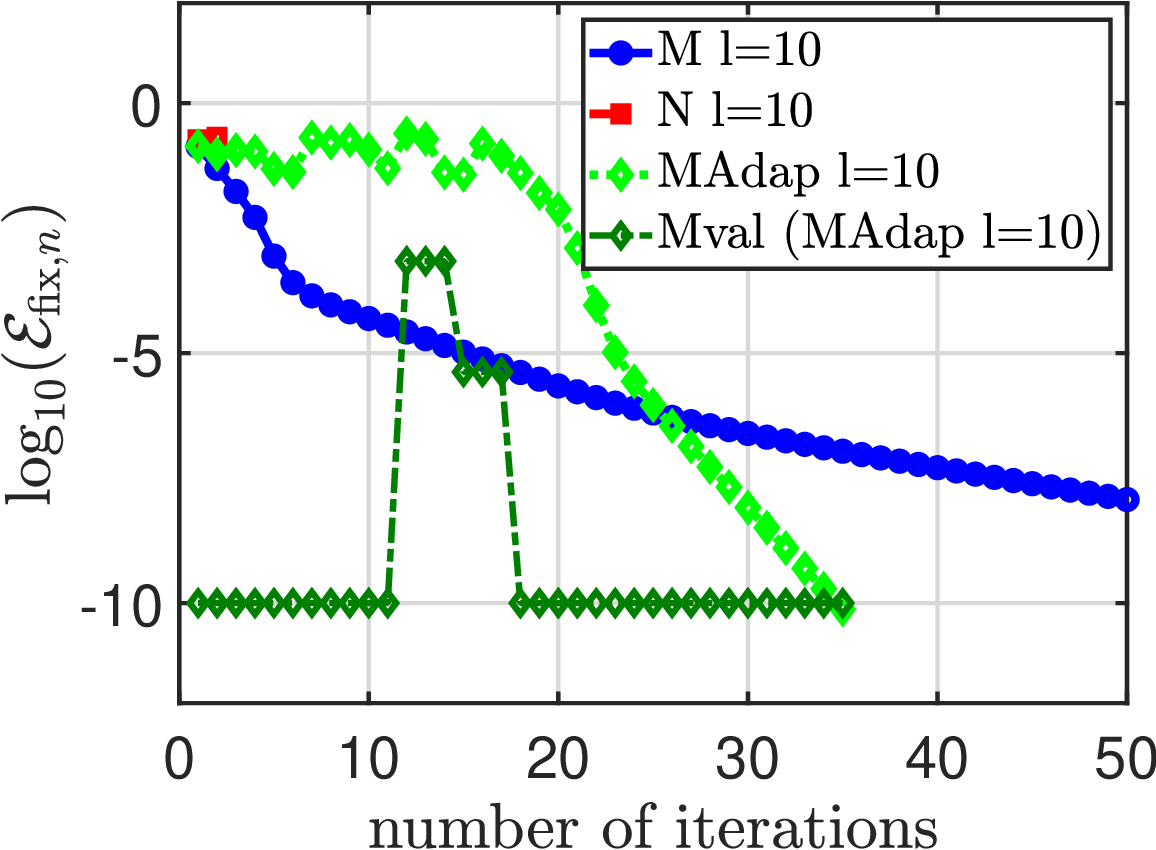}
    \end{subfigure}
    \begin{subfigure}[t]{0.48\textwidth}
        \centering
\subcaption{1D, $\tau = 0.1$}        \includegraphics[width=\linewidth]{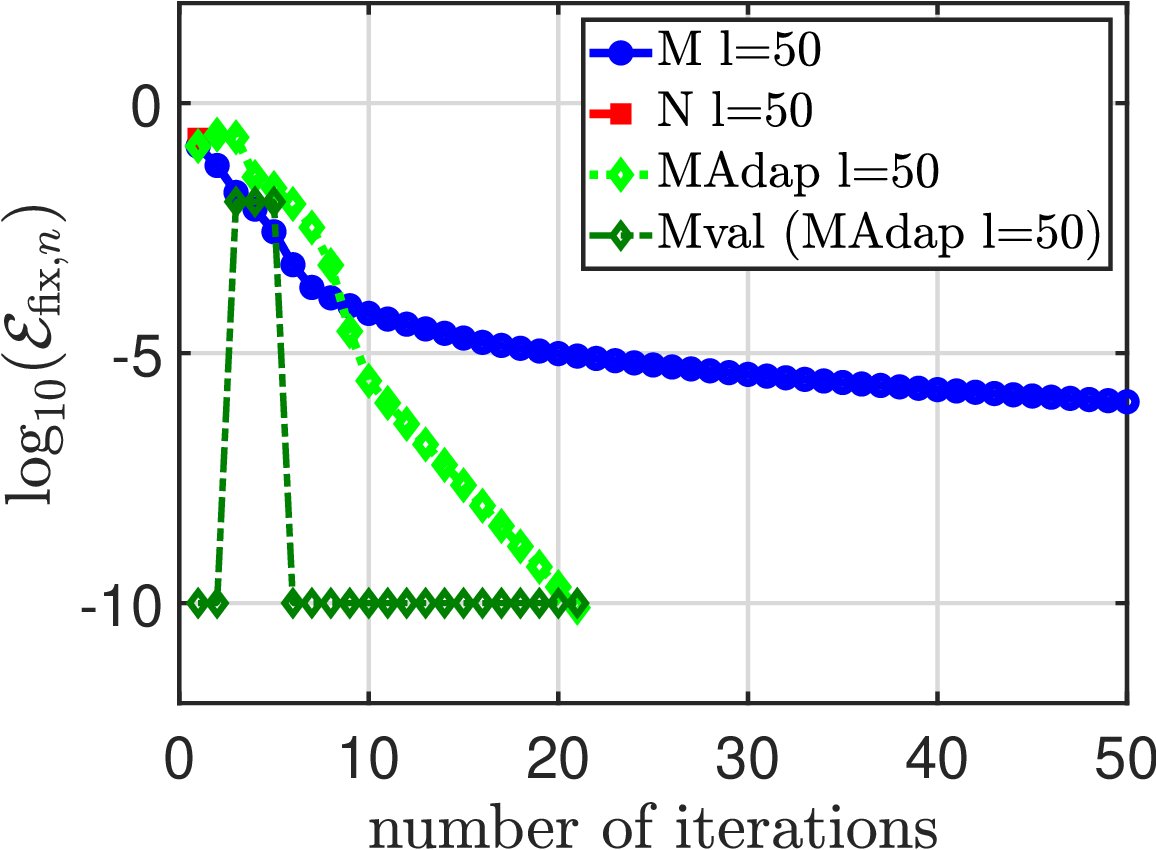}\end{subfigure}
\caption{[\Cref{sec:Richards}] Error $\mathcal{E}_{\mathrm{fix},n}^{i}$ vs. iteration $i$ for different iterative schemes for the first time-step ($\tau=0.1$) and mesh size $h=0.1/l$. 
    The Mval quantity shows how the $M$ varies with iteration for the adaptive scheme.
}
        \label{fig:Richards_ErrVsiter}
\end{figure}   

The departure of the adaptive scheme from asymptotic quadratic convergence is due to the advection term in the Richards equation. The Newton scheme takes a first-order expansion of this term and thus can become quadratic. However, for stability, the M-schemes only use zeroth-order approximation of the term, and therefore, can at most be linear. This is supported by \Cref{fig:Richards _(No Adv)} which shows that if advection is absent, then the quadratic convergence of the adaptive scheme is recovered. Newton also converges in this case even for finer meshes, although the adaptive scheme is always faster.

\begin{figure}[h!]
    \centering
 \begin{subfigure}[t]{0.5\textwidth}
    \centering   \includegraphics[width=\textwidth]{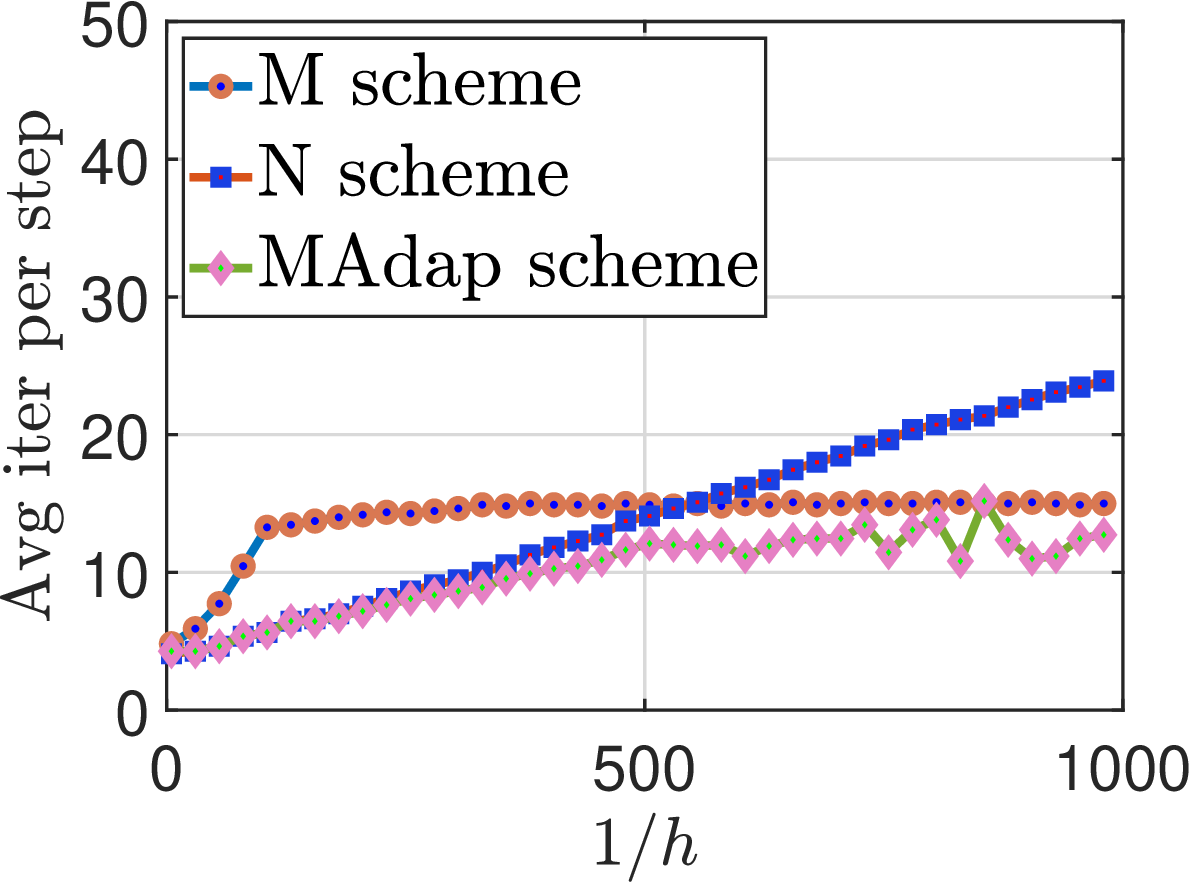}
\end{subfigure}
\hfill
 \begin{subfigure}[t]{0.49\textwidth}
       \centering       \includegraphics[width=\textwidth]{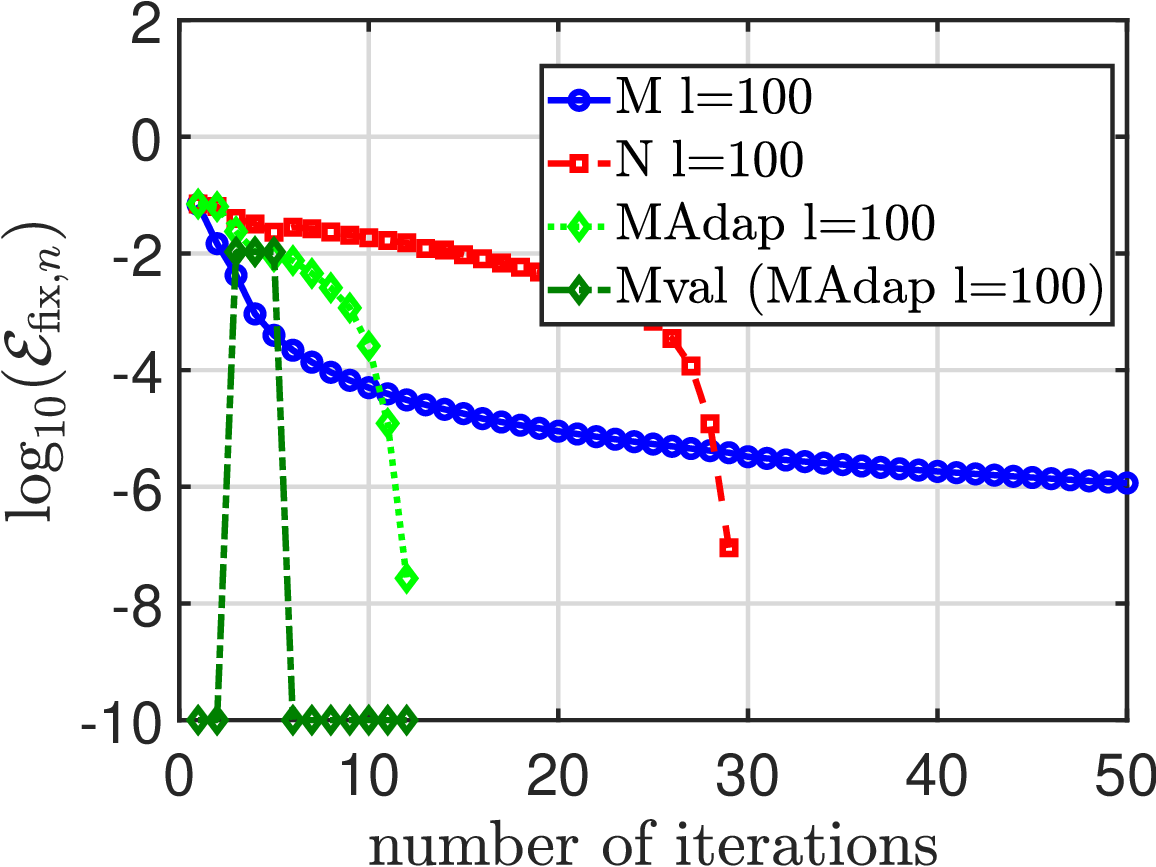}
   \end{subfigure}    
        \caption{[\Cref{sec:Richards}] Results for Richards equation without advection, i.e., $\hat{\bm{g}}=\bm{0}$. \textbf{(left)} Average iterations required per time-step for Richards equation without advection. The stopping criterion is based on the norm defined in \eqref{eq:Stop} \textbf{(right)} Mesh study of different iterative schemes for one time-step ($\tau =0.1$) and mesh size $h = 0.1/l$ in $1D$.}
            \label{fig:Richards _(No Adv)}
    \end{figure}

\section{Conclusion}\label{sec:conclusion}
In this study, we proposed a robust and efficient linearization scheme that can be applied to various nonlinear parabolic partial differential equations \eqref{eq:1}. Our approach effectively tackles the challenges associated with solving degenerate and nonlinear problems by splitting the nonlinearities as algebraic terms \eqref{eq:3} based on a reformulation of the problem \eqref{eq:2}. To ensure stability despite limited solution regularity, we adopt the Euler implicit method for time-discretization. 

In the splitted format, we investigate three different iterative schemes to linearize the problem: the Newton method, the L-scheme which uses a constant linearization coefficient for stability, and the M-scheme, which can be interpreted as the combination of both. While convergence of the Newton scheme cannot be guaranteed, especially for degenerate cases, it is proven that the L-scheme converges irrespective of the initial guess, even in double degenerate cases (\Cref{theo:LS}). The convergence is linear if the problem is single degenerate, although the contraction rate is predicted to become larger for finer time-steps. On the other hand, M-scheme, under additional boundedness assumptions and mild restrictions on the parameter values, also converges for double degenerate cases (\Cref{theo:MS}). The convergence is linear for single or non-degenerate cases, and for the non-degenerate case, the contraction rate scales with the time-step size.  Thus, convergence improves for smaller time-steps.

The performance of the M-scheme strongly depends on the value of the parameter $M$ chosen ($M=0$ corresponding to Newton, and $M$ large being the L-scheme).  Thus, to expedite convergence, we developed an adaptive $M$ selection approach based on a posteriori estimator. The a posteriori estimator provides an upper bound of linearization error for a given choice of $M$ (\Cref{theo:a-posteriori}). This is used to  ensure stability while selecting the smallest possible $M$-value, see
\Cref{alg:M-Adap}. Numerical results reveal that the M-adaptive scheme recovers the quadratic convergence property of Newton while being actually faster and more stable.

The numerical results demonstrate the performance of the double-splitting schemes and verify our predictions. We consider four numerical examples in \Cref{sec:PME,sec:toy,sec:Bio,sec:Richards}, including a well-known benchmark problem, a toy model, along with two additional examples (Biofilm and Richards equations) featuring realistic- parameters. These examples encompass both 1D and 2D cases.

The convergence of the Newton method is found to be highly dependent on discretization, which can make it slower than the M-scheme up to a certain error threshold. In fact, Newton diverges in several cases for fine meshes. L-scheme is limitingly slow in terms of iterations, although it is unconditionally converging. The M-scheme converges for all cases, and conforms with the predictions of \Cref{theo:MS}. It is also more stable in terms of mesh-size, and meets the stopping criteria in similar number of iterations as compared to Newton. However, the clear winner among the schemes is the M-Adapive algorithm, as it converges in all cases, takes the least amount of iterations, and achieves quadratic convergence asymtotically (see \Cref{fig:PME_ErrVsiter,fig:DD_ErrVsiter} e.g.).
 
\textbf{Acknowledgement} AJ acknowledge the financial support of the  HEC grant: 1(2)/HRD/OSS-III/BATCH-3/2022/HEC/384. The work of ISP was supported by the Research Foundation - Flanders (FWO), project G0A9A25N and the German Research Foundation (DFG) through the SFB 1313, project number 327154368.

\appendix
\section{Proof of \Cref{prop:well-posedness}}
\label{Appendix:proof}
\begin{proof}
For a given $s\in L^2(\Om)$, let $(S_s,U_s,W_s)\in \calZ$ solve for all  $(\psi,\phi,\f)\in\calZ$,
\begin{align}\label{eq:fixed-point-argument}
\begin{cases}
            \left(U_s-u_{n-1}, \varphi\right) + \t(\nabla W_s, \nabla \varphi) = \t(\bm{F}(b(s)),\nabla \varphi) + \t\langle f,\f\rangle,\\
        (U_s, \phi) = (b(S_s),\phi),\\    (W_s,\psi)= (B(S_s),\psi).
\end{cases}
\end{align}
The existence and uniqueness of $(S_s,U_s,W_s)\in \calZ$ is proven in Theorem A.1 of \cite{droniou2020high} using the gradient discretization method, and in Theorem 3.1 of \cite{smeets2024robust} using minimization of a convex functional. For $s_1,s_2\in L^2(\Om)$, subtracting the two versions of \eqref{eq:fixed-point-argument}, and using the test function $\varphi=W_{s_1}-W_{s_2}$ in the first equation yields
\begin{align*}
    &(b(S_{s_1})-b(S_{s_2}),B(S_{s_1})-B(S_{s_2})) + \t \|\del (W_{s_1}-W_{s_2})\|^2 \\&=\left(U_{s_1}-U_{s_2}, W_{s_1}-W_{s_2}\right) + \t(\nabla (W_{s_1}-W_{s_2}), \nabla (W_{s_1}-W_{s_2})) \\
    &=\t(\bm{F}(b({s_1}))-\bm{F}(b({s_2})),\nabla (W_{s_1}-W_{s_2}))\leq \frac{\t}{2}\|\bm{F}(b({s_1}))-\bm{F}(b({s_2}))\|^2 + \frac{\t}{2}\|\nabla (W_{s_1}-W_{s_2})\|^2\\
    &\overset{\ref{A.2}}{\leq } \frac{\t L_F}{2} (b({s_1})-b({s_2}),B({s_1})-B({s_2}))+ \frac{\t}{2}\|\nabla (W_{s_1}-W_{s_2})\|^2.
\end{align*}
In the last inequality, the monotonicity of $b,\,B$ functions and $\Phi=B\circ b^{-1}$ has been used.
Hence, if $\t L_F\leq 1$, then we have the contraction result 
\begin{align*}
    &(b(S_{s_1})-b(S_{s_2}),B(S_{s_1})-B(S_{s_2})) + \frac{\t}{2} \|\del (W_{s_1}-W_{s_2})\|^2\leq \frac{1}{2}(b({s_1})-b({s_2}),B({s_1})-B({s_2})).
\end{align*}
Repeating the iterative process $s\mapsto (S_s,U_s,W_s)$ by switching $s$ with $S_s$, one then obtains that $W_s\in H^1_0(\Om)$ must converge to a fixed point. 
\end{proof}


\end{document}